\documentclass{amsart}

\usepackage{amssymb} \usepackage{amsfonts} \usepackage{amsmath}
\usepackage{amsthm} \usepackage{epsfig} 
\usepackage{color}
\usepackage{pinlabel}
\usepackage{amscd}
\usepackage[all]{xy}
\usepackage{pinlabel}
\usepackage{tcolorbox}
\usepackage{seqsplit}
\usepackage{accents}

\usepackage{tikz}
\usepackage{xcolor}

\newtheorem{lemma}{Lemma}
\newtheorem{teo}[lemma]{Theorem}
\newtheorem{prop}[lemma]{Proposition}
\newtheorem{cor}[lemma]{Corollary}

\theoremstyle{definition}
\newtheorem{defn}[lemma]{Definition}

\newtheorem{ex}[lemma]{Exercise} 
\newtheorem{warning}[lemma]{Warning}

\theoremstyle{remark}
\newtheorem{rem}[lemma]{Remark}

\newcommand{\tvect} [2] {\big(\!\!{\tiny\begin{array}{c} #1 \\ #2 \\ \end{array}} \!\!\big)}

\newcommand{\interior}[1]{{\rm int}(#1)}

\newcommand{\Iso}{{\rm Isom}}

\newcommand{\matX}{\ensuremath {\mathbb{X}}}

\newcommand{\matR} {\ensuremath {\mathbb{R}}}

\newcommand{\matZ} {\ensuremath {\mathbb{Z}}}

\newcommand{\matH} {\ensuremath {\mathbb{H}}}
\newcommand{\matS} {\ensuremath {\mathbb{S}}}

\newcommand{\matRP} {\ensuremath {\mathbb{RP}}}

\newcommand{\Isom} {\ensuremath {{\rm Isom}}}

\author{Bruno Martelli}

\title{An introduction to Coxeter polyhedra}

\begin{document}

\begin{abstract}
This paper is an introduction to Coxeter polyhedra in spherical, Euclidean, and hyperbolic geometries. It consists of essentially two parts that could be read independently. In the first we introduce non-obtuse polyhedra in the spherical, Euclidean, and hyperbolic spaces, and prove various fundamental theorems originated from Andreev, Coxeter, and Vinberg. In the second we introduce Coxeter polyhedra and use them to describe regular, semiregular, and uniform polyhedra and tessellations, mostly via the Wythoff construction. 
\end{abstract}

\maketitle

\section*{Introduction}
Coxeter polyhedra are finite-volume polyhedra whose dihedral angles divide $\pi$. They exist in various forms in all the three geometries $\matR^n$, $\matH^n$, $\matS^n$ and in many dimensions $n$, and they lie at the heart of several geometric and algebraic constructions, being intimately connected with geometric symmetries, uniform polyhedra, manifolds of constant curvature, simple Lie algebras, lattices in Lie groups, etc.

Coxeter simplexes have been classified by Coxeter \cite{C}, Lann\'er \cite{L}, Koszul \cite{K} and Chein \cite{Ch}, who produced some very nice tables where these objects are presented via the extremely convenient notation of \emph{Coxeter diagrams}. It turns out that these tables are enough to understand all the Coxeter polyhedra in $\matR^n$ and $\matS^n$. The theory of Coxeter polyhedra in $\matH^n$ is however much richer: in dimension $n=3$ they have been classified by Andreev \cite{A, A2} and Roeder \cite{RHD}, and the classification is a very instructive instance of Perelman's Geometrization of 3-manifolds (and a fundamental ingredient in Thurston's original proof for Haken 3-manifolds).

There is yet no general theory of Coxeter polyhedra in $\matH^n$ in dimension $n\geq 4$, and understanding these entities is a current major subject of research, as it is employing them to construct more complex objects like higher-dimensional hyperbolic manifolds. Coxeter polyhedra have been generalized in various ways, mostly in algebraic and topological directions, notably starting with the well established notion of \emph{Coxeter group}. 

Despite their great importance, there do not seem to exist many available introductory texts for Coxeter polyhedra, and these notes have been written to try to fill this gap. The paper contains essentially two parts that could be read separately. In Sections \ref{poly:section} to \ref{poly23:section} we encounter the fundamental notion of \emph{non-obtuse polyhedron} and prove various fundamental theorems; the most important reference for these sections is Vinberg's excellent paper \cite{V}. In Sections \ref{Coxy:section} and \ref{regy:section} we finally meet Coxeter polyhedra and use them to describe a plethora of polyhedra and tessellations. We do not pursue further, ignoring plenty of additional beautiful examples and applications, to focus on these two introductory parts.

Most of the theorems are provided with complete proofs, a notable exception being the Theorem of Andreev and Roeder for which there already exists an excellent source \cite{RHD}. I have tried as much as possible to use a geometric language, shamelessly exploiting the enormous resources of Wikipedia Commons for the pictures of polyhedra and tessellations, and always preferring an image to a cumbersome notation to describe them in their full splendor. 

\subsection*{Acknowledgments}
Part of these notes were written during the year 2025 to prepare a minicourse on \emph{Hyperbolic manifolds constructed via Coxeter polyhedra} that I gave in Montreal and Ventotene, and a seminar at the Georgia International Topology Conference. I warmly thank the organizers of these conferences for providing excellent environments for research. I also warmly thank the referee for their comments and suggestions.

All the figures in Section \ref{regy:section} are taken from Wikipedia Commons and are either in the Public Domain or have a CC BY-SA 3.0 License. Those with a CC License are: the green polyhedra in Figures \ref{regular:fig}, \ref{archimedean:fig} and \ref{Catalan:fig}, made by Cyp; the hyperbolic 3-dimensional tessellations in Figure \ref{H:fig} and \ref{Hsr:fig}, made by Roice3; Figure \ref{Truncated_cuboctahedron:fig} made by Watchduck; Figure \ref{Honeycombs:fig} made by TED-43; Figure \ref{gyrati:fig} made by Tomruen.

\section{Polyhedra} \label{poly:section}
We fix some notation, briefly introduce the hyperbolic space, and then define polyhedra in all the three geometries $\matR^n, \matH^n, \matS^n$ trying to use a unifying language. 
We define the Gram matrix. Here polyhedra have finite volume by assumption. 

\subsection{Hyperbolic space}
We recall some standard facts in hyperbolic geometry, referring to \cite{M} for more details.
Let $n\geq 0$ be any natural number. We let $\matR^{n,1}$ denote the Minkowski space, that is the vector space $\matR^{n+1}$ equipped with the Lorentzian product
$$\langle x,y\rangle = -x_1y_1+ x_2y_2 + \cdots + x_{n+1}y_{n+1}.$$

We represent the hyperbolic space as usual with the hyperboloid model
$$\matH^n = \big\{x \in \matR^{n,1}, \langle x, x \rangle = -1, x_1 > 0\big\}.$$

The space $\matH^0$ is a point. When $n\geq 1$ the space $\matH^n$ is not compact, and
its compactification $\bar \matH^n$ of $\matH^n$ is obtained by projecting $\matH^n$ in $\matRP^n$ and taking its closure there. The \emph{sphere at infinity} $\partial \matH^n = \bar \matH^n \setminus \matH^n$ is the set of light rays in $\matR^{n,1}$. We denote a light ray as $[v]$ where $v$ is any future-directed vector in it. Every point at infinity $[v]$ determines a foliation of $\matH^n$ into \emph{horospheres} 
$$O_t = \{x\in \matH^n\ |\ \langle x, v \rangle = t \}$$
with $t<0$. Each horosphere
is isometric to $\matR^{n-1}$ and orthogonal to all the geodesics pointing towards $[v]$. The isometry with $\matR^{n-1}$ is obtained by projecting $O_t$ to the affine hyperplane $\{x_1=0, \langle x, v \rangle = t\} \subset \{x_1=0\} = \matR^n$ along rays parallel to $v$.

The compactification $\bar \matH^n \subset \matRP^n$ is a closed disc and using an affine chart it becomes the unit disc in $\matR^n$. This is the \emph{Klein model} for $\bar \matH^n$. 
The closure in $\bar \matH^n$ of a subset $S \subset \matH^n$ is denoted as $\bar S$.

\subsection{Subspaces}
Let $n\geq 0$ be any natural number. Throughout this paper we use the following notation:

\begin{center}
\emph{The symbol $\matX^n$ will always denote either $\matR^n, \matS^n$, or $\matH^n$.}
\end{center} 

The three spaces share many notable features, for instance they all have well-behaved subspaces of all dimensions $0\leq k< n$, a crucial fact to define polyhedra.
A \emph{$k$-dimensional subspace} $S$ in $\matR^n$ is an affine $k$-dimensional subspace. A \emph{$k$-dimensional subspace} $S$ in $\matS^n$ or $\matH^n$ is the intersection $S = W \cap \matS^n$ or $S=W \cap \matH^n$ with a $(k+1)$-dimensional vector subspace $W$ of $\matR^{n+1}$ or $\matR^{n,1}$. In the latter case we require the intersection to be non-empty, that is $W$ should have signature $(k,1)$. 

In any case, a $k$-dimensional subspace of $\matX^n$ is a totally geodesic copy of $\matX^k$. The intersection of two subspaces is either empty or a subspace. A subspace of codimension one is called a \emph{hyperplane} and it cuts $\matX^n$ into two connected components. 


\subsection{Half-spaces}
A \emph{dual unit vector} for $\matX^n$ is a vector $v$ in $\matR^n, \matR^{n+1}, \matR^{n,1}$ depending on $\matX^n = \matR^n, \matS^n, \matH^n$, such that $\langle v,v \rangle = 1$. 

A \emph{half-space} $H \subset \matX^n$ is a subset 
\begin{equation*}
H = \{ x \in \matX^n \ |\ \langle v, x \rangle \leq a \}
\end{equation*}
where $v$ is a dual unit vector, $a=0$ if $\matX^n = \matH^n, \matS^n$, and $a\in \matR$ if $\matX^n = \matR^n$. Both $v$ and $a$ are determined by $H$, and we say that $v$ is the \emph{unit normal vector} of $H$. If $n \geq 1$ then $\partial H$ is a hyperplane and $H$ is the closure of one of the two connected components of $\matX^n$ cut by $H$. If $n=0$ a dual unit vector $v$ may exist only if $\matX^0 = \matS^0$ and in this case $H$ is one of the two points in $\matS^0$. There are no half-spaces in $\matR^0, \matH^0$.


\subsection{Polyhedra}
Let $n\geq 0$ be any natural number.
A \emph{polyhedron} $P$ in $\matX^n$ is the intersection 
$$P = H_1 \cap \cdots \cap H_k$$
of finitely many half-spaces $H_i$ in $\matX^n$, such that the following conditions hold:
\begin{enumerate}
\item $P$ has finite non-zero volume, and 
\item $P$ is contained in the interior of a half-space when $\matX^n = \matS^n$. 
\end{enumerate}

Note that many authors like Vinberg \cite{V} do not assume that $P$ has finite volume. The condition (2) is equivalent to requiring that $P$ contains no antipodal points. Some authors also do not assume (2), and we do so here because it has many advantages (for instance, it makes sense to define the convex hull of points if they lie in the interior of a half-space) and few disadvantages, since one can show that a polyhedron that does not fulfill (2) is a multiple spherical suspension of one that does (for instance, the intersection of two half-planes in $S^2$ with different boundaries is a \emph{bigon}, that is not allowed here as a polyhedron; a bigon with some interior angles $\alpha < \pi$ is the spherical suspension of a segment of length $\alpha$, that is an allowed 1-dimensional polyhedron).

By definition we have 
$$P = \{x \in \matX^n \ |\ \langle x, v_i \rangle \leq a_i \}$$
where $a_i=0$ if $\matX^n \neq \matR^n$. Here $v_i$ is the normal unit vector of $H_i$.

The intersection of $P$ with the boundary of a half-space containing $P$ is called a \emph{face} of $P$, whose \emph{dimension} is the dimension of its \emph{supporting subspace}, the smallest subspace in $\matX^n$ containing it. A face of dimension $k=0,1,n-2,n-1$ is called a \emph{vertex}, \emph{edge}, \emph{ridge}, \emph{facet} respectively.

A polyhedron $P\subset \matX^n$ is called \emph{hyperbolic}, \emph{Euclidean}, and \emph{spherical} depending on whether $\matX^n = \matH^n, \matR^n$, or $\matS^n$.
We will typically consider polyhedra only up to isometries in $\matH^n$, $\matS^n$ and up to similarities in $\matR^n$. The term \emph{polytope} is sometimes employed for polyhedra of dimension higher than 3, but we will not use it here.

A polyhedron in $\matX^1$ is a segment of some finite length $\ell>0$, and we have $\ell < \pi$ if $\matX^1=\matS^1$. A polyhedron in $\matX^0$ is a point: if $\matX^0 = \matS^0$ a point is a half-space, while if $\matX^0 = \matH^0, \matR^0$ the point is the whole space, obtained as the intersection of an empty set of half-spaces.

\subsection{Compact and non-compact polyhedra}
If $P$ lies in $\matS^n$ or $\matR^n$, it is compact. If it lies in $\matH^n$, it may not be, but its closure $\bar P$ in $\bar \matH^n$ of course is, and it intersects $\partial \matH^n$ into finitely many points called \emph{ideal vertices}. 
To avoid potential confusion, the vertices of $P$ contained in $\matH^n$ are sometimes called \emph{real}. The polyhedron $P$ is itself \emph{ideal} if it has no real vertex. 

In the Klein model the subspaces of $\matH^n$ are the affine subspaces of $\matR^n$ intersected with the unit ball. With this model the closure $\bar P \subset \bar \matH^n$ of a polyhedron $P\subset \matH^n$ is just a Euclidean polyhedron contained in the closed unit disc. The faces of $\bar P$ are the faces of $P$, plus its ideal vertices. We stress the fact that the ideal vertices of $P$ do not count as faces for $P$, unless otherwise specified.

\subsection{Exercises}

\begin{ex} \label{unique:ex}
Let $P$ be a polyhedron in $\matX^n$. We have 
$$P = H_1 \cap \cdots \cap H_k$$
for a unique minimal set of half-spaces $H_1,\ldots, H_k$. When $n \geq 1$, the polyhedron has $k$ facets $F_1,\ldots, F_k$ and $\partial H_i$ contains $F_i$. When $n=0$ the polyhedron $P$ is a point and we have $k=1$, $P = H_1$ for $\matX^0 = \matS^0$, and $k=0$, $P=\matX^0$ for $\matX^0 = \matH^0, \matR^0$.
\end{ex}


\begin{ex} \label{CH:ex}
The convex hull of some points (that lie in the interior of a half-space in the spherical case) is well-defined in $\matX^n$. The convex hull of finitely many points in $\matX^n$ that are not contained in a hyperplane is a polyhedron. This holds also in $\bar \matH^n$ with few adjustments (we require $n\geq 2$, otherwise we might get segments of infinite length, and the points should not be contained in the closure of a hyperplane). 

Every polyhedron in $\matS^n, \matR^n$ is the convex hull of its vertices.  Every polyhedron in $\matH^n$ is the convex hull in $\bar\matH^n$ of its real and ideal vertices, with its ideal vertices removed.
\end{ex}

From this exercise we deduce that every face of a polyhedron $P$ is itself a polyhedron in its supporting subspace, except when $P\subset \matH^n$ and the face is an edge with at least one ideal endpoint (because it has infinite length).

\begin{ex}
The faces of a polyhedron $P \subset \matS^n, \matR^n$, after adding $\emptyset$ and $P$, form a \emph{lattice}: they form a poset by inclusion, and every two faces have a least upper bound and a greatest lower bound. If $P\subset \matH^n$, this is true for the compactification $\bar P$ (that is, it is true if we consider ideal vertices as faces).
\end{ex}

Two polyhedra, possibly of different geometries, are \emph{combinatorially equivalent} if they have isomorphic face lattices. They are \emph{combinatorially dual} if their face lattices are isomorphic after reversing the inclusions of one of them.

\begin{ex}
Every combinatorial equivalence between two compact polyhedra can be realized via a canonical homeomorphism, by taking barycentric subdivisions and coordinates (that are well-defined in any geometry $\matX^n$!).
\end{ex}

\begin{ex} \label{simplex:ex}
A polyhedron in $\matX^n$ has at least $n+1$ facets, and it has $n+1$ if and only if it is combinatorially a simplex (possibly with some ideal vertices if $\matX^n = \matH^n$). 
\end{ex}

\begin{ex} \label{join:ex}
The product $P\times Q$ of two Euclidean polyhedra $P\subset \matR^m$ and $Q \subset \matR^n$ is a Euclidean polyhedron in $\matR^{m+n}$. The \emph{join} 
$$P*Q = \big\{(x\cos \theta , y\sin \theta ) \in \matR^{m+1} \times \matR^{n+1}\ |\ x \in P, y \in Q, \theta \in [0,\pi/2]\big\}$$
of two spherical polyhedra $P\subset \matS^m$ and $Q \subset \matS^n$ is a spherical polyhedron in $\matS^{m+n+1}$.
\end{ex}

The product of two Euclidean segments is a rectangle, and the join of two spherical segments is a spherical tetrahedron. The join of two points in $\matS^0$ is a segment in $\matS^1$ of length $\ell = \pi/2$, and the join of a point in $\matS^0$ and a segment in $\matS^1$ of length $\ell< \pi/2$ is a spherical triangle with interior angles $\ell, \pi/2, \pi/2$. More generally, the join of two spherical simplexes is a spherical simplex.

\begin{ex} \label{ortho:ex}
Let $S_1, \ldots, S_k \subset \matH^n$ be $k\leq n$ hyperplanes. One of the following assertions holds: 
\begin{enumerate}
\item $S_1\cap \cdots \cap S_k \neq \emptyset$;
\item The hyperplanes are all orthogonal to some horosphere;
\item The hyperplanes are all orthogonal to some $(k-1)$-space $Z \subset \matH^n$.
\end{enumerate}
\end{ex}
\begin{proof}[Hint]
Consider $\matH^n$ inside $\matRP^n$. The hyperplanes $S_i$ extend to hyperplanes $\hat S_i$ in $\matRP^n$, that must intersect at some $P \in \matRP^n$. We get (1), (2), (3) depending on whether $P$ lies in $\matH^n$, $\partial \matH^n$, or outside of $\bar \matH^n$. 
\end{proof}

\subsection{Links and dihedral angles}
Let $P$ be a polyhedron in $\matX^n$. The \emph{link} of a $k$-dimensional face $F$ of $P$ is a polyhedron in $\matS^{n-k-1}$ obtained  by intersecting $P$ with a small rescaled $\matS^{n-k-1}$ contained in a $(n-k)$-subspace intersecting $F$ orthogonally in some point $x\in \interior F$ and centered at $x$, see Figure \ref{links:fig}.

\begin{figure}
 \begin{center}
\centering
\labellist
\small\hair 2pt
\pinlabel $v$ at 47 70
\pinlabel $e$ at 218 120
\pinlabel $x$ at 218 80
\endlabellist
  \includegraphics[width = 7 cm]{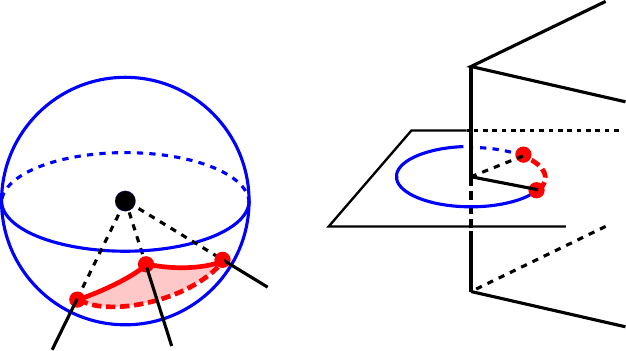}
 \end{center}
 \caption{The link of a (real) vertex $v$ and of an edge $e$ of a three-dimensional polyhedron is a spherical triangle and a spherical arc respectively (both drawn in red).}  \label{links:fig}
\end{figure}

The link of a ridge $F$ is a segment in $\matS^1$ of some length $\alpha \in (0, \pi)$, that we record as the \emph{dihedral angle} of $P$ at $F$, see Figure \ref{links:fig}-(right). We will see that dihedral angles are fundamental for our understanding of polyhedra.

If $P\subset \matH^n$, the \emph{link} of an ideal vertex $v$ of $P$ is the polyhedron in $\matR^{n-1}$ obtained by intersecting $P$ with a small horosphere centered at $v$. The dihedral angle of an ideal vertex of a polygon is zero by convention. 

We remark that the link of a face is a spherical polyhedron, while the link of an ideal vertex is a Euclidean polyhedron.


\subsection{Gram matrix} \label{Gram:subsection}
Let $P$ be a polyhedron in $\matX^n$. By Exercise \ref{unique:ex} we have 
$$P = H_1 \cap \cdots \cap H_k$$
for a unique minimal set of half-spaces $H_1,\ldots, H_k$, whose boundaries contain the $k$ facets $F_1,\ldots, F_k$ of $P$ if $n\geq 1$. Let $v_i$ be the unit normal vector of $H_i$. 
The \emph{Gram matrix} of $P$ is the $k\times k$ symmetric matrix $G$ with entries
\begin{equation*} 
G_{ij} = \langle v_i, v_j \rangle.
\end{equation*}

The Gram matrix is clearly invariant under isometries of $\matH^n, \matS^n$ and similarities of $\matR^n$. The low-dimensional cases are easily understood. If $P$ is a point in $\matS^0$, we have $k=1$ and hence $G= (1)$. If $P$ is a point in $\matR^0, \matH^0$ we have $k=0$ and hence $G= \emptyset$. If $P$ is a segment
in $\matX^1$ of some length $\ell$ then $G$ equals
\begin{equation} \label{G2:eq}
\begin{pmatrix} 1 & -\cos \ell \\ -\cos \ell & 1 \end{pmatrix}, \qquad
\begin{pmatrix} 1 & -1 \\ -1 & 1 \end{pmatrix}, \qquad
\begin{pmatrix} 1 & -\cosh \ell \\ -\cosh \ell & 1 \end{pmatrix}
\end{equation}
depending on $\matX^1 = \matS^1, \matR^1, \matH^1$. If $\matX^1 = \matS^1$ then $\ell < \pi$.
If $n\ge 2$ we have
$$
G_{ij} = \begin{cases} 
1 & {\rm\ if\ } i = j, \\
-\cos \alpha & {\rm if\ } \partial H_i {\rm\ and\ } \partial H_j {\rm\ are\ incident\ with\ angle}\ \alpha, \\
-1 & {\rm if\ } \partial H_i {\rm\ and\ } \partial H_j {\rm\ are\ parallel}, \\
-\cosh d & {\rm if\ } \partial H_i {\rm\ and\ } \partial H_j {\rm\ are\ ultraparallel\ with\ distance\ } d.
\end{cases}
$$

\begin{figure}
 \begin{center}
\centering
\labellist
\small\hair 2pt
\endlabellist
  \includegraphics[width = 9 cm]{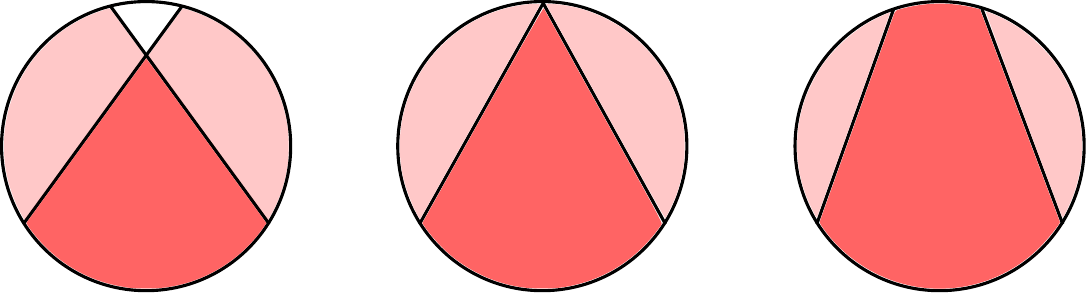}
 \end{center}
 \caption{The hyperplanes $\partial H_i$ and $\partial H_j$ can be incident, parallel, or ultraparallel in $\matH^n$.}  \label{FiFj:fig}
\end{figure}

See Figure \ref{FiFj:fig}. The angle $\alpha$ is the interior one with respect to $P$, and coincides with the dihedral angle of the face $F_i\cap F_j$ when this intersection is non-empty. Two disjoint hyperplanes in $\matH^n$ are \emph{parallel} or \emph{ultraparallel} depending on whether their closures in $\bar \matH^n$ intersect or not. Two hyperplanes can be parallel only in $\matR^n$ or $\matH^n$, and ultraparallel only in $\matH^n$.

We deduce in particular that $G_{ij} \leq 1$ and $G_{ij} = 1$ if and only if $i=j$. In the geometries $\matS^n$ and $\matR^n$ we also have $G_{ij} > -1$ and $G_{ij} \geq -1$ respectively.
We say that $G$ is \emph{decomposable} if 
$$G = \begin{pmatrix} G_1 & 0 \\ 0 & G_2 \end{pmatrix}$$
after possibly acting simultaneously on rows and columns via some permutation $\sigma$. As every symmetric matrix, $G$ decomposes uniquely (up to permutations) into some indecomposable principal submatrices.

\begin{ex} \label{G:ex}
The Gram matrix $G$ of a polyhedron $P$ is decomposable if and only if either $\matX^n = \matR^n$ and $P=P_1\times P_2$ with $\dim P_i > 0$, or $\matX^n = \matS^n$ and $P=P_1*P_2$, and $G_i$ is the Gram matrix of $P_i$.
\end{ex}

The following exercise is of fundamental importance.

\begin{ex}
Let $P \subset \matX^n$ be a polyhedron with $n\geq 1$. The unit normal vectors $v_1,\ldots, v_k$ generate the space $\matR^n, \matR^{n+1}$, or $\matR^{n,1}$. Therefore the signature of $G$ is 
$$(n,0,k-n), \quad (n+1,0,k-n-1), \quad {\rm or} \quad (n,1,k-n-1)$$
depending on the geometry $\matX^n = \matR^n, \matS^n$, or $\matH^n$. 
\end{ex}
\begin{proof}[Hint]
To prove that $v_1,\ldots, v_k$ generate,
use the fact that $P$ is contained in the interior of a half-space if $\matX^n = \matS^n$ and that it has finite volume if $\matX^n = \matR^n, \matH^n$.
\end{proof}

\section{Non-obtuse polyhedra} \label{non-obty:section}
We introduce a class of particularly well-behaved polyhedra called \emph{non-obtuse}, which contains all the yet-to-be-defined Coxeter polyhedra. This class can be defined in two natural ways, that are luckily equivalent by a theorem of Andreev \cite{A3}, whose proof is seldom reported despite its fundamental importance in the theory of Coxeter polyhedra. The theory of non-obtuse polyhedra was masterfully described by Vinberg \cite{V}, a source that we strongly suggest for further reading. Most of the material here is taken from there.

\subsection{Definition}
Let $P\subset \matX^n$ be a polyhedron, with Gram matrix $G$. There are two natural ways to define when $P$ is \emph{non-obtuse}:
\begin{enumerate}
\item If $G_{ij} \leq 0$ for all $i\neq j$, or
\item If the dihedral angles of all the ridges are $\leq \pi/2$.
\end{enumerate}
We adopt (1) as a definition, because it is stronger and robust, and it makes sense also in low dimensions. Later on, we will prove Andreev's Theorem \ref{nonobtuse:teo} that reassuringly asserts that (1) $\Longleftrightarrow$ (2) when $n\geq 2$.

We easily classify the low-dimensional non-obtuse polyhedra.
A polyhedron $P$ in $\matX^0$ is a point, we have $G=(1)$ or $\emptyset$ depending on the geometry, and hence the point $P$ is non-obtuse in any case. A polyhedron $P \subset \matX^1$ is a segment of some length $\ell>0$ (with $\ell < \pi$ if $\matX^1=\matS^1$) and its Gram $2\times 2$ matrix $G$ is as in \eqref{G2:eq} from Section \ref{Gram:subsection}. If $\matX^1 = \matH^1, \matR^1$, the segment is always non-obtuse, while if $\matX^1 = \matS^1$ the segment is non-obtuse $\Longleftrightarrow$ $\ell \leq \pi/2$.

By Exercise \ref{G:ex} the class of non-obtuse polyhedra is closed under products and joins. 
The following theorem shows that the non-obtuse condition is very restrictive in the spherical and Euclidean geometries:

\begin{teo} \label{non-obtuse:teo}
Every non-obtuse polyhedron in $\matS^n$ is a simplex. Every non-obtuse polyhedron in $\matR^n$ is a product of simplexes.
\end{teo}
\begin{proof}
Let $G$ be the Gram matrix of a non-obtuse polyhedron $P$ in $\matS^n$ or $\matR^n$. By Exercises \ref{join:ex} and \ref{G:ex} we may suppose that $G$ is indecomposable. Therefore $G=I-B$ for some indecomposable $B\geq 0$. By the Perron -- Frobenius Theorem the matrix $B$ has a largest simple positive eigenvalue $\lambda>0$ with positive eigenvector $v>0$. Therefore $G$ has a lowest simple eigenvalue $1-\lambda$ with the same eigenvector $v>0$. 

The signature of $G$ is either $(n,0,k-n)$ or $(n+1,0,k-n-1)$ depending on whether we work in $\matR^n$ or $\matS^n$. Since the lowest eigenvalue of $G$ is simple, the signature is one of these: $(n,0,0), (n,0,1)$, $(n+1,0,0)$, or $(n+1,0,1)$. By Exercise \ref{simplex:ex} the first case is excluded, and the second and third cases yield a simplex. In the fourth case we would have $\lambda = 1$ and $Gv=0$, which gives a dependence relation for the columns of $G$ with positive coefficients. Since they generate $\matR^{n+1}$, the same relation holds for the normal vectors of the facets of $P$, a contradiction, since the scalar product of each such vector with any fixed interior point of $P$ is negative.
\end{proof}

\begin{figure}
 \begin{center}
  \includegraphics[width = 10 cm]{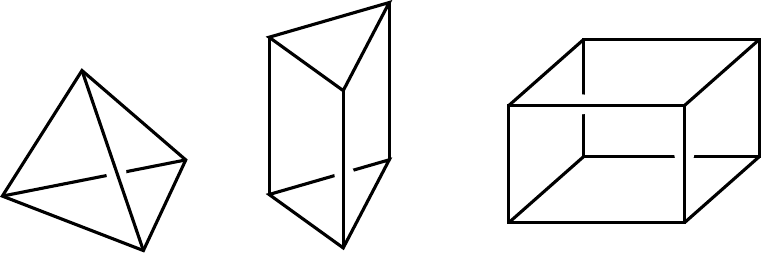}
 \end{center}
 \caption{The only nonobtuse polyhedra in $\matR^3$.} \label{nonobtuse3:fig}
\end{figure}

The only non-obtuse polygons in $\matR^2$ are the non-obtuse triangles and the rectangles.
The only non-obtuse polyhedra in $\matR^3$ are those shown in Figure \ref{nonobtuse3:fig}, that are a tetrahedron with non-obtuse dihedral angles, a rectangular prism with non-obtuse base triangles, and a rectangular parallelepiped.
We will encounter many more types of non-obtuse polyhedra in the hyperbolic space $\matH^n$.

\begin{prop} \label{face:non-obtuse:prop}
Every face $f$ of a non-obtuse polyhedron $P\subset \matX^n$ is either a non-obtuse polyhedron or an edge of infinite length (only if $\matX^n = \matH^n$). The dihedral angles of $f$ are smaller or equal than the corresponding ones of $P$.
\end{prop}
\begin{proof}
Let $F_1$ be a facet of $P$, adjacent to some facets $F_2, \ldots, F_h$. If $F_1$ is not an edge of infinite length, it is a polyhedron. Let $v_1,\ldots, v_h$ be their unit normal vectors. The facet $F_1$ is contained in a hyperplane $S$ and has facets $F_1 \cap F_i$ with $i=2,\ldots, h$, with (unnormalized) normal vectors $v_i' = v_i - \langle v_i,v_1 \rangle v_1$. We get
\begin{align*}
\langle v_i', v_j' \rangle & = \langle v_i - \langle v_i,v_1 \rangle v_1, v_j - \langle v_j,v_1 \rangle v_1\rangle \\
& = \langle v_i, v_j \rangle - \langle v_i, v_1 \rangle \langle v_j, v_1 \rangle \leq \langle v_i, v_j \rangle \leq 0
\end{align*}
for every $i \neq j$. Hence $F_1$ is non-obtuse, and its dihedral angles are $\leq$ the corresponding ones of $P$. By iterating we deduce this for every face of $P$.
\end{proof}

As anticipated, Andreev proved the following natural and useful criterion \cite{A3}.

\begin{teo} \label{nonobtuse:teo}
If a polyhedron $P\subset \matX^n$ with $n\geq 2$ has all dihedral angles $\leq \pi/2$, two facets intersect $\Longleftrightarrow$ their supporting hyperplane do. Therefore $P$ is non-obtuse.  
\end{teo}
\begin{proof}
We start by noting that the proof of Proposition \ref{face:non-obtuse:prop} applies to this context and shows that the dihedral angles of the facets of $P$ are also $\leq \pi/2$. 

We first consider the spherical case. We prove by induction on $n = \dim P$ that $P$ is in fact a simplex. If $n=2$ we get a triangle since the angles of a $k$-gon in $\matS^2$ sum to $>\pi(k-2)$. For general $n\geq 3$, by the induction hypothesis every link and every facet of $P$ is a simplex, and this easily implies that $P$ is a simplex.

We turn to the geometries $\matX^n = \matR^n, \matH^n$. It is convenient to exceptionally allow polyhedra to have infinite volume, and to prove the assertion in this more general setting, by induction on the dimension $n$ and the number $k$ of facets of $P$. By what already proved in the spherical setting, the links of all the points are simplexes, so in particular two facets may intersect only in a ridge.

The polyhedron $P$ has some facets $F_1,\ldots, F_k$ and is the intersection of half-spaces $H_1,\ldots, H_k$. Suppose that there are two disjoint facets $F_i, F_j$ adjacent to the same $F_h$, such that $\partial H_i \cap \partial H_j \cap (\matH^n \setminus H_h) \neq \emptyset$ as in Figure \ref{nonobtuse:fig}. This is the key configuration that we want to rule out.

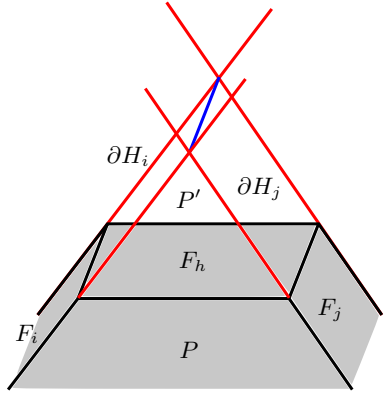
\begin{figure}
 \begin{center}
\centering
\labellist
\small\hair 2pt
\pinlabel $P$ at 90 20
\pinlabel $F_h$ at 92 64
\pinlabel $F_i$ at 10 28
\pinlabel $F_j$ at 160 40
\pinlabel $P'$ at 90 95
\pinlabel $\partial H_i$ at 60 116
\pinlabel $\partial H_j$ at 125 100
\endlabellist
  \includegraphics[width = 5 cm]{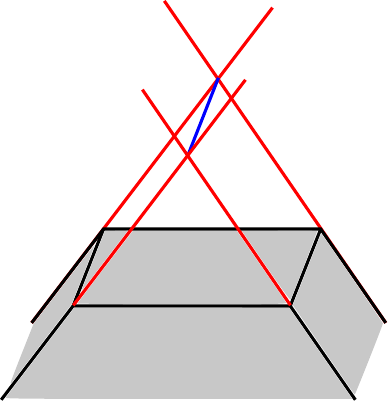}
 \end{center}
 \caption{A facet $F_h$ and two incident facets $F_i,F_j$ of $P$ such that $\partial H_i \cap \partial H_j \cap (\matH^n \setminus H_k) \neq \emptyset$. This is the key configuration that we want to rule out while proving Theorem \ref{nonobtuse:teo}. The dihedral angles of $P$ are non-obtuse, and this leads to a contradiction.} \label{nonobtuse:fig}
\end{figure}

The figure also shows the polyhedron $P' = H_i\cap H_j \cap (\matH^n \setminus \interior {H_h})$, that has three facets and three ridges. If $n=2$ then $P'$ is a triangle whose inner angles sum to $> \pi$, a contradiction. If $n\geq 3$, by the induction hypothesis on $F_h$ the supporting subspaces $\partial H_h \cap \partial H_i, \partial H_h \cap \partial H_j$ of the ridges $F_h\cap F_i, F_h\cap F_j$ do not intersect (since the ridges do not). By Exercise \ref{ortho:ex} there is either a horosphere or a plane that is orthogonal to $\partial H_h, \partial H_i, \partial H_j$, and in both cases the three hyperplanes bound a triangle there with inner angles $>\pi$, a contradiction. 

For every $h\leq k$ we consider the (possibly infinite volume) polyhedron with $k-1$ facets $P_h = \cap_{i\neq h} H_i$. By what just proved, every ridge of $P_h$ is contained in a ridge of $P$. Therefore the dihedral angles of $P_h$ are all $\leq \pi/2$, and by our induction hypothesis if two facets of $P_h$ are disjoint then their supporting hyperplanes are. Therefore the same holds for $P$ for every pair of facets $F_i,F_j$ with $i,j \neq h$. Since $h$ is arbitrary, this holds for every pair $i,j$. The proof is complete.
\end{proof}

We note that the closures in $\bar \matH^n$ of two disjoint faces of a polyhedron $P\subset \matH^n$ could intersect at infinity at an ideal point. This happens for instance for many pairs of faces in a regular ideal octahedron. 
Two such faces are nevertheless considered to be disjoint here, following our definitions.

\subsection{Principal submatrices}
Let $P\subset \matX^n$ be a polyhedron, with Gram matrix $G$. We can infer much of the geometry of $P$ by looking at the principal matrices of $G$.

Let $F_1,\ldots, F_k$ be the facets of $P$. A face $f$ of $P$ determines a principal submatrix $G_f$ of $G$ consisting of the $G_{ij}$ such that the facets $F_i$ and $F_j$ contain $f$. 
Similarly, an ideal vertex $v$ of $P$ determines a principal submatrix $G_v$ of $G$ consisting of the $G_{ij}$ such that the closures in $\bar\matH^n$ of the facets $F_i$ and $F_j$ contain $v$.

\begin{prop} \label{Gf:prop}
The matrices $G_f$, $G_v$ are the Gram matrices of the links of $f$, $v$. 
\end{prop}
\begin{proof}
We first consider $f$. The unit normal vectors $v_j$ such that $F_j$ contains $f$ span a positive definite vector space $W$ whose orthogonal $W^\perp$ contains $f$ (or a parallel copy of it if $\matX^n = \matR^n$), and the link of $f$ can be realized in the unit sphere of $W$ as a polyhedron with the same normal vectors $v_j$.

We turn to $v$. Here $\matX^n = \matH^n$. The unit normal vectors $v_j$ such that $\bar F_j$ contains $v$ span a positive semi-definite vector space $W$ whose orthogonal $W^\perp$ is the light ray $v$. A parallel affine copy $W'$ of $W$ intersects $\matH^n$ in a horosphere, and the link of $V$ can be realized in $W' \cap \{x_1=0\}$ as a polyhedron with normal vectors $v_j$.
\end{proof}

We now apply Theorem \ref{non-obtuse:teo} and deduce the following.

\begin{cor}
If $P\subset \matH^n$ is a non-obtuse polyhedron, the links of all its faces and ideal vertices are also non-obtuse. In particular these are spherical simplexes and products of Euclidean simplexes, respectively.
\end{cor}

\begin{cor}
Every non-obtuse polyhedron $P\subset \matH^n$ is \emph{simple}, that is every $h$-dimensional face is contained in exactly $n-h$ facets.
\end{cor}

We should warn the reader that if a non-obtuse polyhedron $P\subset \matH^n$ is not compact then its compactification $\bar P$ is not necessarily simple, because an ideal vertex may be contained in the closures of more than $n$ facets: the ideal regular octahedron is an example.
The following definition is crucial.

\begin{defn}
Let $P\subset \matX^n$ be a polyhedron with Gram matrix $G$. A principal submatrix of $G$ is 
\begin{itemize}
\item \emph{spherical} if it is positive definite;
\item \emph{Euclidean} if it has rank $n-1$ and decomposes into undecomposable matrices, each of signature $(k,0,1)$ for some $k>0$.
\end{itemize}
\end{defn}

\begin{lemma} \label{bijection:lemma}
Let $P \subset \matH^n$ be a non-obtuse polyhedron. The assignments $f \to G_f$ and $v \to G_v$ yield a bijective correspondence between the faces and ideal vertices of $P$ and the spherical and Euclidean principal submatrices of $G$.
\end{lemma}
\begin{proof}
The submatrices $G_f$ and $G_v$ are Gram matrices of a non-obtuse spherical and flat polyhedron, that is a simplex and a product of simplexes by Theorem \ref{non-obtuse:teo}. Therefore they are spherical and Euclidean. 

Conversely, suppose that by selecting a set $J \subset \{1,\ldots, k\}$ of rows and columns we get a spherical principal submatrix $H$. The vectors $v_j$ with $j \in J$ are thus independent and span a positive definite subspace $W$, and $S = W^\perp \cap \matH^n$ is a subspace. We now prove that $f=S\cap P$ is a face of $P$ with $G_f = H$.

The orthogonal projection $\pi\colon \matR^{n,1} \to W^\perp$ is
$$\pi(x) = x-\sum_{j,l\in J} (H^{-1})_{jl} \langle x, v_{j} \rangle v_{l}.$$
To show this, note that
$$\langle \pi(x), v_{i} \rangle = 
\langle x, v_{i} \rangle - \sum_{j,l\in J} (H^{-1})_{jl} \langle x, v_{j} \rangle H_{li} =
\langle x, v_{i} \rangle - \langle x, v_{i} \rangle = 0
$$
for every $i\in J$ and therefore $\pi(x) \in W^\perp$.
The projection $\pi$ induces an orthogonal projection $\pi\colon \matH^n \to S$ that is of the same form up to renormalizing. A point $x \in \matH^n$ lies in $P$ if and only if 
$$\langle x, v_i \rangle \leq 0$$
for all $i=1,\ldots, k$. If this holds, then 
$$\langle \pi(x), v_i \rangle =
\langle x-\sum_{j,l\in J} (H^{-1})_{jl} \langle x, v_{j} \rangle v_{l}, v_i \rangle
= \langle x, v_i \rangle -\sum_{j,l \in J} (H^{-1})_{jl} \langle x, v_{j} \rangle G_{li}.
$$

We note that $H = I - B$ for some $B\geq 0$ with largest eigenvalue $<1$ because $H$ is positive definite. Therefore $\|B\|<1$ and $H^{-1} = I + B + B^2 + \cdots$, therefore $H^{-1} \geq 0$. We deduce that $\langle \pi(x), v_i \rangle \leq 0$ for every $i\not \in J$, we already know that $\langle \pi(x), v_i \rangle = 0$ for all $i\in J$, and therefore $\pi(x) \in P$. This implies that $\pi(P) \subset P$. In particular $f = S \cap P = \pi(P)$ is not empty and is hence a face of $P$ with $G_f = H$.

Finally, let a set $J \subset \{1,\ldots, k\}$ of rows and columns provide a Euclidean principal submatrix $H$. We have $H = I-B$ with $B\geq 0$ that decomposes into indecomposable matrices, each with largest eigenvalue $\lambda = 1$. By the Perron -- Frobenius Theorem there is a $w>0$ with $Bw=w$ and hence $Hw=0$. The vector
$$ v = \sum_{j \in J} w_j v_j$$
is orthogonal to each $v_j$ with $j \in J$. Therefore $v \in W \cap W^\perp$ is isotropic, where $W$ is the space generated by $v_j$ with $j \in J$. We have $v\neq 0$, since for any interior point $x \in P$ we have $\langle x, v_j \rangle < 0$ and hence $\langle x,v \rangle < 0$. Therefore $[v] \in \partial \matH^n$.

We have $\langle v, v_i \rangle \leq 0$ for all $i\not\in J$. Therefore $[v]$ is an ideal vertex of $P$. It remains to prove that the facets incident to $[v]$ are precisely the $F_j$ with $j\in J$. We know that each $F_j$ is incident to $[v]$, hence $H$ is a principal submatrix of $G_{[v]}$, and we want to prove that $H=G_{[v]}$. 
From Theorem \ref{non-obtuse:teo} we know that $G_{[v]}$ decomposes into some $h$ matrices $G_{[v]}^i$ of signature $(n_i,0,1)$, with $n_1+\cdots+n_h = n-1$. The matrix $H$ is Euclidean and hence it also decomposes into some $h'$ matrices $H^i$ of signature $(n_i',0,1)$ with $n_1' + \cdots + n_{h'}' = n-1$. Each $H^i$ is a submatrix of some $G_{[v]}^i$. We deduce that $H=G_{[v]}^i$.
\end{proof}

This lemma is very important because it tells us that the whole combinatorial structure of $P$ may be deduced from its Gram matrix $G$.
In the proof we have also shown this interesting geometric fact.

\begin{prop}
Let $f$ be a face of a non-obtuse polyhedron $P \subset \matH^n$. The orthogonal projection $\pi\colon \matH^n \to S$ onto the subspace $S$ containing $f$ sends $P$ to $f$.
\end{prop}

\subsection{Vinberg's Realization Theorem}
A theorem of Vinberg \cite{V} characterizes completely the Gram matrices of non-obtuse hyperbolic polyhedra. 


\begin{teo} \label{realization:teo}
A symmetric $k\times k$ matrix $G$ is the Gram matrix of a non-obtuse polyhedron $P\subset \matH^n$ with $k$ facets if and only if $G_{ii} =1, G_{ij} \leq 0$ for all $i\neq j$, $G$ has signature $(n,1,k-n-1)$, and moreover:
\begin{enumerate}
\item $G$ contains at least one spherical submatrix of rank $n$;
\item Each spherical submatrix of rank $n-1$ is contained in 2 distinct submatrices of $G$, each of which is either spherical of rank $n$ or Euclidean.
\end{enumerate}
The polyhedron $P$ is uniquely determined by $G$ up to isometries of $\matH^n$.
\end{teo}
\begin{proof}
Let $G$ be a $k\times k$ symmetric matrix with $G_{ii}=1$, $G_{ij} \leq 0$ for all $i\neq j$, and signature $(n,1,k-n-1)$. By linear algebra we can find some generators $v_1,\ldots, v_k \in \matR^{n,1}$ such that $G_{ij} = \langle v_i, v_j \rangle$ for all $i,j$. 
We have $G = I - B$ with $B\geq 0$. By the Perron -- Frobenius Theorem $G$ has a lowest eigenvalue $\lambda < 0$ with eigenvector $w\geq 0$. Set
$$v = \sum_{j=1}^k w_jv_j.$$
We have
$$\langle v, v_i \rangle = \sum_{j=1}^k w_j\langle v_j, v_i \rangle = (Gw)_i  = \lambda w_i \leq 0
$$
for all $i$, with a strict inequality for some $i$, hence $\langle v,v\rangle < 0$. Up to reversing all the vectors $v_j$ we may suppose that by rescaling $v$ we get a point in $\matH^n$. We define
$$P = \{x \in \matH^n\ |\ \langle x, v_i \rangle \leq 0 \}$$
and note that it has non-empty interior since it contains the rescaled $v$. The hyperplane $S_i = \{\langle x, v_i \rangle = 0\}$ intersects $P$ in a facet $F_i$: to show this, note that the orthogonal projection $\pi \colon \matH^n \to S_i$ is 
$$\pi(x) = x- \langle x,v_i \rangle v_i$$
and $x \in P$ easily implies $\pi(x) \in P$. Therefore $G$ is the Gram matrix of $P$. Since $v_1,\ldots, v_k$ are unique up to isometry, the matrix $G$ determines $P$.

It remains to prove that $P$ has finite volume if and only if the conditions (1) and (2) hold. We project $\matH^n$ inside $\matRP^n$ and define 
$$\hat P = \{[x] \in \matRP^n\ |\ \langle x, v_i \rangle \leq 0\}.$$

Since the vectors $v_i$ generate $\matR^{n,1}$, the subset $\hat P\subset \matRP^n$ is a polyhedron in some affine chart $\hat P \subset \matR^n \subset \matRP^n$. We have $P = \hat P \cap \matH^n$ and $P$ has finite volume precisely when $\hat P \setminus P$ consists of finitely many (possibly none) points (the ideal vertices of $P$). This holds if and only if the 1-skeleton of $\hat P$ is entirely contained in $\bar \matH^n$, and this is in turn equivalent to the following requirements: (1) $\hat P$ has at least one vertex in $\bar \matH^n$, with a neighbourhood entirely in $\bar \matH^n$, and (2) every edge of $\hat P$ departing from one vertex in $\bar \matH^n$ with a neighbourhood entirely in $\bar \matH^n$ must end in another such vertex in $\bar \matH^n$. By Lemma \ref{bijection:lemma} (whose proof does not require $P$ to have finite volume) a vertex in $\bar \matH^n$ with a neihbourhood entirely in $\bar\matH^n$ corresponds to either a spherical submatrix of rank $n$ or a Euclidean one of rank $(n-1)$, and an edge exiting from it to a spherical submatrix of rank $n-1$, so
we can rephrase (1) and (2) as stated.
\end{proof}

It is instructive to use points (1) and (2) to deduce the following.

\begin{ex} \label{indec:ex}
The Gram matrix $G$ of a non-obtuse $P\subset \matH^n$ is indecomposable.
\end{ex}

We also mention for completeness the flat and spherical cases, whose proof is simpler. In light of Theorem \ref{non-obtuse:teo} it suffices to consider simplexes. Let $n\geq 1$.

\begin{teo} \label{realization2:teo}
A symmetric $(n+1)\times (n+1)$ matrix is the Gram matrix of a non-obtuse simplex $P\subset \matR^n$ or $P\subset \matS^n$ if and only if $G_{ii} = 1, G_{ij} \leq 0$ for all $i\neq j$, and $G$ has signature $(n,0,1)$ or $(n+1,0,0)$ respectively. The simplex $P$ is uniquely determined by $G$ up to similarities of $\matR^n$ or isometries of $\matS^n$.
\end{teo}
\begin{proof}
By linear algebra we can find some generators $v_j$ in $\matR^{n}$ or $\matR^{n+1}$ having Gram matrix $G$. In the Euclidean case we define $P = \{x \in \matR^n\ |\ \langle x, v_i \rangle \leq 1 \}$.
In the spherical case we have $G=I-B$ with $B\geq 0$, so by Perron Frobenius $G$ has lowest eigenvalue $\lambda > 0$ with eigenvector $w\geq 0$. We set
$v = -\sum_{j} w_jv_j$ and prove that
$\langle v, v_i \rangle = -\sum_{j=1}^k w_j\langle v_j, v_i \rangle = -(Gw)_i  = -\lambda w_i \leq 0
$ with a strict inequality for some $i$. Hence $P = \{x \in \matS^n\ |\ \langle x, v_i \rangle \leq 0 \}$ contains the rescaled $v$ in its interior.
\end{proof}

\section{Polyhedra in dimension 2 and 3} \label{poly23:section}
We apply the theory exposed in the previous section to list and study non-obtuse polyhedra in dimension $n=2$ and $n=3$. It turns out that in these dimensions the non-obtuse polyhedra are completely classified in all geometries. 

Recall that when we say that a polyhedron in $\matX^n$ is unique, we always mean up to isometry in $\matS^n$ and $\matH^n$, and up to similarities in $\matR^n$. 

\subsection{Polygons} Non-obtuse polygons are easily classified.

\begin{prop} \label{triangles:prop}
For every $0\leqslant \alpha, \beta, \gamma \leqslant \pi/2$ there is a unique triangle in $\matX^2$ with angles $\alpha, \beta, \gamma$, where $\matX^2 = \matH^2, \matR^2, \matS^2$ depends on whether the sum $\alpha+\beta+\gamma$ is smaller, equal to, or larger than $\pi$.
\end{prop}
\begin{proof}
The Gram matrix of one such triangle is
$$G = \begin{pmatrix}
1 & -\cos \alpha & -\cos \beta \\
-\cos \alpha & 1 & -\cos \gamma \\
-\cos \beta & -\cos \gamma & 1
\end{pmatrix}$$
and its determinant is
\begin{align*}
\det G & = 1-\cos^2\alpha -\cos^2 \beta - \cos^2\gamma - 2\cos\alpha\cos\beta\cos\gamma \\
& = -4\cos\frac{\alpha+\beta+\gamma}2 \cdot \cos\frac{\alpha+\beta-\gamma}2
\cdot \cos\frac{\alpha-\beta+\gamma}2 \cdot \cos\frac{-\alpha+\beta+\gamma}2.
\end{align*}
We have $\det G > 0, =0, <0$ precisely when $\alpha+\beta+\gamma > \pi, =\pi, <\pi$. The signature of $G$ is accordingly $(3,0,0), (2,0,1), (2,1,0)$ and gives a triangle in $\matS^2, \matR^2, \matH^2$.
\end{proof}

\begin{ex}
For every $k\geq 4$ and $0\leq \theta_1, \ldots, \theta_k \leq \pi/2$ with $\sum\theta_i < (k-2)\pi$ there is a (typically non unique) polygon in $\matH^2$ with consecutive angles $\theta_1,\ldots,\theta_k$. 
\end{ex}

\subsection{Tetrahedra}
Consider a vertex $v$ of a non-obtuse polyhedron $P \subset \matX^3$ as in Figure \ref{v:fig}-(left). The figure shows the dihedral angles $\alpha_1, \alpha_2, \alpha_3 \leq \pi/2$ of the edges and the interior angles $\theta_1,\theta_2,\theta_3\leq \pi/2$ of the faces incident to $v$. Since the link of $v$ is a spherical triangle with angles $\alpha_i$, we have $\alpha_1+\alpha_2+\alpha_3>\pi$. Proposition \ref{face:non-obtuse:prop} says that $\theta_i \leq \alpha_i$. By the spherical law of cosines in fact we have
\begin{equation} \label{cosines:eqn}
\cos \theta_i = \frac{\cos \alpha_i + \cos \alpha_{i+1} \cos \alpha_{i+2}}{\sin \alpha_{i+1} \sin \alpha_{i+2}}.
\end{equation}

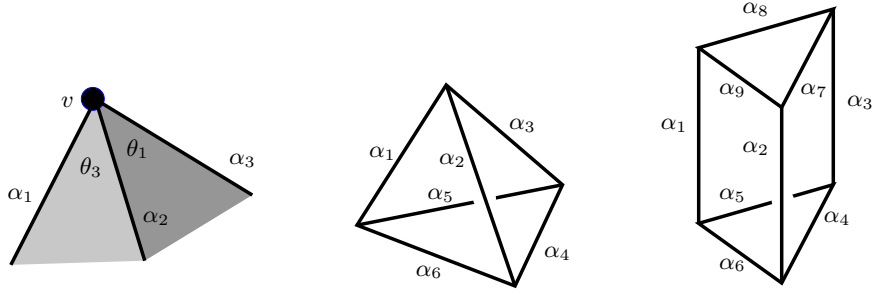
\begin{figure}
 \begin{center}
\centering
\labellist
\small\hair 2pt
\pinlabel $v$ at 25 80
\pinlabel $\alpha_1$ at 5 40
\pinlabel $\alpha_2$ at 63 30
\pinlabel $\alpha_3$ at 100 55
\pinlabel $\theta_3$ at 35 52
\pinlabel $\theta_1$ at 55 60
\pinlabel $\alpha_1$ at 160 58
\pinlabel $\alpha_2$ at 190 55
\pinlabel $\alpha_3$ at 220 70
\pinlabel $\alpha_4$ at 235 15
\pinlabel $\alpha_5$ at 185 40
\pinlabel $\alpha_6$ at 180 7
\pinlabel $\alpha_1$ at 285 70
\pinlabel $\alpha_2$ at 320 60
\pinlabel $\alpha_3$ at 365 80
\pinlabel $\alpha_4$ at 355 30
\pinlabel $\alpha_5$ at 310 40
\pinlabel $\alpha_6$ at 310 10
\pinlabel $\alpha_7$ at 345 85
\pinlabel $\alpha_8$ at 320 120
\pinlabel $\alpha_9$ at 310 85
\endlabellist
  \includegraphics[width = 11 cm]{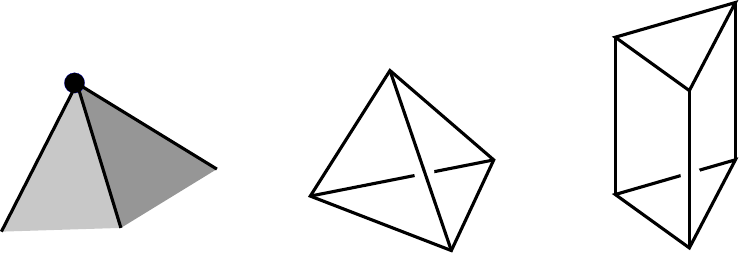}
 \end{center}
 \caption{A real vertex $v$ of a non-obtuse polyhedron $P\subset \matH^3$. Here $\alpha_1,\alpha_2,\alpha_3 \leq \pi/2$ are the dihedral angles and $\theta_1,\theta_2,\theta_3 \leq \pi/2$ are the angles of the faces adjacent to $v$, with $\theta_i$ opposite to $\alpha_i$ (left). A tetrahedron with dihedral angles $\alpha_1,\ldots, \alpha_6$ (center) and a triangular prism with dihedral angles $\alpha_1,\ldots,\alpha_9$ (right).} \label{v:fig}
\end{figure}

With this formula we can deduce the angles of the faces of $P$ from its dihedral angles. It is also valid if $v$ is ideal: in this case $\alpha_1+\alpha_2+\alpha_3=\pi$ and  $\theta_i=0$.

Let $T$ be a tetrahedron as in Figure \ref{v:fig}-(center), with some abstract dihedral angles $0 < \alpha_i \leq \pi/2$ assigned to its edges. We assume that $\alpha_i+\alpha_j+\alpha_k\geq \pi$ at each vertex, and we use \eqref{cosines:eqn} to assign three abstract angles $0\leq  \theta_i,\theta_j,\theta_k \leq \pi/2$ to each triangular face of $T$, that depend on the dihedral angles $\alpha_i$. 

The following theorem is proved in the compact case by Roeder \cite{R}. The statement is surprisingly simple and similar to Proposition \ref{triangles:prop}.

\begin{teo}
There exists a unique tetrahedron in $\matX^3$ with dihedral angles $\alpha_i$, where
$\matX^3 = \matH^3, \matR^3, \matS^3$ depends on whether the sum  $\theta_i + \theta_j + \theta_k$ of the angles of some (and hence every) face is smaller, equal to, or larger than $\pi$.
\end{teo}
\begin{proof}
The condition is clearly necessary since each face lies in a copy of $\matS^2, \matR^2, \matH^2$ correspondingly. To show that it is sufficient, we write the Gram matrix 
$$G = \begin{pmatrix}
1 & -\cos \alpha_{1} & -\cos \alpha_{2} & -\cos \alpha_6 \\
-\cos \alpha_1 & 1 & -\cos\alpha_3 & -\cos\alpha_5 \\
-\cos \alpha_2 & -\cos\alpha_3 & 1 & -\cos\alpha_4 \\
-\cos \alpha_6 & -\cos\alpha_5 & -\cos\alpha_4 & 1
\end{pmatrix}
$$
and note that since $\alpha_i+\alpha_j+\alpha_k\geq \pi$ at every vertex, every principal $3\times 3$ submatrix is either spherical or Euclidean. Set $s_i = \sin\alpha_i$. Via Gauss moves we find
\begin{align*}
\det G & = 
s_1^2s_2^2s_6^2 \det \begin{pmatrix}
1 & -\cos \theta_i & -\cos \theta_j \\
-\cos \theta_i & 1 & -\cos \theta_k \\
-\cos \theta_j & -\cos \theta_k & 1
\end{pmatrix} \\
& =-4s_1^2s_2^2s_6^2 
\cos\!\frac{\theta_i+\theta_j+\theta_k}2 \cos\!\frac{\theta_i+\theta_j-\theta_k}2
\cos\!\frac{\theta_i-\theta_j+\theta_k}2  \cos\!\frac{-\theta_i+\theta_j+\theta_k}2
\end{align*}
where $\theta_i,\theta_j,\theta_k$ are the interior angles of the front face of $T$ in Figure \ref{v:fig}-(center). If at least one principal $3\times 3$ matrix is spherical, it has signature $(3,0,0)$, and hence the signature of $G$ is $(3,1,0), (3,0,1), (4,0,0)$ depending on whether $\det G <0, =0, >0$, that is on whether $\theta_i+\theta_j+\theta_k <\pi, =\pi, >\pi$. If all the principal $3\times 3$ matrices are Euclidean, we have $\alpha_i+\alpha_j+\alpha_k=\pi$ at each vertex, hence $\theta_a=0$ for all $a$ and therefore $\det G < 0$ implies that the signature is $(3,1,0)$.
\end{proof}

Every assignment of dihedral angles to $T$ such that $\alpha_i+\alpha_j+\alpha_k \geq \pi$ at each vertex has a unique realization in the appropriate geometry. If $\alpha_i+\alpha_j+\alpha_k=\pi$ at some vertex, the tetrahedron is hyperbolic and this vertex is ideal. When all the vertices are ideal we can easily deduce that $\alpha_1=\alpha_4, \alpha_2=\alpha_5,\alpha_3=\alpha_6$.

\begin{ex}
Consider a triangular prism with angles $0<\alpha_1,\ldots, \alpha_9\leq \pi/2$ as in Figure \ref{v:fig}-(right). We assume that $\alpha_i+\alpha_j+\alpha_k \geq \pi$ at every vertex. There exists a triangular prism in $\matH^3$ with dihedral angles $\alpha_i$ if and only if the following holds:
\begin{enumerate}
\item $\alpha_1+\alpha_2+\alpha_3 < \pi$;
\item $(\alpha_4,\ldots,\alpha_9) \neq (\pi/2, \ldots, \pi/2)$.
\end{enumerate}
The hyperbolic polyhedron is unique.
\end{ex}

\subsection{The Andreev -- Roeder Theorem}
An elegant theorem of Andreev characterizes completely the non-obtuse hyperbolic polyhedra in $\matH^3$. The original proof in the compact case \cite{A} contained a gap that was fixed by Roeder \cite{RHD}. The extension to the non-compact case is also due to Andreev \cite{A2}, and it is obtained by approximating a non-compact polyhedron via compact ones.

Let $P\subset \matR^3$ be a 3-dimensional polyhedron such that every vertex is adjacent to either 3 or 4 edges. Assume that $P$ is neither a tetrahedron nor a triangular prism. Let us assign some abstract dihedral angles $0<\alpha_i\leq \pi/2$ to the edges of $P$. 

\begin{figure}
 \begin{center}
\centering
\labellist
\small\hair 2pt
\pinlabel $\alpha_1$ at -10 70
\pinlabel $\alpha_2$ at 30 40
\pinlabel $\alpha_3$ at 95 95
\pinlabel $\alpha_1$ at 110 70
\pinlabel $\alpha_2$ at 150 40
\pinlabel $\alpha_3$ at 255 60
\pinlabel $\alpha_4$ at 215 95
\pinlabel $\alpha_1$ at 270 70
\pinlabel $\alpha_2$ at 310 40
\endlabellist
  \includegraphics[width = 12 cm]{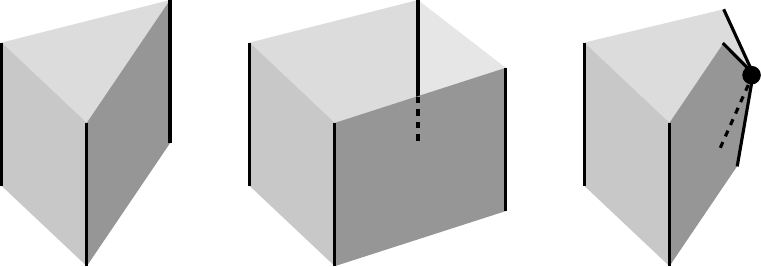}
 \end{center}
 \caption{Some configurations of faces. In the left (center) figure we suppose that the 6 (8) endpoints of the 3 (4) edges cointaining the labels $\alpha_i$ are all distinct. In the right figure we suppose that the left face does not contain the right vertex.} \label{Andreev:fig}
\end{figure}
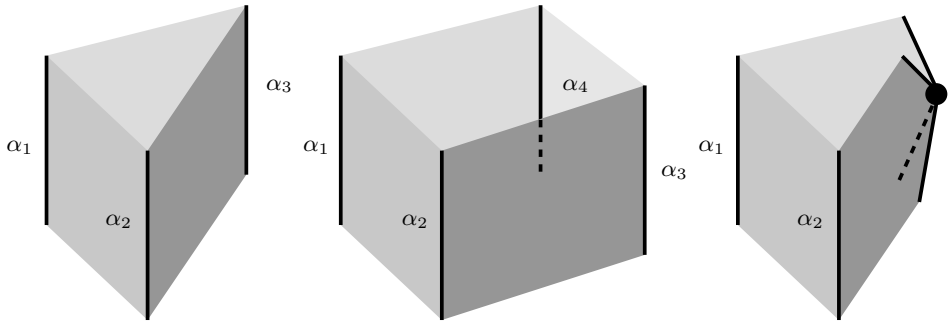

\begin{teo}
The polyhedron $P$ can be realized as a hyperbolic polyhedron $P\subset \matH^3$ with the assigned dihedral angles if and only if the following holds:
\begin{enumerate}
\item $\alpha_i+\alpha_j+\alpha_k\geq \pi$ at each 3-valent vertex;
\item $\alpha_i +\alpha_j+\alpha_k+\alpha_l= 2\pi$ at each 4-valent vertex;
\item $\alpha_1 + \alpha_2 + \alpha_3 < \pi$ for every three faces as in Figure \ref{Andreev:fig}-(left);
\item $\alpha_1 + \alpha_2 + \alpha_3 +\alpha_4 < 2\pi$ for every four faces as in Figure \ref{Andreev:fig}-(center);
\item $\alpha_1+\alpha_2 < \pi$ for every three faces as in Figure \ref{Andreev:fig}-(right).
\end{enumerate}
The hyperbolic polyhedron $P$ is unique up to isometry.
\end{teo}

Conditions (2), (4), (5) are equivalent to $(\alpha_i, \alpha_j, \alpha_k, \alpha_l) = (\pi/2,\pi/2,\pi/2,\pi/2)$, $(\alpha_1,\alpha_2,\alpha_3,\alpha_4) \neq (\pi/2,\pi/2,\pi/2,\pi/2)$, $(\alpha_1,\alpha_2) \neq (\pi/2,\pi/2)$ respectively.

\subsection{Higher dimension?}
Having completely classified the non-obtuse polyhedra in dimension $n\leq 3$ and in all geometries, it is natural to wonder whether this elegant picture extends somehow in higher dimension. The scene changes abruptly when $n\geq 4$: there is no known general classification, not even conjecturally, of non-obtuse polyhedra, say, in $\matH^4$ or $\matH^5$. 

\section{Coxeter polyhedra} \label{Coxy:section}

We finally introduce the main protagonist of this paper. Coxeter polyhedra are non-obtuse polyhedra with nice dihedral angles: they inherit all the good properties of non-obtuse polyhedra, and add many more, so that they can be used as building blocks to construct various more complicated objects, like uniform polyhedra, discrete groups, tessellations, hyperbolic manifolds $\ldots$

\subsection{Tessellations}
A \emph{tessellation} of $\matX^n$ is a locally finite covering $T$ of $\matX^n$ with polyhedra that pairwise intersect only in mutual faces. All the polyhedra and their faces form the \emph{faces} of $T$, and those of dimension $n$ are called \emph{facets}. If $\matX^n = \matH^n$ the \emph{ideal vertices} of $T$ are those of the polyhedra in $T$.


\subsection{Coxeter polyhedron}
Let $P\subset \matX^n$ be a polyhedron with $n\geq 2$ if $\matX^n = \matH^n$ and $n\geq 1$ if $\matX^n = \matR^n$. We say that $P$ is a
\emph{Coxeter polyhedron} if its Gram matrix $G$ satisfies the following requirement: whenever $G_{ij} \in (-1,1)$ we must have $G_{ij} = - \cos (\pi/m_{ij})$ for some $m_{ij} \geq 2$. In particular, we have $G_{ij} \leq 0$ for every $i\neq j$, and hence $P$ is non-obtuse.

This definition is robust because it is expressed solely in terms of $G$, but we need a more geometric equivalent formulation. In virtue of Andreev's Theorem \ref{nonobtuse:teo}, when $n\geq 2$ the polyhedron $P$ has $k$ facets $F_1,\ldots, F_k$, and $G_{ij} = -\cos(\pi/m_{ij})$ if and only if the facets $F_i$ and $F_j$ intersect with dihedral angle $\pi/m_{ij}$.
Therefore a Coxeter polyhedron is:
\begin{itemize}
\item a point in $\matS^0$, or 
\item a segment in $\matS^1$ of length $\ell = \pi/k$ for some $k\geq 2$, or any segment in $\matR$, or
\item a polyhedron $P \subset \matX^n$ with $n\geq 2$ whose dihedral angles divide $\pi$.  
\end{itemize}

A Coxeter polyhedron $P\subset \matX^n$ determines a \emph{Coxeter group} $\Gamma < \Isom(\matX^n)$ as follows. If $n\geq 1$, let $F_1,\ldots, F_k$ be the facets of $P$ and let $r_i \in \Isom(\matX^n)$ be the reflection along the hyperplane containing $F_i$, for $i=1,\ldots, k$. The Coxeter group $\Gamma$ is the group generated by these reflections $r_1,\ldots, r_k$. In the special case where $P$ is a point in $\matS^0$, we define $\Gamma = \Isom(\matS^0) = \{ \pm 1\}$ generated by $r_1=-1$. 

The great relevance of Coxeter polyhedra in geometry stems from the following fundamental fact, proved by Coxeter \cite{C} in 1934 for $\matS^n$ and $\matR^n$ and then easily generalized to $\matH^n$.

\begin{teo}
The Coxeter group $\Gamma < \Iso(\matX^n)$ is discrete and $\{g(P)\ |\ g \in \Gamma\}$ is a tessellation of $\matX^n$. The Coxeter group $\Gamma$ has the following presentation:
$$\langle\ r_l\ | \ r_l^2, \ (r_ir_j)^{m_{ij}}\ \rangle
$$
where $l=1,\ldots, k$ and $i \neq j$ are such that $G_{ij} = - \cos (\pi/m_{ij})$. When $n \geq 2$ this is equivalent to saying that the facets $F_i, F_j$ intersect with dihedral angle $\pi/m_{ij}$.

Every discrete group $\Gamma < \Iso(\matX^n)$ generated by reflections along hyperplanes with finite volume quotient is the Coxeter group of some Coxeter polyhedron $P$.
\end{teo}
\begin{proof}[Sketch of the proof]
We proceed by induction on $n$. For every $g\in \Gamma$ we define a copy $P_g$ of $P$, and we identify the facet $F_i$ of $P_g$ with the same facet $F_i$ of $P_{gr_i}$ for all $i,g$. We show that the resulting space $X$ is naturally isometric to $\matX^n$. 

We first prove this locally. For every $h$-face $f$ of $P$, the subgroup $\Gamma_f<\Gamma$ generated by the reflections $r_i$ such that $f\subset F_i$ is the Coxeter group of the link $Q \subset \matS^{n-h-1}$ of $f$. By the inductive hypothesis $\{g(Q)\ |\ g\in\Gamma_f\}$ form a tessellation of $\matS^{n-h-1}$.  Every point in $X$ is contained in the translate of some face $f$, and since its link in $X$ is $\matS^{n-h-1}$, it has a neighbourhood isometric to a small ball in $\matX^n$.

We have proved that $X$ is locally isometric to $\matX^n$, and moreover it is complete. If $P$ is compact this is obvious; if $P\subset \matH^n$ has some ideal vertex $v$ some care is needed, and we conclude by induction as above using the link $Q \subset \matR^{n-1}$ of $v$.

The developing map $X \to \matX^n$ that sends $P_g$ to $g(P)$ is a local isometry between two complete metric spaces, hence it is a covering, hence an isometry because $\matX^n$ is simply connected. Thus via this identification we get the first part of the theorem. 

We prove that a presentation for $\Gamma$ is as stated: let $\Gamma'$ be the abstract group defined via our candidate presentation. We have a surjection $\Gamma' \to \Gamma$, and by repeating the same construction using $\Gamma'$ instead of $\Gamma$ we get another $X'$ that covers $\matX^n$ isometrically. We deduce that $X=X'$ and $\Gamma = \Gamma'$.

Finally, if $\Gamma < \Isom(\matX^n)$ is discrete and has finite volume quotient, the fixed hyperplanes of all the reflections in $\Gamma$ cut $\matX^n$ into a tessellation where each polyhedron $P$ is a fundamental domain, and $\Gamma$ is the Coxeter group of $P$.
\end{proof}

By Proposition \ref{Gf:prop} the link of a face or of an ideal vertex of a Coxeter polyhedron is also a Coxeter polyhedron, of spherical and Euclidean type respectively. 

\begin{warning}
The face of a Coxeter polyhedron may not be a Coxeter polyhedron! Its dihedral angles are non-obtuse but may not divide $\pi$. 
\end{warning}

\subsection{Coxeter diagrams}
We know from Theorems \ref{realization:teo} and \ref{realization2:teo} that a Coxeter polyhedron $P\subset \matX^n$ is fully determined by its $k\times k$ Gram matrix $G$. It is often convenient to describe $G$ via a \emph{Coxeter diagram}, that is a graph $D$ with $k$ nodes, that we temporarily label as $1,\ldots, k$, and: 
\begin{enumerate}
\item One edge decorated with $m_{ij}$ connecting the nodes with labels $i,j$ if $G_{ij} = - \cos (\pi/m_{ij})$ and $m_{ij} \geq 3$. If $m_{ij}=3$, the number $m_{ij}$ is omitted;
\item One \emph{thick} edge connecting the nodes with labels $i,j$ if $G_{ij} = -1$; 
\item One \emph{dashed} edge decorated with $d_{ij}$
connecting the nodes with labels $i,j$ if $G_{ij} = - \cosh d_{ij}$. 
\end{enumerate}


The Coxeter diagram $D$ is \emph{spherical}, \emph{Euclidean}, or \emph{hyperbolic} according to the geometry of $\matX^n$; the geometry can be deduced directly from $D$ since it depends on $\det G$ being positive, null, or negative, and $G$ is fully encoded by $D$.

When $n\geq 1$ the nodes in $D$ correspond to the facets of $P$, and when $n\geq 2$ the edges of type (1) correspond to the ridges of $P$.
When $n\geq 2$, two nodes are \emph{not} connected by an edge $\Longleftrightarrow$ the corresponding facets meet with angle $\pi/2$.
Coxeter diagrams are visually convenient when $P$ has few facets and many right-angled dihedral angles, since no edge is drawn for these. 

A Coxeter polyhedron $P$ is \emph{irreducible} if its diagram $D$ is connected, and \emph{reducible} otherwise. This is equivalent to the corresponding Gram matrix $G$ being indecomposable. By Exercises \ref{G:ex} and \ref{indec:ex}, a Coxeter polyhedron $P$ is reducible if and only if it is either a product (in $\matR^n$) or a join (in $\matS^n$) of two Coxeter polyhedra of smaller dimensions. Every Coxeter polyhedron in $\matH^n$ is automatically irreducible, and every simplex in $\matR^n$ is automatically irreducible.

\subsection{Coxeter subdiagrams}
Let a Coxeter diagram $D$ describe a Coxeter polyhedron $P \subset \matX^n$.
A set of nodes in $D$ generates a subdiagram $D'\subset D$ that consists of these nodes plus all the edges in $D$ joining them. 
Subdiagrams correspond to principal submatrices of the Gram matrix. 

If a subdiagram $D' \subset D$ is the Coxeter diagram of some Coxeter polytope $Q\subset \matX^k$, we say that it is a \emph{Coxeter subdiagram}.
By Lemma \ref{bijection:lemma}, the faces of $P$ correspond to the spherical Coxeter subdiagrams of $D$, while the ideal vertices of $P$ correspond to the Euclidean subdiagrams with $k=n-1$. 

\begin{figure}
 \begin{center}
\centering
\labellist
\small\hair 2pt
\endlabellist
  \includegraphics[width = 11 cm]{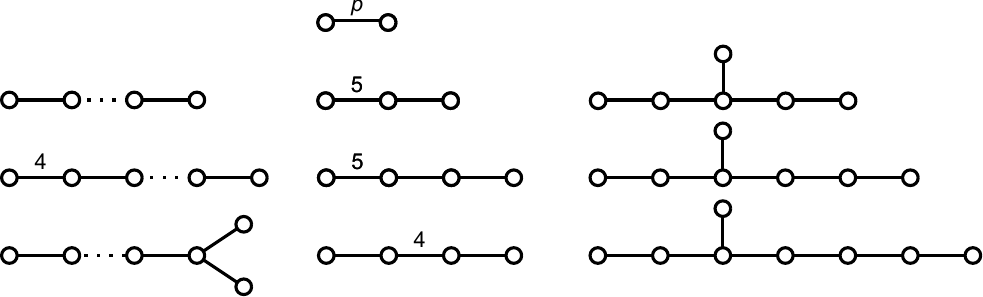}
 \end{center}
 \caption{The Coxeter diagrams of the irreducible spherical Coxeter simplexes. The diagrams in the first column have at least 1, 2, 4 nodes respectively, and we also have $p\geq 5$ (the cases $p=3,4$ are covered by other diagrams). The diagrams with one node, and with two nodes with edge label $p\geq 3$, represent a point in $\matS^0$ and an arc in $\matS^1$ of length $\pi/p$.}  \label{Coxeter-spherical:fig}
\end{figure}

\begin{figure}
 \begin{center}
\centering
\labellist
\small\hair 2pt
\endlabellist
  \includegraphics[width = 11 cm]{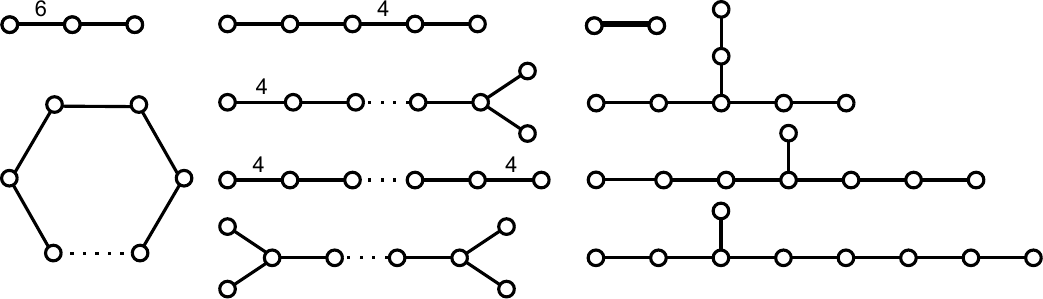}
 \end{center}
 \caption{The Coxeter diagrams of the Euclidean Coxeter simplexes. The bottom left diagram is a closed polygon with at least 3 nodes. The three diagrams in the bottom center have at least 4, 3, 5 nodes respectively. The top right diagram with two nodes and one thickened edge represents a segment in $\matR$.}  \label{Coxeter-Euclidean:fig}
\end{figure}

\begin{figure}
 \begin{center}
\centering
\labellist
\small\hair 2pt
\endlabellist
  \includegraphics[width = 11 cm]{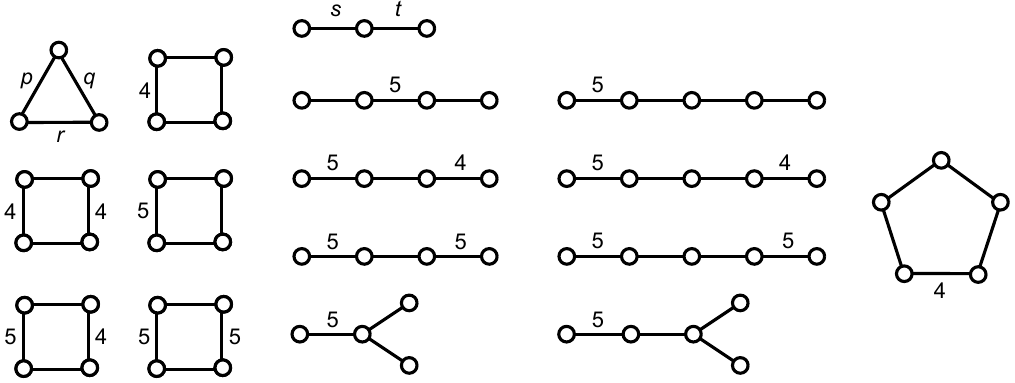}
 \end{center}
 \caption{The Coxeter diagrams of the compact hyperbolic Coxeter simplexes. Here $p,q,r,s,t\geq 3$ with $(p,q,r)\neq (3,3,3)$, and $(s,t) \neq (3,3), (3,4), (3,5), (3,6)$ and their permutations.}  \label{Coxeter-hyperbolic:fig}
\end{figure}

\begin{figure}
 \begin{center}
\centering
\labellist
\small\hair 2pt
\endlabellist
  \includegraphics[width = 11.5 cm]{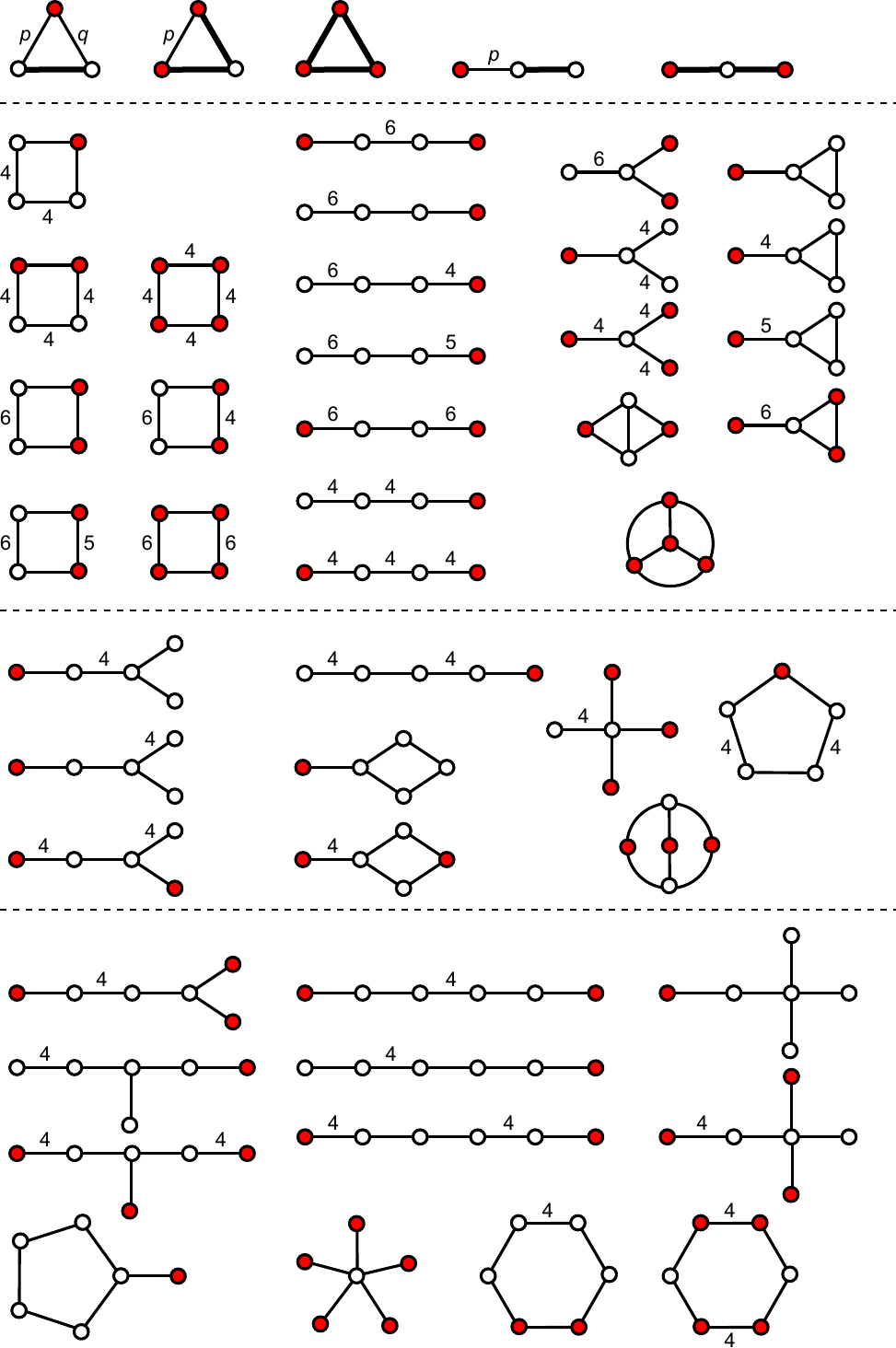}
 \end{center}
 \caption{The Coxeter diagrams of the non-compact hyperbolic Coxeter simplexes of dimension $2,3,4$ and 5. Here $p,q\geq 3$. The red nodes indicate the facets that are opposite to the ideal vertices.}  \label{Coxeter-hyperbolic-ideal:fig}
\end{figure}

\begin{figure}
 \begin{center}
\centering
\labellist
\small\hair 2pt
\endlabellist
  \includegraphics[width = 12.5 cm]{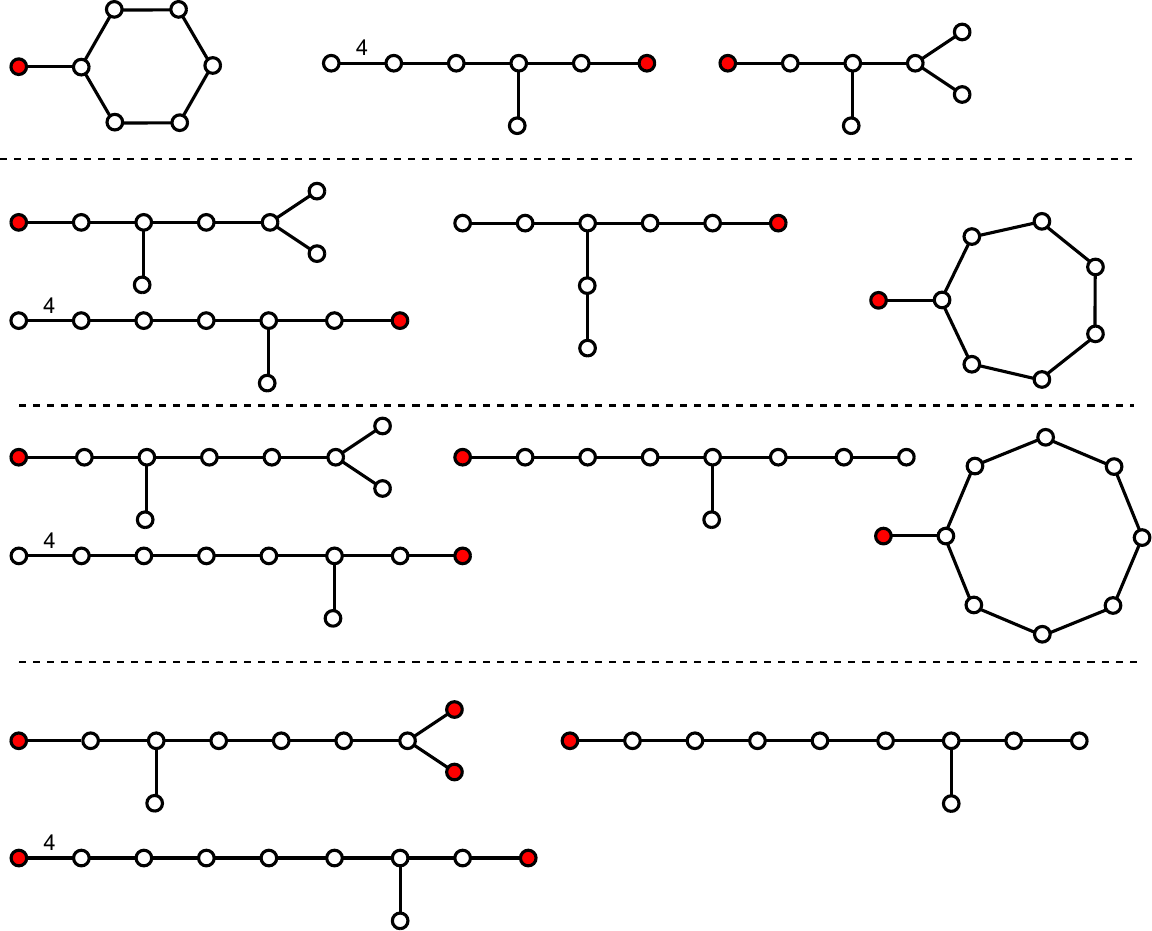}
 \end{center}
 \caption{The Coxeter diagrams of the non-compact hyperbolic Coxeter simplexes of dimension $6, 7, 8$ and $9$. The red nodes indicate the facets that are opposite to the ideal vertices.}  \label{Coxeter-hyperbolic-ideal2:fig}
\end{figure}

\subsection{Coxeter simplexes}
The Coxeter diagram of a Coxeter $n$-simplex is quite peculiar: it has $n+1$ nodes and no dashed edges, and no thick edges if $n\geq 3$; by removing a node, that corresponds to some facet $F$, we get the (spherical or Euclidean) Coxeter diagram of the link of the (real or ideal) vertex opposite to $F$.

The classification of Coxeter simplexes is due to Coxeter \cite{C} for $\matS^n$, $\matR^n$, to Lann\'er \cite{L} for $\matH^n$ in the compact case, and to Koszul \cite{K} and Chein \cite{Ch} in the non-compact case. They have altogether proved the following remarkable theorem. Recall that every Coxeter simplex in $\matR^n$ or $\matH^n$ is automatically irreducible.

\begin{teo} \label{irreducible:Coxeter:teo}
The irreducible Coxeter simplexes in $\matS^n, \matR^n, \matH^n$ are precisely those represented by the Coxeter diagrams shown in Figure \ref{Coxeter-spherical:fig}, \ref{Coxeter-Euclidean:fig}, \ref{Coxeter-hyperbolic:fig}, \ref{Coxeter-hyperbolic-ideal:fig}, and \ref{Coxeter-hyperbolic-ideal2:fig}.
\end{teo}

\begin{proof}[Sketch of the proof]
By Vinberg's Realization Theorem \ref{realization:teo}, a diagram with $n+1$ nodes and some edges, some of which are labeled with integers $\geq 3$, is the Coxeter diagram of a $n$-simplex in some $\matX^n$ $\Longleftrightarrow$ each of the $n+1$ subdiagrams obtained by removing one node is either a spherical or a Euclidean Coxeter diagram. 

We proceed by induction on $n$. Coxeter triangles are easily classified. Having already classified all the Coxeter diagrams representing $(n-1)$-simplexes, we make a list of all the connected diagrams with $n+1$ nodes such that by removing any vertex we always get the disjoint union of some connected Coxeter diagrams, that are either all spherical or all Euclidean. We identify the geometry of the new $n$-simplex by calculating the determinant of the Gram matrix.

If we are interested only in compact polyhedra, then only spherical subdiagrams are allowed. This makes the classification much shorter and easier to obtain.
\end{proof}

The following corollary is particularly useful when we want to study the faces and ideal vertices of a Coxeter polyhedron $P$.

\begin{cor}
Let $D$ be a Coxeter diagram. A subdiagram $D'\subset D$ is a Coxeter spherical (Euclidean) subdiagram $\Longleftrightarrow$ it is a disjoint union of diagrams shown in Figure \ref{Coxeter-spherical:fig} (\ref{Coxeter-Euclidean:fig}).
\end{cor}

\begin{rem}
The spherical Coxeter diagrams in Figure \ref{Coxeter-spherical:fig}, listed from top left downwards, are often denoted via the symbols $A_n, B_n, D_n, I_p, H_3, H_4, F_4, E_6, E_7, E_8$. 
All the subscripts except $p$ indicate the number of nodes. 
Sometimes the notation $G_2 = I_6$, $H_2=I_5$, and $C_n=B_n$ is also employed. The Euclidean ones in Figure \ref{Coxeter-Euclidean:fig} are also labeled with similar symbols $\tilde G_2, \tilde A_n, \tilde F_4, \tilde B_n, \tilde C_n, \tilde D_n, \tilde I_1, \tilde E_6, \tilde E_7, \tilde E_8$. 
\end{rem}

\section{Regular and uniform polyhedra and tessellations} \label{regy:section}
In this section we introduce and describe a series of very symmetric and beautiful geometric objects. We study the polyhedra and tessellations with the highest degrees of symmetries: these are called, from the most symmetric to the least, \emph{regular}, \emph{semiregular}, and \emph{uniform}. In some contexts these objects are completely classified, in some others there are only conjecturally complete lists, and in some situations there are just too many objects and no conjectural general picture. 

Coxeter diagrams are of course the most powerful tool to study these very symmetric entities. We describe in particular a geometric fruitful manipulation called the \emph{Wythoff construction} that transforms a Coxeter diagram into a uniform polyhedron or tessellation. Most (but not all) of the symmetric objects that we describe here will be obtained in this way.

\subsection{Definitions}
Tessellations of dimension $n$ are similar to polyhedra of dimension $n+1$ in many aspects, the most important one being that they both have faces of dimension $\leq n$ that are themselves polyhedra.

Let $X$ be either a tessellation of $\matX^n$ or a polyhedron in $\matX^{n+1}$. In both cases $X$ has faces of dimension $\leq n$ and facets of dimension $n$. We exceptionally consider all the ideal vertices of $X$ as vertices, and hence also as faces, of $X$, with the obvious containments (an ideal vertex $v$ is contained in every face whose closure contains $v$).  
A \emph{flag} in $X$ is a sequence $f_0\subset \dots \subset f_{n}$ where $f_i$ is an $i$-face of $X$. 

An \emph{isometry} of $X$ is an isometry of $\matX^n$ or $\matX^{n+1}$ that preserves the faces of $X$ as a set (when $X$ is a polyhedron this is equivalent to an isometry that preserves $X$).

\begin{defn}
We say that $X$ is:
\begin{enumerate}
\item \emph{Regular} if its isometries act transitively on the flags of $X$;
\item \emph{Semiregular} if its isometries act transitively on the vertices of $X$, and all the facets are regular;
\item \emph{Uniform} if its isometries act transitively on the vertices of $X$, and all the facets are regular (if $n=2$) or uniform (defined recursively, if $n\geq 3$).
\end{enumerate}
\end{defn}

Here ideal vertices count as vertices, so the transitive action on the vertices implies in all cases that the vertices of $X$ are either all real or all ideal. A polyhedron or tessellation is regular if and only if its isometry group acts transitively on the maximal simplexes of its barycentric subdivision. Of course (1) $\Longrightarrow$ (2) $\Longrightarrow$ (3). We have (1) $\Longleftrightarrow$ (2) if $n=1$ and (2) $\Longleftrightarrow$ (3) if $n\leq 2$.



\subsection{The Wythoff construction} \label{Wythoff:subsection}
We introduce a geometric construction that generates many uniform tessellations out of a single Coxeter simplex.

Let $P\subset \matX^n$ be an irreducible compact Coxeter simplex, equipped with a fixed \emph{seed} point $p\in P$. 
Pick the half-lines $l_1,\ldots,l_{n+1} \subset \matX^n$ centered in $p$ orthogonal to the facets of $P$ (pointing outward, like the normal vectors of these facets; since $P$ is irreducible, in the spherical case every facet has distance $<\pi/2$ from $p$ and hence the half-lines are well-defined).
The \emph{dual star} in $\matX^n$ with center $p$ is the union of the $\tvect{n+1}2$ distinct $(n-1)$-dimensional cones with vertex $p$ obtained as the convex hull of $n-1$ distinct half-lines in $l_1,\ldots,l_{n+1}$ (if $n=2$ the dual star is just the union of the three half-lines $l_1\cup l_2 \cup l_3$). The polyhedron $P$ intersects the dual star into a codimension-1 complex $Y$ that depends on $p$, see Figure \ref{Wythoffian:fig}.

\begin{figure}
 \begin{center}
\centering
\labellist
\small\hair 2pt
\endlabellist
  \includegraphics[width = 12.5 cm]{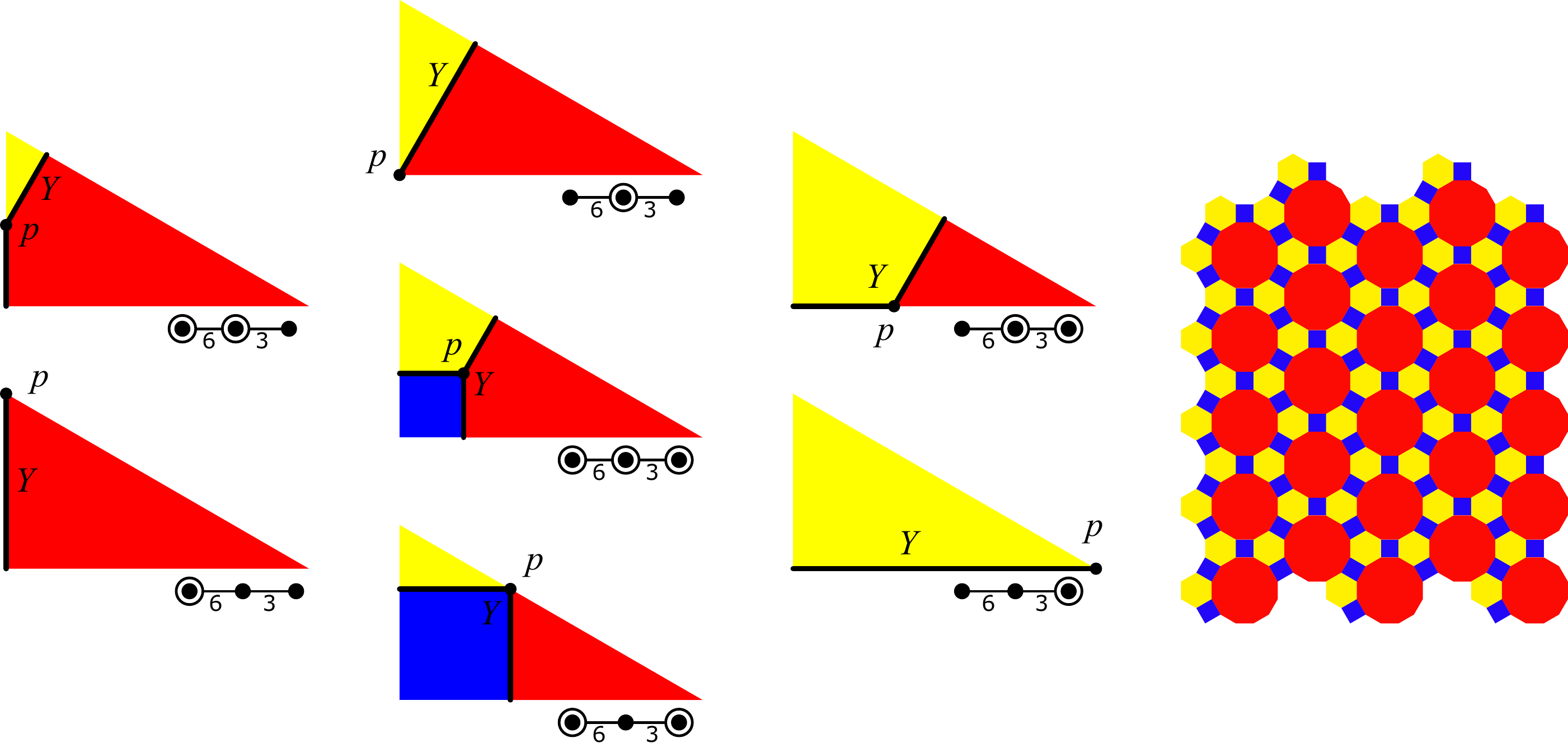}
 \end{center}
 \caption{The Wythoff construction on a Euclidean triangle with angles $\pi/2, \pi/3,\pi/6$ at various points $p$. The 1-dimensional complex $Y$ is drawn in black and depends on $p$. The central configuration produces the tessellation shown on the right. In the examples shown here $p$ is always equidistant from the sides that do not contain it and hence the resulting tessellation is uniform. The Coxeter -- Wythoff diagram is shown in each case.}  \label{Wythoffian:fig}
\end{figure}

Recall that $P=\matX^n/\Gamma$ where $\Gamma$ is generated by the reflections along the facets of $P$. The preimage $\tilde Y$ of $Y$ in $\matX^n$ along the quotient map $\matX^n \to P$ is a codimension-1 subcomplex in $\matX^n$, and the closures of the connected components of its complement form a tessellation $T$ of $\matX^n$ that depends only on $P$ and on the seed $p$. 
See Figure \ref{Wythoffian:fig}.
We say that $T$ is obtained from $P$ via the \emph{Wythoff construction} with seed $p$. 

\begin{rem}
If $\matX^n=\matS^n$, we interpret the tessellation $T$ as a polyhedron $Q \subset \matR^{n+1}$ by taking the convex hull of its vertices. In fact in this case the construction is much simpler to define: the polyhedron $Q$ is just the convex hull of the $\Gamma$-orbit of $p$. By construction $T$ and $Q$ have the same combinatorics.
\end{rem}

\begin{ex} \label{ideal:ex}
The Wythoff construction extends naturally to the case where $P\subset \matH^n$ is a simplex with only one ideal vertex $v$ and the seed $p$ is positioned at $v$. The construction produces a tessellation $T$ of $\matH^n$ into ideal polyhedra. (If we allow more ideal vertices, or a different positioning for $p$, we very likely get tessellations with infinite volume polyhedra, and these are not allowed as tessellations here.)
\end{ex}

\subsection{Well positioned seeds} The Wythoff construction depends continuously on the seed, and we now show that by putting the seed in some nice position we are guaranteed to have a uniform polyhedron or tessellation.

\begin{ex}
Let $P \subset \matX^n$ be a compact simplex. Every face $f$ of $P$ contains a unique point $p$ that is equidistant from all the facets of $P$ not containing $f$.
\end{ex}

We say that a point $p\in P$ as in the previous exercise is \emph{well positioned}. 

\begin{prop}
Let $P\subset \matX^n$ be a Coxeter simplex. A tessellation obtained from the Wythoff construction is uniform if and only if the seed $p$ is well positioned.
\end{prop}
\begin{proof}
By construction the isometry group of the tessellation acts transitively on the vertices, and also on the vertices of each face of the tessellation (fixing the face). By induction on $n$ one sees that such a tessellation is uniform if and only if all the edges have the same length, and this holds precisely when $p$ is well positioned.
\end{proof}

\begin{ex}
If $P\subset \matH^n$ has one ideal vertex $v$ and the seed $p$ is at $v$, the resulting tessellation $T$ is uniform.
\end{ex}


\subsection{Coxeter -- Wythoff diagrams}
We now translate everything into some appropriate diagrams, that will enable us to apply the Wythoff machinery in a simple and systematic way.

A \emph{Coxeter -- Wythoff diagram} is a diagram $D$ of an irreducible Coxeter simplex $P\subset \matX^n$ with some (at least one) encircled nodes. If $P$ is non-compact hyperbolic, we require that $P$ has only one ideal vertex $v$ and $D$ has only one encircled node, corresponding to the facet opposite to $v$.

The encircled nodes determine a seed point $p$ in $P$. In the ideal case, we set $p=v$. In the compact case $p$ is the well positioned point in the face that is the intersection of the facets corresponding to the unencircled nodes. The point $p$ is thus equidistant from the facets corresponding to the encircled nodes. See some examples in Figure \ref{Wythoffian:fig}.

\begin{figure}
 \begin{center}
\centering
\labellist
\small\hair 2pt
\endlabellist
  \includegraphics[width = 12.5 cm]{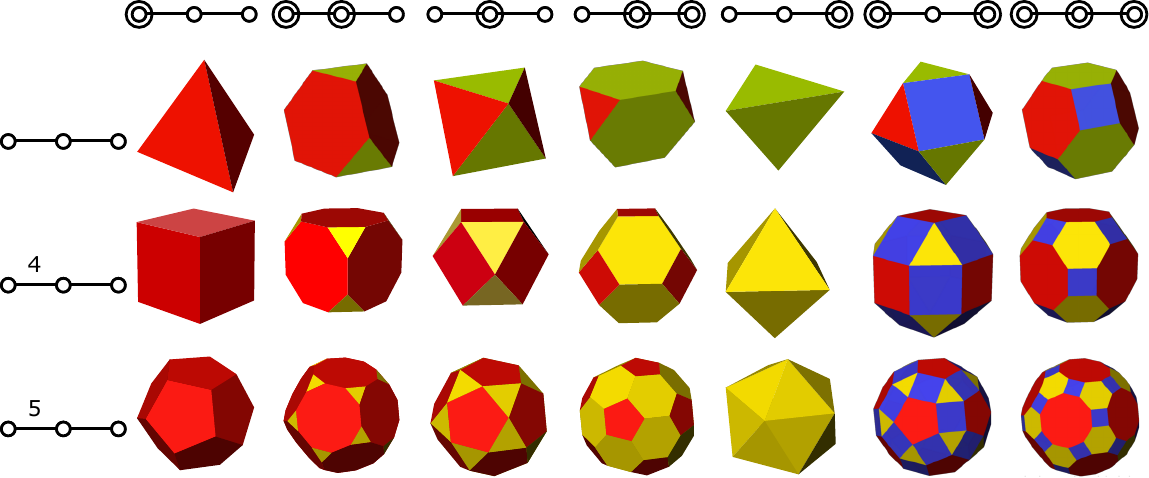}
 \end{center}
 \caption{The uniform polyhedra obtained from the Wythoff construction from a spherical diagram with three nodes.}  \label{Wythoff:fig}
\end{figure}

A Coxeter -- Wythoff diagram $D$ determines a Coxeter polyhedron $P$ and a well-positioned seed $p$, and therefore a uniform tessellation $T$ of $\matX^n$ by applying the Wythoff construction. 
When $\matX^n = \matS^n$ we interpret the uniform tessellation $T$ of $\matS^n$ as a uniform polyhedron $Q$ in $\matR^{n+1}$, by taking the convex hull of its vertices. 
The reader may verify that the spherical Coxeter -- Wythoff diagrams with three nodes produce precisely the uniform polyhedra in $\matR^3$ shown in Figure \ref{Wythoff:fig}.

\subsection{Subdiagrams describe faces}
We now introduce a simple combinatorial method to perfectly understand the face structure of a polyhedron or tessellation produced by a Wythoff construction.

Let $D$ be a Coxeter -- Wythoff diagram, producing a uniform tessellation $T$ of $\matX^n$. A \emph{Coxeter -- Wythoff subdiagram} of $D$ is a proper Coxeter subdiagram $D'\subset D$ such that each connected component of $D'$ contains at least one encircled node. 

A Coxeter -- Wythoff subdiagram $D'\subset D$ with $h$ nodes represents the spherical link of some codimension-$h$ face $f$ of $P$, and it also determines a uniform polyhedron $P_f\subset \matR^{h}$ that is the product of the Euclidean polyhedra produced (via the Wythoff construction) by each connected component of $D'$.
We can check that the tessellation $T$ has a $h$-face orthogonally transverse to $f$ and combinatorially equivalent to $P_f$, and that every face of $T$, considered up to the action of $\Gamma$, arises uniquely in this way. See Figure \ref{Wythoffian:fig} for some examples. We summarize our discoveries:


\begin{prop}
The $h$-faces of $T$, considered up to the action of $\Gamma$, are in natural bijection with the Coxeter -- Wythoff subdiagrams of $D$ with $h$ nodes.

The flags in $T$, considered up to the action of $\Gamma$, are in bijection with the chains $D_1 \subset \cdots \subset D_n=D$, where $D_i$ is a Coxeter -- Wythoff subdiagram with $i$ nodes.
\end{prop}

As an instructive example, in Figure \ref{Wythoff:fig} the faces sharing the same colour in each polyhedron lie in the same $\Gamma$-orbit, and there are 1, 2, or 3 orbits depending on the number of Coxeter -- Wythoff subdiagrams with 2 nodes. 

\subsection{Regular polyhedra and tessellations}
We use Coxeter -- Wythoff diagrams to classify the regular polyhedra in $\matR^{n+1}$ and the regular tessellations of $\matR^n, \matH^n$. (A regular tessellation of $\matS^n$ is equivalent to a regular polyhedron in $\matR^{n+1}$.) 
Polyhedra and tessellations are always considered up to isometries in $\matH^k, \matS^k$ and similarities in $\matR^k$. A Coxeter diagram of type
\begin{center}
  \includegraphics[width = 3 cm]{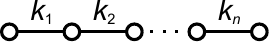}
  \vspace{-.1 cm}
\end{center}
is called \emph{linear}. The linear Coxeter diagrams are listed in Tables \ref{Coxeter-spherical:fig}, \ref{Coxeter-Euclidean:fig}, \ref{Coxeter-hyperbolic:fig}, and \ref{Coxeter-hyperbolic-ideal:fig}. 

\begin{teo} \label{linear:Coxeter:teo}
Every Coxeter -- Wythoff diagram of type
\begin{center}
  \vspace{.1 cm}
  \includegraphics[width = 3 cm]{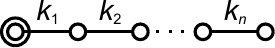}
\end{center}
determines a regular polyhedron in $\matR^{n+1}$ or a regular tessellation of $\matR^n$ or $\matH^n$.  Every regular polyhedron in $\matR^{n+1}$ or regular tessellation in $\matR^n$ or $\matH^n$ is uniquely obtained in this way.
\end{teo}

\begin{proof}
The diagram contains a unique Coxeter -- Wythoff subdiagram $D_i$ with $i$ nodes for every $1\leq i\leq n$, hence a unique sequence $D_1\subset \cdots \subset D_n$, therefore the polyhedron or tessellation has a unique flag up to the action of $\Gamma$ and is regular. 

Conversely, let $X$ be a regular tessellation or polyhedron. Since it is regular, it must be preserved by reflections along the codimension 2 faces. Therefore any simplex $P$ in the barycentric subdivision is Coxeter, described by some Coxeter diagram $D$, and $X$ is obtained from $P$ by the Wythoff construction. Therefore $X$ is obtained from some Coxeter -- Wythoff diagram $D$. Regularity implies that $D$ shouold contain a single Coxeter -- Wythoff subdiagram $D_i$ with $i$ nodes for every $i$, and this easily implies that $D$ is necessarily as shown. The uniqueness of the description follows from the discussion below.
\end{proof}

The regular polyhedron in $\matR^{n+1}$ or tessellation in $\matR^n, \matH^n$ produced by the Coxeter -- Wythoff diagram of Theorem \ref{linear:Coxeter:teo} is denoted with the \emph{Schl\"afli symbol} $\{k_1,\ldots,k_n\}$. Every $h$-face of the polyhedron or tessellation is a copy of the regular polyhedron $\{k_1,\ldots,k_{h-1}\}$, actually a hyperbolic version of it if we are in $\matH^n$.
From the classification of the linear Coxeter diagrams in Theorem \ref{irreducible:Coxeter:teo} we immediately deduce the following. 

\begin{cor} \label{regular:cor}
The regular polyhedra in $\matR^{n+1}$ are:
\begin{gather*}
\{p\}, \ \{3,3\}, \ \{3,4\}, \ \{3,5\}, \ \{4,3\}, \ \{5,3\}, \\
\{3,3,3\},\ \{3,3,4\},\ \{3,3,5\},\ \{3,4,3\},\ \{4,3,3\},\ \{5,3,3\}, \\
\{3,\ldots,3\}, \ \{4,3,\ldots,3\}, \ \{3,\ldots,3,4\}
\end{gather*}
with $p\geq 3$. The regular tessellations of $\matR^n$ are:
\begin{gather*}
\{\infty\}, \ \{3,6\}, \ \{4,4\}, \  \{6,3\}, \  \{4,3,4\}, \\  
\{3,3,4,3\}, \  \{3,4,3,3\}, \  \{4,3,3,4\}, \  \{4,3, \ldots, 3,4\}.
\end{gather*} 
The regular tessellations of $\matH^n$ with compact polyhedra are:
\begin{gather*}
\{p,q\}, \  \{3,5,3\},\ \{4,3,5\}, \ \{5,3,4\}, \ \{5,3,5\}, \\ 
\{3,3,3,5\}, \ \{4,3,3,5\}, \ \{5,3,3,3\}, \ \{5,3,3,4\}, \ \{5,3,3,5\} 
\end{gather*}
with $(p-2)(q-2)>4$. The regular tessellations of $\matH^n$ with ideal polyhedra are:
$$\{p,\infty\}, \ \{3,3,6\}, \ \{3,4,4\}, \ \{4,3,6\}, \ \{5,3,6\}, \ \{3,4,3,4\}, \ \{3,3,3,4,3\} $$
with $p\geq 3$.
\end{cor}
\begin{proof}
These arise from the linear Coxeter diagrams in Figures \ref{Coxeter-spherical:fig}, \ref{Coxeter-Euclidean:fig}, \ref{Coxeter-hyperbolic:fig}, \ref{Coxeter-hyperbolic-ideal:fig}, \ref{Coxeter-hyperbolic-ideal2:fig}. In the ideal case we should consider only Coxeter diagrams with one red extremal vertex, and this vertex should be precisely the encircled one. 
\end{proof}

The Schl\"afli symbol $\{k_1,\ldots, k_n\}$ encodes nicely the combinatorial properties of the regular polyhedron or tessellation. Its facets are copies of the regular polyhedron $\{k_1,\ldots, k_{n-1}\}$, with $k_n$ of them meeting at each codimension 3 face. If there are no ideal vertices, the inverted symbol $\{k_n, \ldots, k_1\}$ describes a combinatorially dual polyhedron or tessellation, sharing the same original Coxeter simplex, but obtained via the Wythoff construction with a different seed. We now describe all these regular objects in more detail.

\subsubsection{Regular polyhedra}
The Euclidean regular polyhedra were classified in all dimensions by Schl\"afli \cite{S}, and a standard beautiful reference is Coxeter \cite{C2}. The polyhedra $\{3,3\}$, $\{3,4\}$, $\{3,5\}$, $\{4,3\}$, $\{5,3\}$ are the regular tetrahedron, octahedron, icosahedron, cube, and dodecahedron in Figure \ref{regular:fig}. The three infinite families
$$\{3,\ldots,3\}, \ \{4,3,\ldots,3\}, \ \{3,\ldots,3,4\}$$
describe respectively the regular $n$-simplex, the $n$-cube, and the $n$-cross-polytope, that is dual to the $n$-cube, and is the convex hull of $\pm e_1, \ldots, \pm e_n$ in $\matR^n$.

\begin{figure}
 \begin{center}
  \includegraphics[width = 12.5 cm]{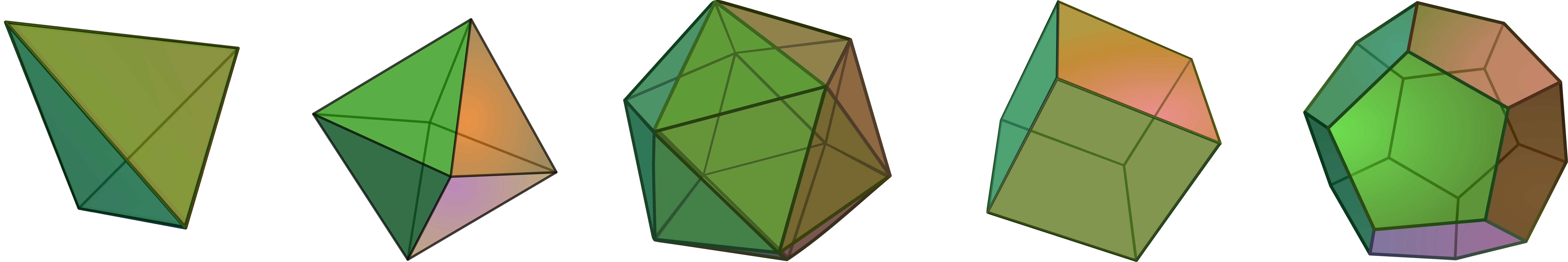}
 \end{center}
 \caption{The five regular Euclidean three-dimensional polyhedra.}  \label{regular:fig}
\end{figure}

\begin{figure}
 \begin{center}
  \includegraphics[width = 6 cm]{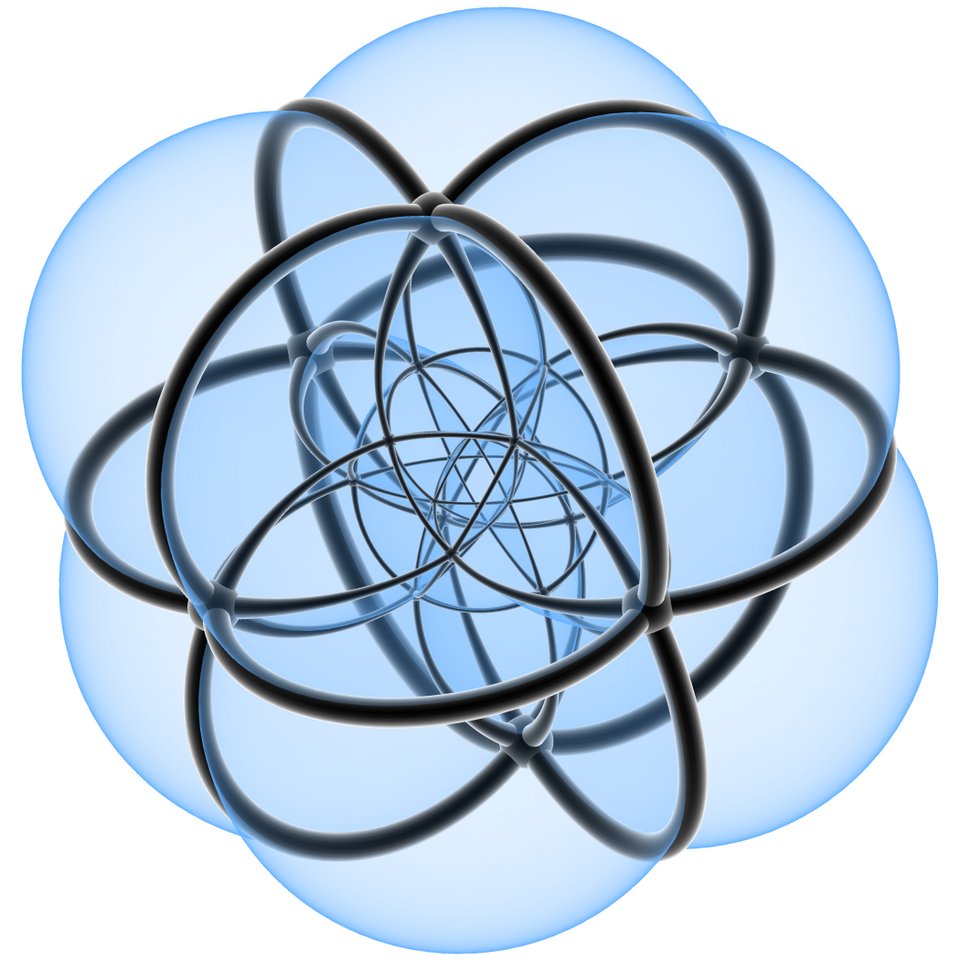}
  \includegraphics[width = 6 cm]{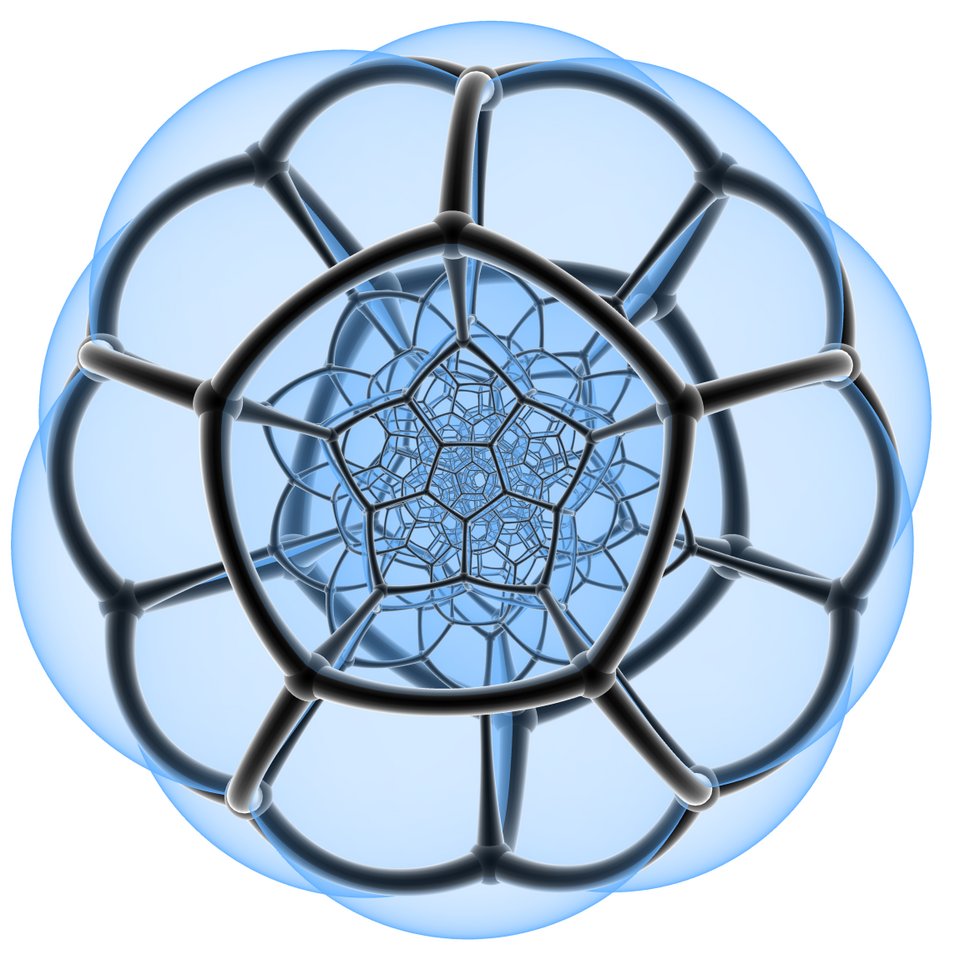}
 \end{center}
 \caption{The tessellations of $\matS^3$ into 24 octahedra (left) and 120 dodecahedra (right) given by the 24-cell and the 120-cell, respectively. The figure shows their stereographic projections in $\matR^3$, hence faces are spherical.}  \label{cells:fig}
\end{figure}

In dimension 4 there are three additional regular polyhedra
$$\{3,3,5\},\ \{3,4,3\},\ \{5,3,3\}$$
called respectively the \emph{600-cell}, the \emph{24-cell}, and the \emph{120-cell}. They can be elegantly defined using quaternions. The unit quaternions $S^3 \subset \matR^4$ contain the \emph{binary tetrahedral} and \emph{binary icosahedral subgroups} $T_{24}^* <I_{120}^*$. The 24-cell and the 600-cell are the convex hulls of these groups. The 120-cell is the dual of the 600-cell.

The 24-cell has 24 octahedral facets and 24 vertices, each with a cubic link. This is the only self-dual regular polyhedron in all dimensions different from a simplex and a polygon. The 600-cell has 600 tetrahedral facets and 120 vertices, each with an icosahedral link. Conversely, the 120-cell has 120 dodecahedral facets and 600 vertices, each with a tetrahedral link. See Figure \ref{cells:fig}.

We may wonder what are the regular polyhedra in $\matH^n$ and $\matS^n$. Given its symmetries, every regular polyhedron $P\subset \matH^n$ centered at the origin in the Klein model is also regular in the Euclidean sense. Therefore a regular polyhedron in $\matH^n$ ($\matS^n$) is combinatorially like a Euclidean one, only with smaller (larger) dihedral angles. The dihedral angles vary continuously with the size of the polyhedron.

\subsubsection{Regular tessellations}
Up to similarities, the regular tessellations of $\matR^n$ are: 
\begin{itemize}
\item The tessellation $\{\infty\}$ of $\matR$ by equal segments; 
\item The tessellations $\{3,6\}$, $\{4,4\}$, $\{6,3\}$ of $\matR^2$ by triangles, squares, hexagons; 
\item The tessellation $\{4,3,\ldots,3,4\}$ of $\matR^n$ by $n$-cubes; 
\item The dual tessellations $\{3,3,4,3\}$ and $\{3,4,3,3\}$ of $\matR^4$, made respectively by cross-polytopes and 24-cells. 
\end{itemize}

The tessellations in $\matR^2$ and $\matR^3$ are shown in Figure \ref{T0:fig}. The two additional regular four-dimensional tessellations may look unexpected: the dihedral angle of the cross-polytope and of the 24-cell in $\matR^4$ is in fact indeed $2\pi/3$, and in these tessellations there are three polyhedra around each triangular ridge. The vertex links of the two tessellations form respectively the 24-cell and the hypercube.

\begin{figure} 
 \begin{center}
  \includegraphics[width = 2.5 cm]{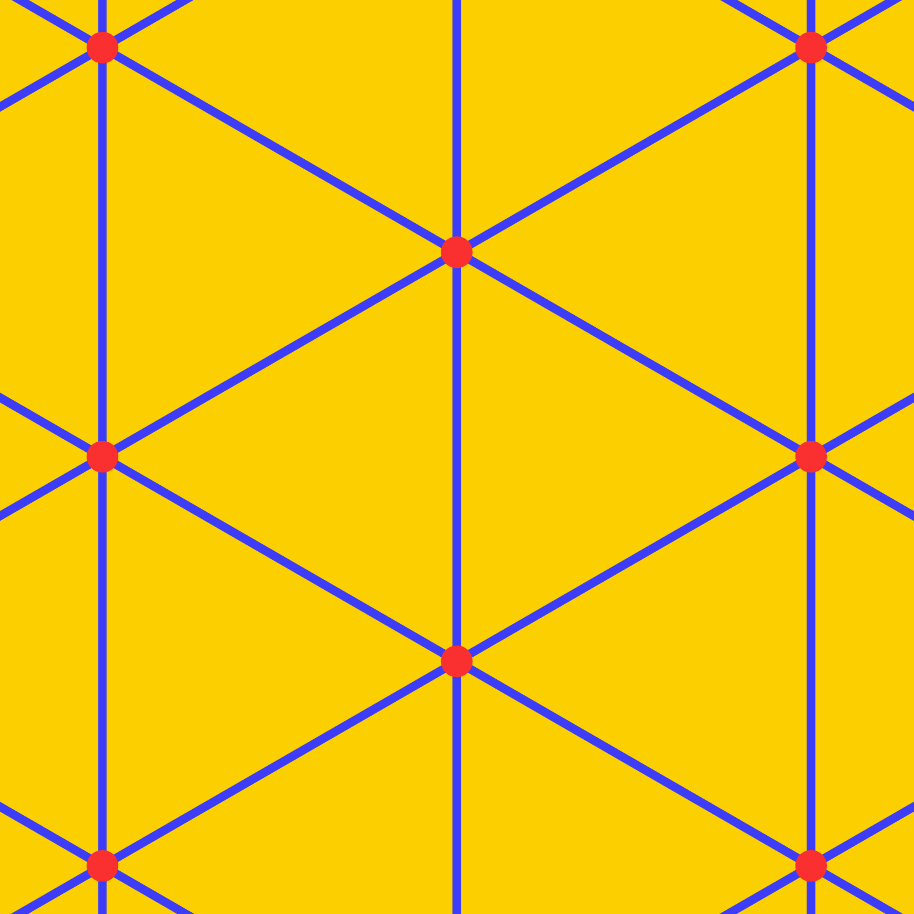}
  \includegraphics[width = 2.5 cm]{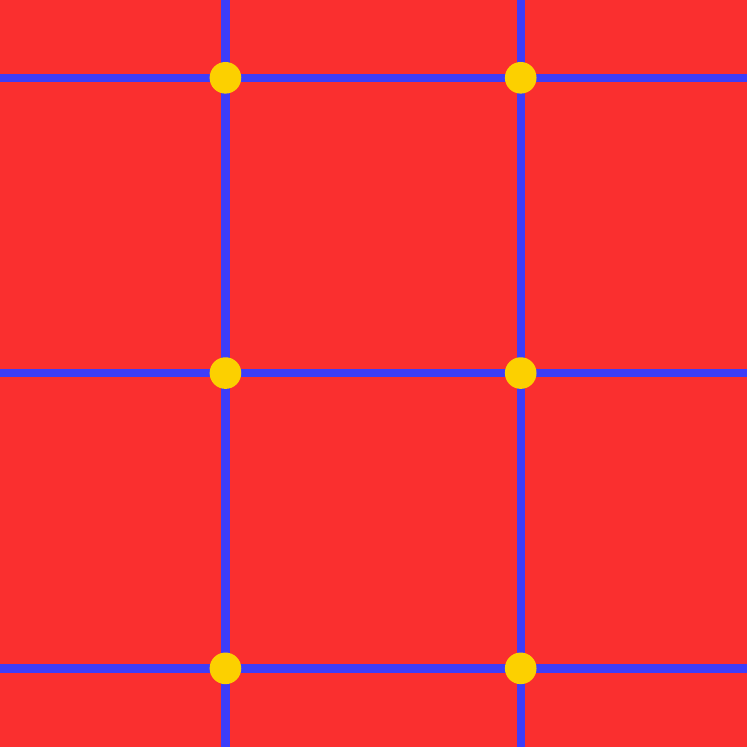}
  \includegraphics[width = 2.5 cm]{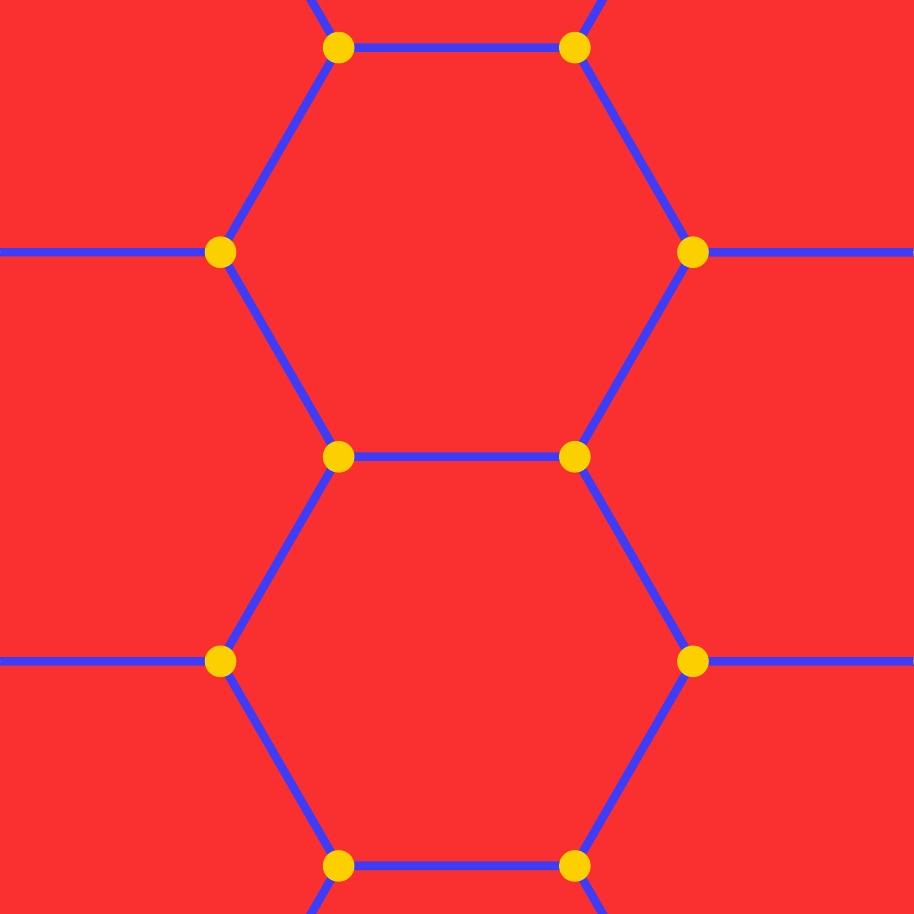}
  \includegraphics[width = 2.5 cm]{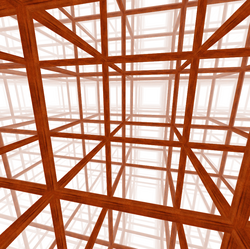}
 \end{center}
 \caption{The three regular tessellations $\{3,6\}, \{4,4\}, \{6,3\}$ of $\matR^2$ and the cubic tessellation $\{4,3,4\}$ of $\matR^3$.}  \label{T0:fig}
\end{figure}

Up to isometries, the regular tessellations of $\matH^n$ are:
\begin{itemize}
\item The tesselations $\{p,q\}$ of $\matH^2$ by regular polygons with $(p-2)(q-2)>4$;
\item The tesselations $\{p,\infty\}$ of $\matH^2$ by ideal regular polygons with $p\geq 3$;
\item The tessellations $\{3,5,3\}$, $\{4,3,5\}$, $\{5,3,4\}$, $\{5,3,5\}$,
$\{3,3,6\}$, $\{3,4,4\}$, $\{4,3,6\}$, $\{5,3,6\}$ of $\matH^3$ by all the 5 regular polyhedra;
\item The tessellations $\{3,3,3,5\}$, $\{4,3,3,5\}$, $\{5,3,3,3\}$, $\{5,3,3,4\}$, $\{5,3,3,5\}$, $\{3,4,3,4\}$ of $\matH^4$ by simplexes, hypercubes, 120-cells (3 times), and 24-cells;
\item The tessellation $\{3,3,3,4,3\}$ of $\matH^5$ made by ideal cross-polytopes.
\end{itemize}

\begin{figure} 
 \begin{center}
  \includegraphics[width = 3 cm]{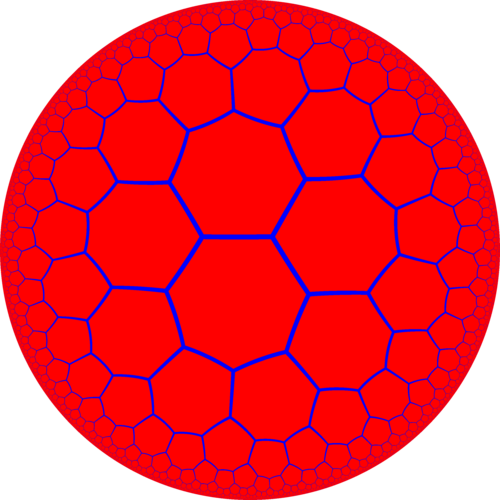}
  \includegraphics[width = 3 cm]{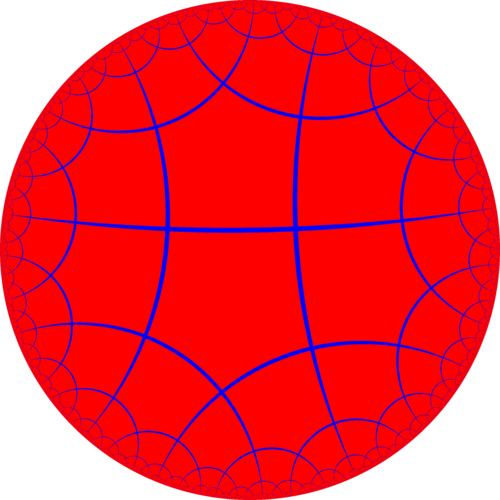}
  \includegraphics[width = 3 cm]{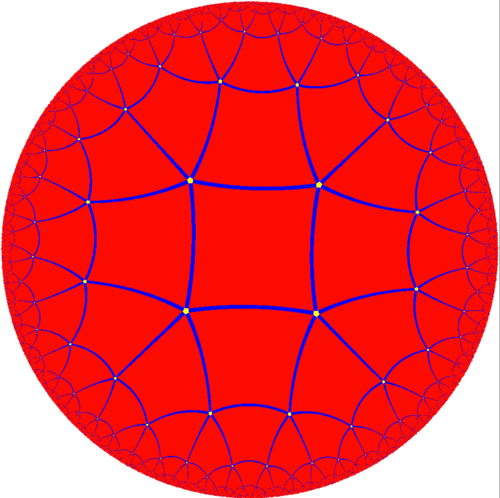}
  \includegraphics[width = 3 cm]{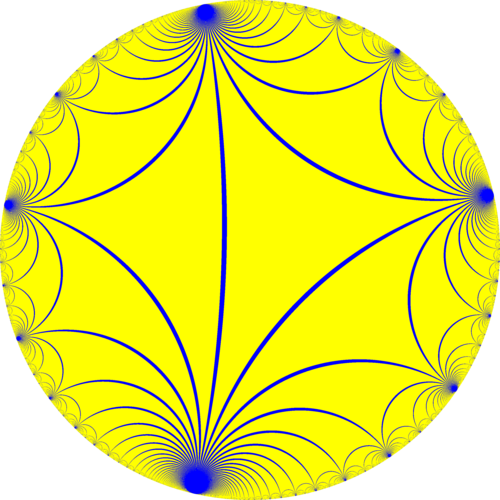}
 \end{center}
 \caption{The tessellations $\{7,3\}$, $\{5,4\}$, $\{4,5\}$, and $\{3,\infty\}$ of $\matH^2$.
 }  \label{H2:fig}
\end{figure}

\begin{figure} 
 \begin{center}
  \includegraphics[width = 5.5 cm]{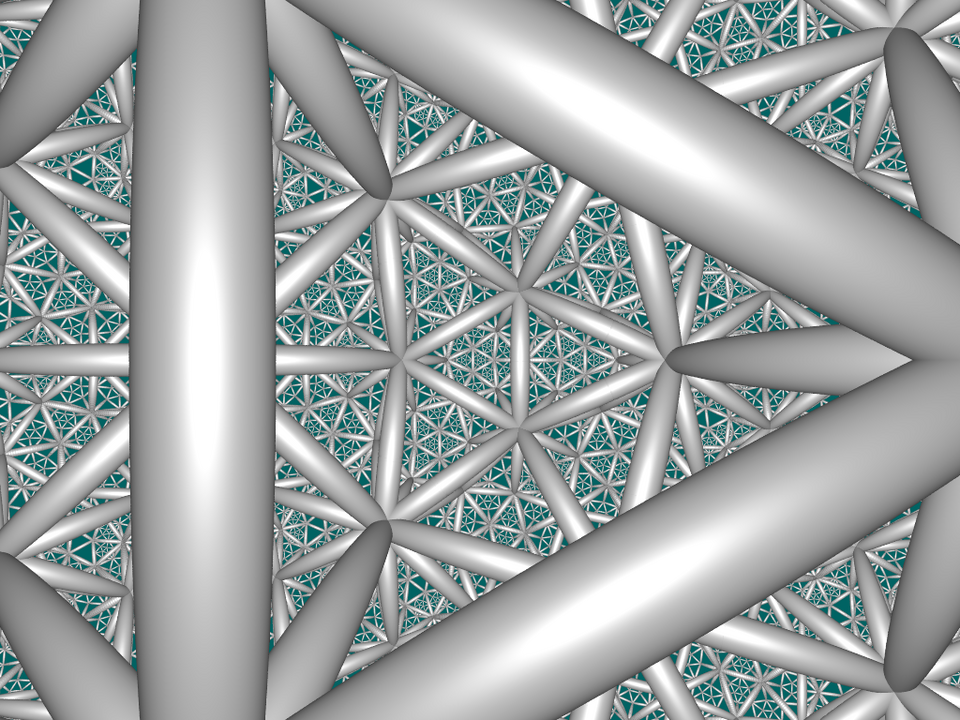} \quad
  \includegraphics[width = 5.5 cm]{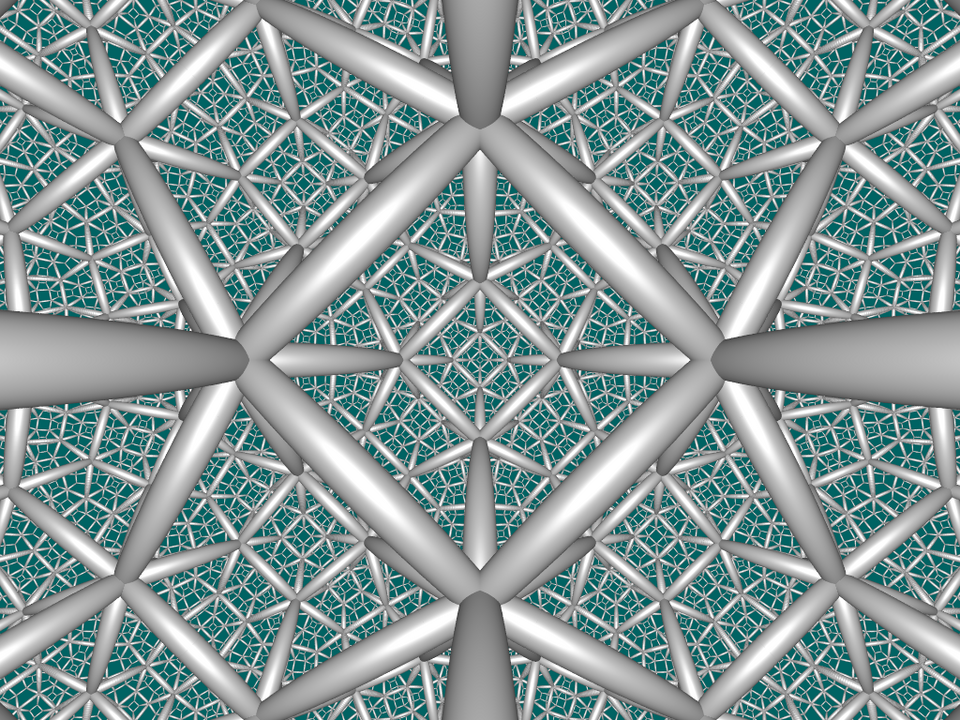}\\
  \vspace{.3 cm}
  \includegraphics[width = 5.5 cm]{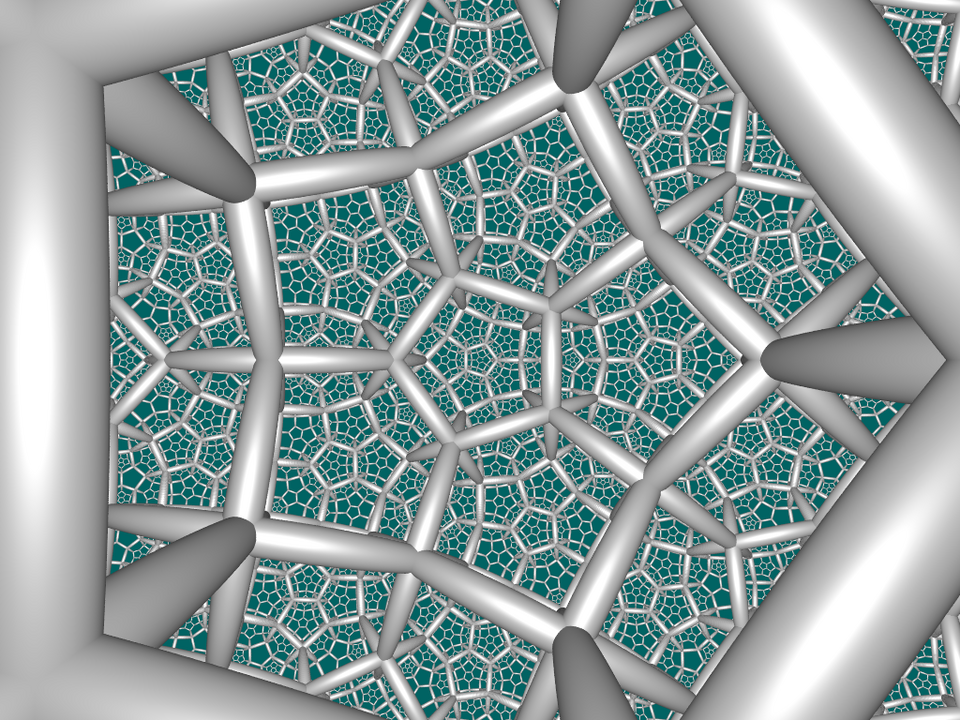} \quad
  \includegraphics[width = 5.5 cm]{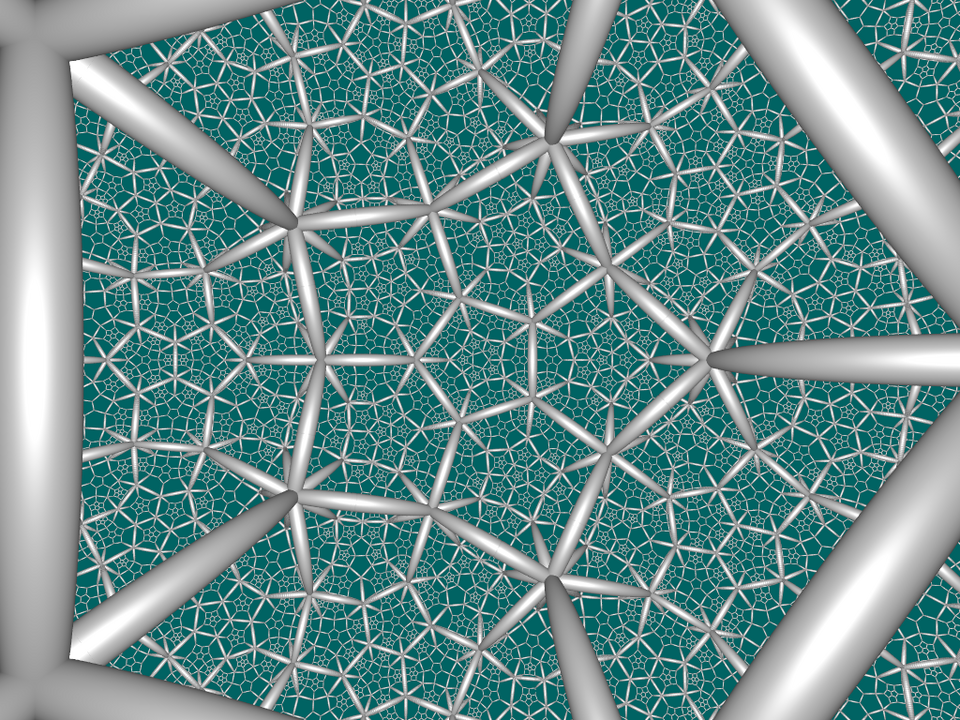} \\
  \vspace{.3 cm}
  \includegraphics[width = 5.5 cm]{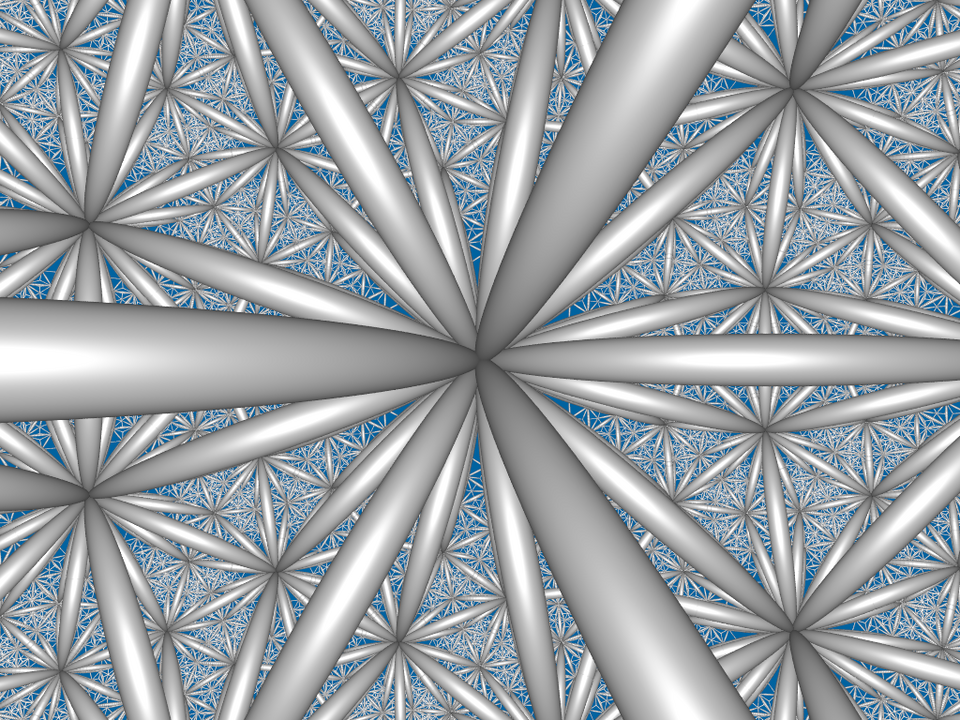} \quad
  \includegraphics[width = 5.5 cm]{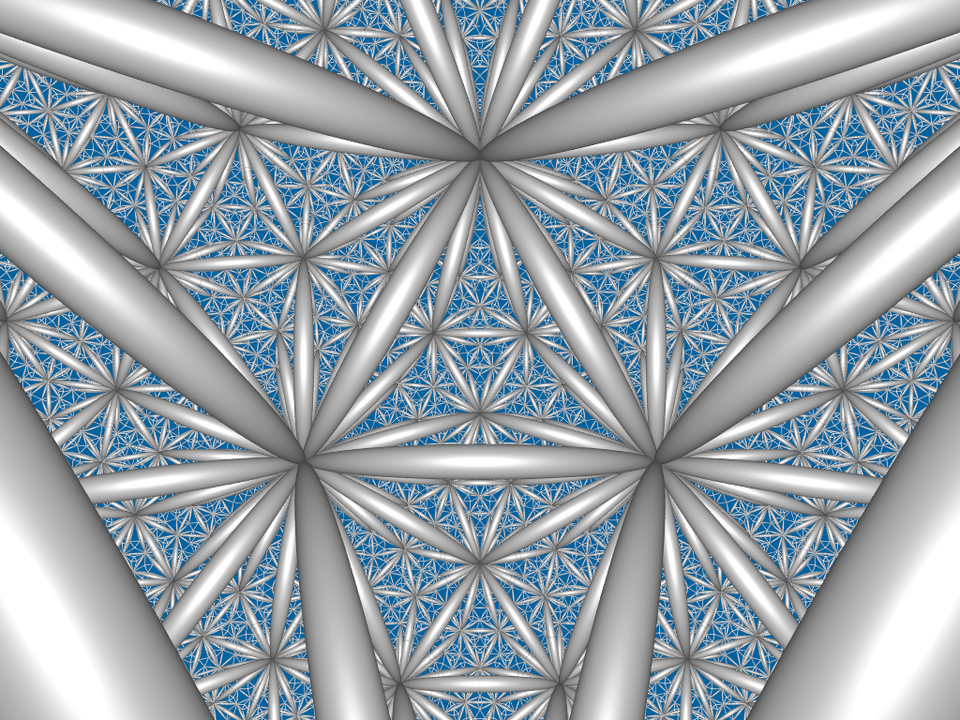}\\
  \vspace{.3 cm}
  \includegraphics[width = 5.5 cm]{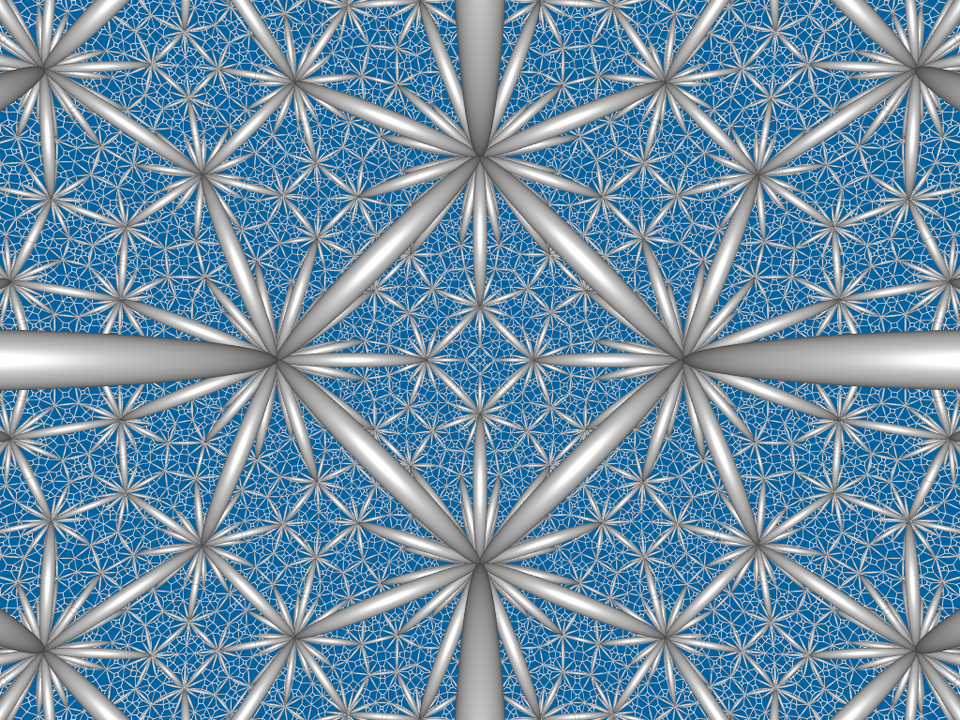} \quad
  \includegraphics[width = 5.5 cm]{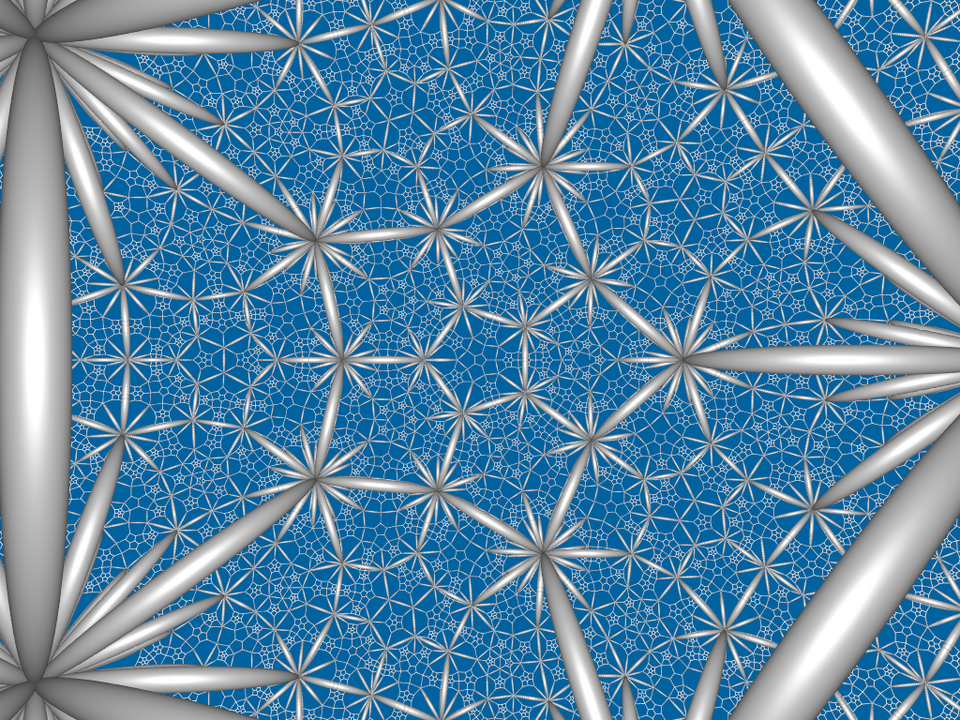} 
 \end{center}
 \caption{The tessellations $\{3,5,3\}$, $\{4,3,5\}$, $\{5,3,4\}$, $\{5,3,5\}$ of $\matH^3$ into compact icosahedra, cube, dodecahedra, dodecahedra. The tessellations 
$\{3,3,6\}$, $\{3,4,4\}$, $\{4,3,6\}$, $\{5,3,6\}$ of $\matH^3$ into ideal tetrahedra, octahedra, cube, dodecahedra.
 }  \label{H:fig}
\end{figure}

Some regular tessellations of $\matH^2$ are shown in Figure \ref{H2:fig}.
The 8 regular tessellations of $\matH^3$ are shown in Figure \ref{H:fig}. In $\matH^4$ we have one tessellation by compact simplexes with dihedral angles $2\pi/5$, one by compact hypercubes with dihedral angles $2\pi/5$, three by compact 120-cells with dihedral angles $2\pi/3, \pi/2, 2\pi/5$ respectively, and one by ideal 24-cells with dihedral angle $\pi/2$. The latter induces on every horosphere centered at some ideal point the cubic tessellation of $\matR^3$. Finally, in $\matH^5$ we have one tessellation by ideal cross-polytopes with dihedral angles $2\pi/3$, which induces on the horospheres centered at the ideal vertices the tessellation of $\matR^4$ by cross-polytopes mentioned above.

\begin{table} 
\begin{center}
\begin{tabular}{c||cccc}
\phantom{\Big|} polyhedron & $\theta = \frac \pi 3$ & $\theta = \frac{2\pi}5$ & $\theta = \frac \pi 2$ &
$\theta = \frac{2\pi}3$ \\
\hline\hline
\rule{0pt}{3ex}
\phantom{\Big|}  tetrahedron & ideal $\matH^3$ & $\matS^3$ & $\matS^3$ & $\matS^3$ \\
\phantom{\Big|}  cube & ideal $\matH^3$ & $\matH^3$ & $\matR^3$ & $\matS^3$ \\
\phantom{\Big|}  octahedron & & & ideal $\matH^3$ & $\matS^3$ \\
\phantom{\Big|} icosahedron & & & & $\matH^3$ \\
\phantom{\Big|} dodecahedron & ideal $\matH^3$ & $\matH^3$ & $\matH^3$ & $\matS^3$ \\
\hline
\phantom{\Big|}  4-simplex & & $\matH^4$ & $\matS^4$ & $\matS^4$ \\
\phantom{\Big|}  4-cube & & $\matH^4$ & $\matR^4$ & $\matS^4$ \\
\phantom{\Big|}  4-cross & & & & $\matR^4$ \\
\phantom{\Big|}  24-cell & & & ideal $\matH^4$ & $\matR^4$ \\
\phantom{\Big|}  120-cell & & $\matH^4$ & $\matH^4$ & $\matH^4$ \\
\hline
\phantom{\Big|}  5-cross & & & & ideal $\matH^5$ \\
\hline
\phantom{\Big|}  $n$-simplex & & & $\matS^n$ & $\matS^n$ \\
\phantom{\Big|}  $n$-cube & &  & $\matR^n$ & $\matS^n$ 
\end{tabular}
\end{center}
\vspace{.3 cm}
\caption{A complete list of all the regular polyhedra in $\matX^n$ with dihedral angle $\theta=2\pi/k$ for $n\geq 3$. Each such polyhedron is a facet in a regular tessellation of $\matX^n$.}
\label{nice:regular:table}
\end{table}

Table \ref{nice:regular:table} summarises the occurrence of each regular polyhedron as a facet in a regular tessellation, with its dihedral angles. Regular tessellations in $\matS^n$ can be interpreted as regular polyhedra in $\matR^{n+1}$.

\subsection{Some exercises}

\begin{ex} \label{linear3:ex}
The Coxeter -- Wythoff linear diagrams
 \begin{center}
 \vspace{.1 cm}
  \includegraphics[width = 8 cm]{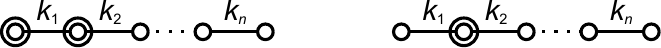}
 \end{center}
represent respectively the \emph{truncation} and the \emph{rectification} of the regular polyhedron or tessellation $\{k_1,\ldots,k_n\}$. On a polyhedron, both operations consist in cutting off appropriate isometric open star neighbourhoods of the vertices so that the resulting polyhedron has all edges of the same length: the star neighbourhoods are disjoint in a truncation and intersect in points in a rectification, see Figure \ref{Wythoff:fig}. The rectification can also be defined as the convex hull of the midpoints of the edges. On a tessellation, these star neighbourhoods are not removed and yield new facets.
\end{ex}

\begin{ex}
Consider the Coxeter -- Wythoff linear diagram with $n$ nodes
 \begin{center}
 \vspace{.1 cm}
  \includegraphics[width = 2.7 cm]{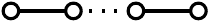}
 \end{center}
some of them being encircled. The \emph{seed vector} $v = (a_1,\ldots, a_{n+1}) \in \matR^{n+1}$ is defined by setting $a_1=0$, and recursively $a_{i+1}$ equals $a_i+1$ if the $i$-th node is encircled, and $a_i$ if it is not. The polyhedron determined by the Coxeter -- Wythoff diagram is the convex hull of the vertices obtained by permuting the coordinates of $v$. It is a polyhedron in some hyperplane $x_1+\cdots + x_{n+1} = C$. In particular the diagram
 \begin{center}
 \vspace{.1 cm}
  \includegraphics[width = 3 cm]{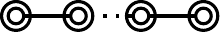}
 \end{center}
represents the $n$-\emph{permutohedron}, the convex hull of all permutations of $(0,1,\ldots,n)$. See some examples in Figure \ref{permutohedra:fig}.

The Coxeter -- Wythoff linear diagram with $n$ nodes
 \begin{center}
\labellist
\small\hair 2pt
\pinlabel $4$ at 20 15
\endlabellist
 \vspace{.1 cm}
  \includegraphics[width = 2.7 cm]{linear2}
 \end{center}
(some of which are encircled) is similar: the \emph{seed vector} $w=(c_1,\ldots, c_n) \in \matR^n$ is defined by setting $c_1=1$ if the first node is encircled, and $c_1=0$ otherwise; then $c_{i+1} = c_i + \sqrt 2$ if the $i$-th node is encircled, and $c_{i+1}=c_i$ otherwise for $i\geq 2$.
The polyhedron determined by the Coxeter -- Wythoff diagram
is the convex hull of the vertices obtained by permuting the coordinates of $w$ and changing their signs. If all the nodes are encircled we get an \emph{omnitruncated $n$-cube} as in Figure \ref{Truncated_cuboctahedron:fig}.
\end{ex}

\begin{figure}
 \begin{center}
  \includegraphics[width = 12.5 cm]{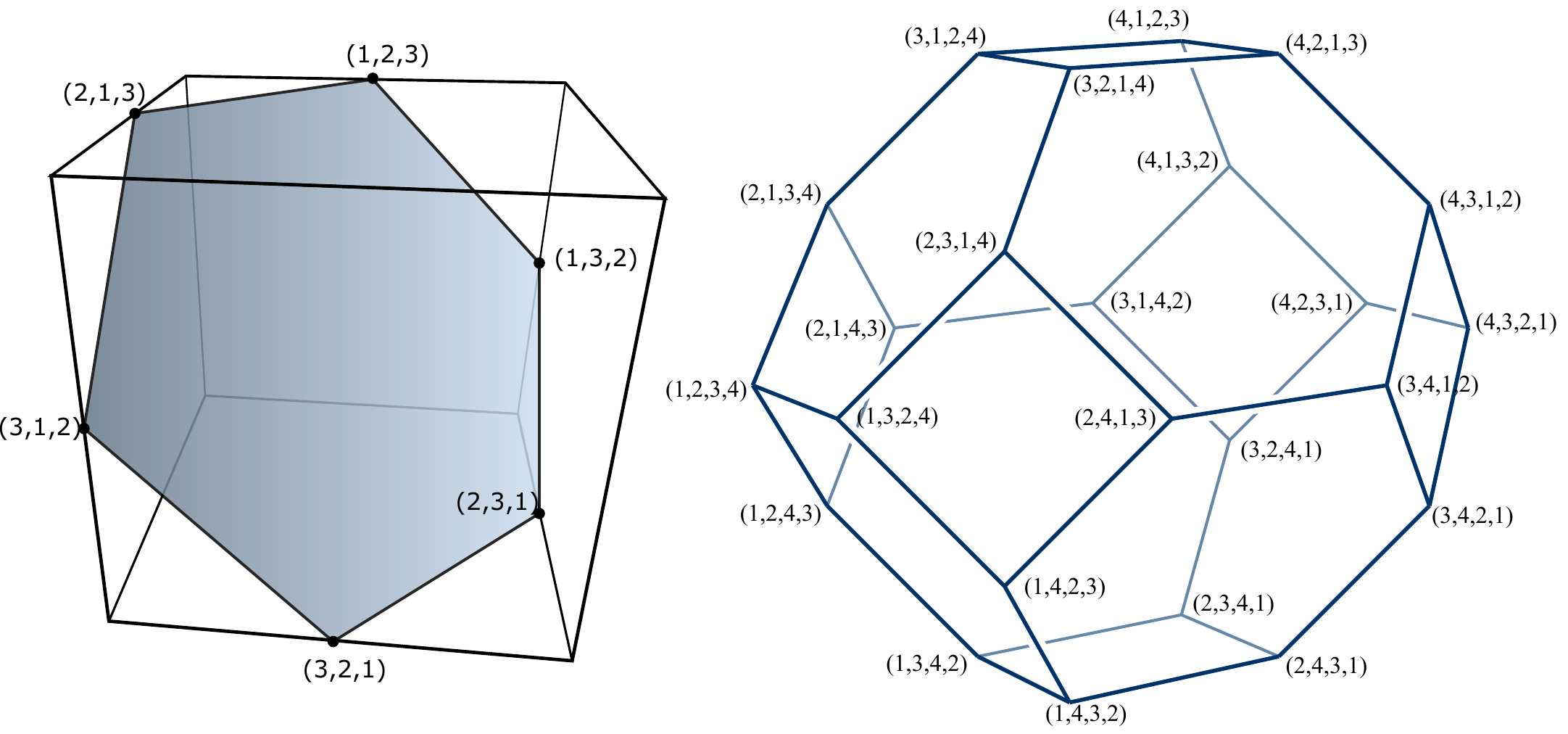}
 \end{center}
 \caption{The permutohedra in dimension $n=2,3$.}  \label{permutohedra:fig}
\end{figure}

\begin{figure}
 \begin{center}
 \labellist
\small\hair 2pt
\pinlabel $(-1-\sqrt 2,-1)$ at 20 30
\pinlabel $(-1-\sqrt 2,1)$ at 18 60
\pinlabel $(-1,1+\sqrt 2)$ at 25 95
\pinlabel $(1,1+\sqrt 2)$ at 65 95
\pinlabel $(1+\sqrt 2,-1)$ at 72 30
\pinlabel $(1+\sqrt 2,1)$ at 72 60
\pinlabel $(-1,-1-\sqrt 2)$ at 25 -5
\pinlabel $(1,-1-\sqrt 2)$ at 65 -5
\endlabellist
  \vspace{.2 cm}
  \includegraphics[width = 5.5 cm]{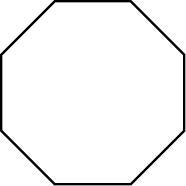}
  \hspace {.5 cm}
  \includegraphics[width = 6 cm]{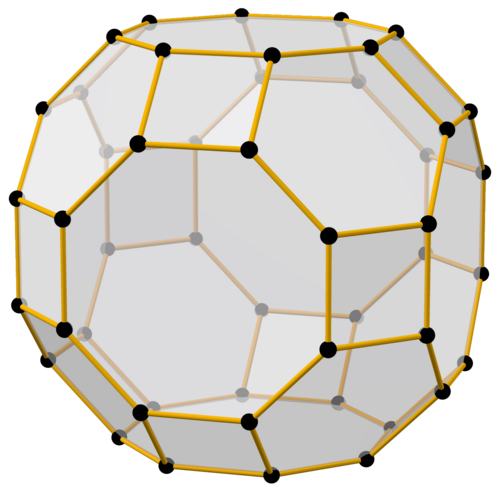}
  \vspace{.2 cm}
 \end{center}
 \caption{The omnitruncated cubes in dimension $n=2,3$.}  \label{Truncated_cuboctahedron:fig}
\end{figure}

\begin{ex} \label{demicube:ex}
The Coxeter -- Wythoff diagram with $n$ nodes
 \begin{center}
  \includegraphics[width = 3.3 cm]{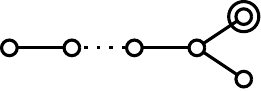}
 \end{center}
represents a $n$-\emph{demicube}, that is the convex hull of the vertices of a cube $[0,1]^n$ whose entries sum to an even number. When $n=3$ or $4$ this is a tetrahedron or cross-polytope. When $n=5$ the facets are tetrahedra and cross-polytopes, so the 5-demicube is semiregular. In general the facets are simplexes and $(n-1)$-demicubes, so the $n$-demicube is uniform but not semiregular when $n\geq 6$.
\end{ex}

\begin{ex} \label{diagonal_tess:ex}
The Euclidean Coxeter -- Wythoff circular diagram with $n$ nodes
 \begin{center}
  \includegraphics[width = 2 cm]{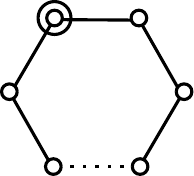}
 \end{center}
one of which is encircled, represents the Euclidean tessellation of $\matR^{n-1}$ obtained by representing $\matR^{n-1}$ as the diagonal hyperplane $H = \{x_1+\cdots +x_{n} = 0\} \subset \matR^{n}$ and intersecting it with the standard cubic tessellation of $\matR^{n}$ with vertices in $\matZ^n$. If $n=2,3,4$ the facets are respectively regular triangles, regular tetrahedra and octahedra, and regular 4-simplexes and rectified 4-simplexes. The tessellation is regular for $n=2$, semiregular for $n=3$, uniform but not semiregular for $n\geq 4$.
\end{ex}

\begin{ex} \label{move:ex}
We define a move of Coxeter -- Wythoff diagrams:
 \begin{center}
\labellist
\small\hair 2pt
\pinlabel $1$ at 0 18
\pinlabel $2$ at 30 18
\pinlabel $n-1$ at 60 18
\pinlabel $n$ at 90 18
\pinlabel $n+1$ at 120 55
\pinlabel $n+2$ at 120 3
\pinlabel $1$ at 183 18
\pinlabel $2$ at 213 18
\pinlabel $n-1$ at 243 18
\pinlabel $n$ at 273 18
\pinlabel $n+1$ at 303 18
\pinlabel $n+2$ at 343 18
\endlabellist
  \includegraphics[width = 10 cm]{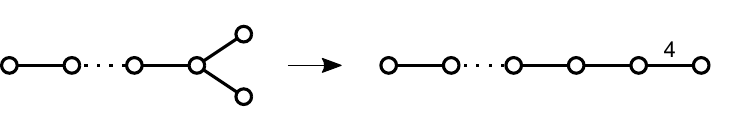}
 \end{center}
We encircle some of the $n+2$ nodes in the left diagram arbitrarily, with the only requirement that the node $n+1$ is encircled $\Longleftrightarrow$ the node $n+2$ is. We then encircle the node $i\neq n+2$ in the right diagram if and only if the corresponding node $i$ in the left is encircled, and we do not encircle $n+2$. 
This move does not modify the resulting uniform polyhedron. This holds for every $n\geq 1$, and the case $n=1$ can be deduced from Figure \ref{Wythoff:fig}.
\end{ex} 
\begin{proof}[Hint] The symmetric Coxeter simplex described by the left diagram decomposes into two smaller simplexes described by the right diagram.
\end{proof}

\subsection{Semiregular polyhedra} \label{semiregular:subsection}
We now classify all the semiregular polyhedra.

\subsubsection{Wythoffian}
We say that a uniform polyhedron or tessellation is \emph{Wythoffian} if it may be produced from a Wythoffian construction.

\begin{figure}
 \begin{center}
  \includegraphics[width = 12 cm]{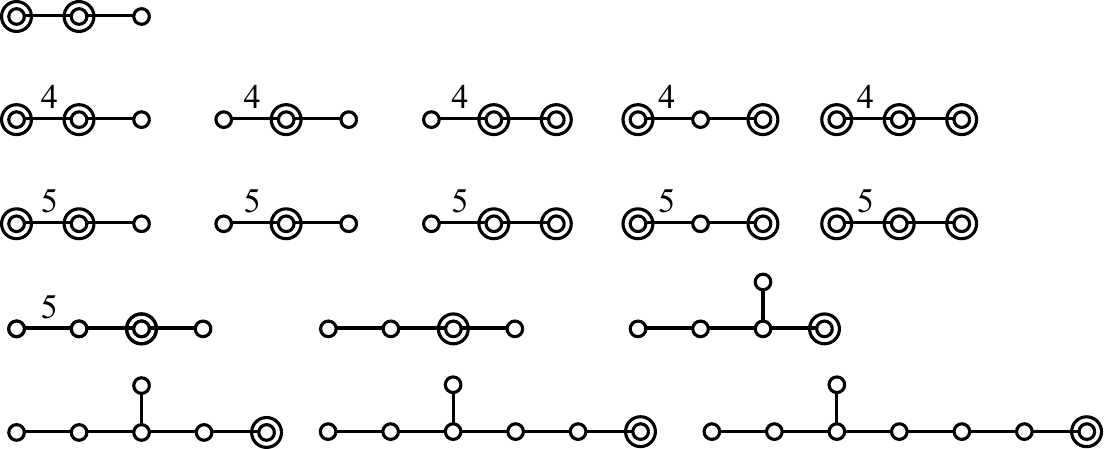}
 \end{center}
 \caption{Coxeter -- Wythoff diagrams that produce semiregular polyhedra that are not regular.}  \label{Gosset:fig}
\end{figure}

\begin{prop}
The Wythoffian semiregular polyhedra in $\matR^n$ that are not regular are precisely those produced by the Coxeter -- Wythoff diagrams in Figure \ref{Gosset:fig}.
\end{prop}
\begin{proof}
All the spherical diagrams with 3 nodes were analyzed in Figure \ref{Wythoff:fig}, and after excluding duplicates and regular polyhedra we get 11 types as in Figure \ref{Gosset:fig}.

The remaining 6 diagrams with $>3$ nodes in Figure \ref{Gosset:fig} indeed produce semiregular polyhedra: the one with 5 nodes is a 5-demicube by Exercise \ref{demicube:ex}, and by analyzing their Coxeter -- Wythoff subgraphs we discover that the facets of the other 5 polyhedra are cross-polytopes and simplexes (we use Exercise \ref{move:ex}). Therefore they are also semiregular and not regular. Finally, by examining all the spherical diagrams in Figure \ref{Coxeter-spherical:fig} with $\geq 4$ nodes one checks that only those in Figure \ref{Gosset:fig} have only regular facets and are not themselves regular.
\end{proof}

We now list all the semiregular polyhedra, distinguishing from the dimension $n=3$ where semiregular is equivalent to uniform, and $n\geq 4$ where the semiregular condition is much stronger.

\subsubsection{Dimension 3}
The complete classification of semiregular polyhedra in dimension $n=3$ was apparently known to Archimedes, see Walsh for a proof \cite{W}:

\begin{figure}
 \begin{center}
  \includegraphics[width = 12.5 cm]{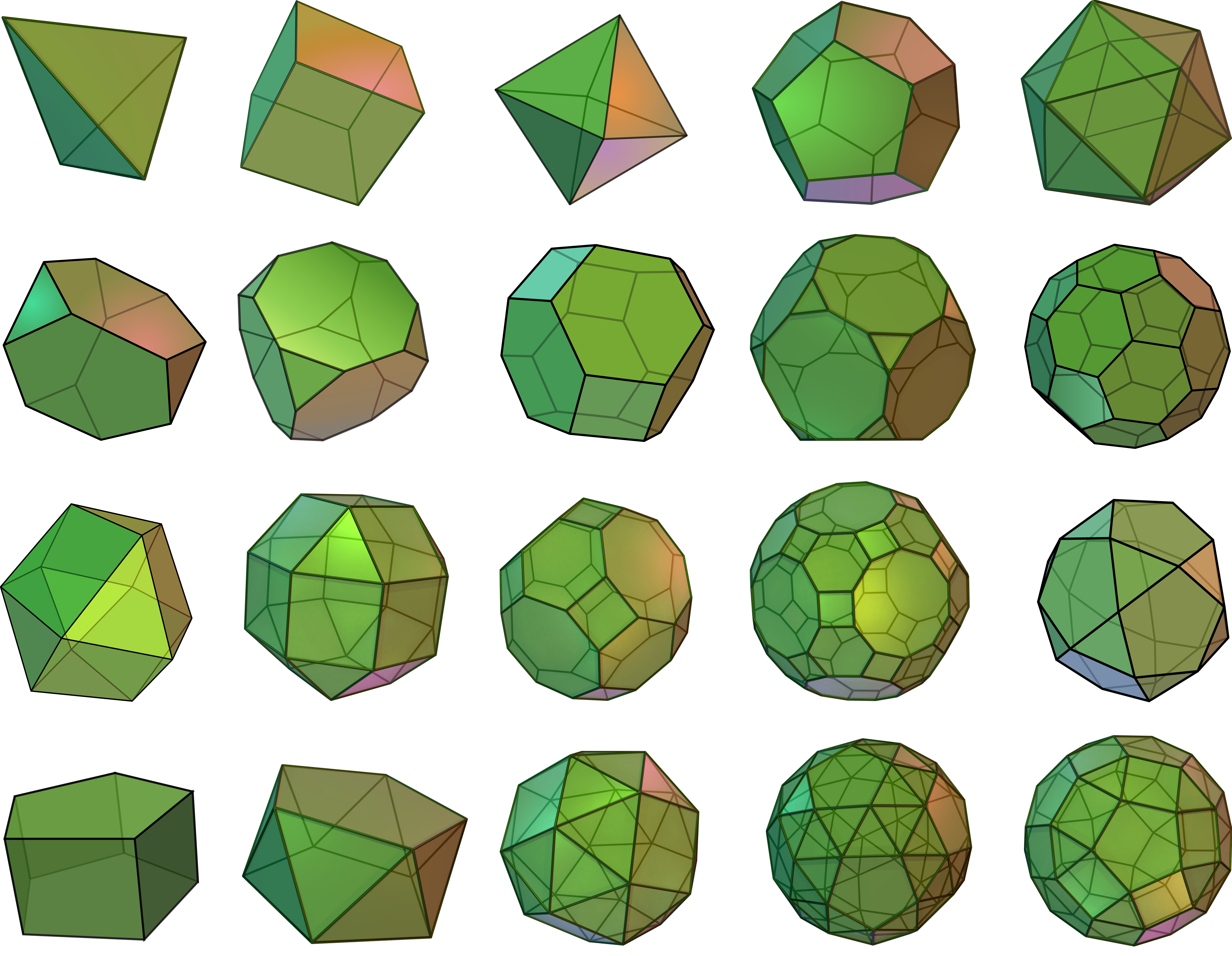}
 \end{center}
 \caption{The semiregular 3-dimensional Euclidean polyhedra. These are the 5 regular solids (first row), the 13 \emph{Archimedean polyhedra} (second and third row, plus the last three of the fourth row), and two infinite families of prisms and antiprisms, each with two $n$-gon bases (the first two of the last row, here drawn with $n=5$). }  \label{archimedean:fig}
\end{figure}

\begin{teo}
The semiregular polyhedra in $\matR^3$ are:
\begin{itemize}
\item The 5 regular polyhedra;
\item The 13 \emph{Archimedean} polyhedra;
\item The two infinite families of prisms and antiprisms.
\end{itemize}
See Figure \ref{archimedean:fig}. 
\end{teo}

Among the 13 Archimedean polyhedra, 11 are Wythoffian and arise from the 11 diagrams in the first three lines of Figure \ref{Gosset:fig} as shown in Figure \ref{Wythoff:fig}, and two are not Wythoffian (the third and fourth in the bottom row of Figure \ref{archimedean:fig}).

\subsubsection{Dimension $\geq 4$}
The list of all the semiregular polyhedra that are not regular in dimension $n\geq 4$ is quite short. It was discovered by Gosset \cite{G} in 1899, and proved to be complete by Blind -- Blind \cite{BB} almost a century later in 1991.

\begin{table} 
\begin{center}
\begin{tabular}{c||cccc}
\phantom{\Big|} dim & polyhedron & facets & vertices & link\\
\hline\hline
\rule{0pt}{3ex}
\phantom{\Big|} $4$ & rectified 4-simplex & 5 octa, 5 tetra & 10 & 3-prism \\
\phantom{\Big|} $4$ & rectified 600-cell & 600 octa, 120 icosa & 720 & 5-prism \\
\phantom{\Big|} $4$ & snub 24-cell & 120 tetra, 24 icosa & 96 & trid icos\\
\hline
\phantom{\Big|} $5$ & 5-demicube & 10 cross, 16 simpl & 16 & rect 4-simpl\\
\hline
\phantom{\Big|} $6$ & $2_{21}$ & 27 cross, 72 simpl & 27 & 5-demicube\\
\hline
\phantom{\Big|} $7$ & $3_{21}$ & 126 cross, 576 simpl & 56 & $2_{21}$ \\
\hline
\phantom{\Big|} $8$ & $4_{21}$ & 2160 cross, 17280 simpl & 240 & $3_{21}$
\end{tabular}
\end{center}
\vspace{.3 cm}
\caption{The semiregular (not regular) polyhedra of dimension $n\geq 4$. The last column shows the link of the vertices. That of the snub 24-cell is a polyhedron called \emph{tridiminished icosahedron}.}
\label{semiregular:table}
\end{table}

\begin{teo}
There are 7 semiregular and not regular polyhedra in $\matR^n$ with $n\geq 4$, listed in Table \ref{semiregular:table}. Only the snub 24-cell is not Wythoffian.
\end{teo}

The 6 Wythoffian polyhedra in Table \ref{semiregular:table} correspond to the 6 diagrams in the last two lines of Figure \ref{Gosset:fig}, see also Exercise \ref{linear3:ex}. The \emph{snub 24-cell} is the convex hull of the $120-24 = 96$ points in $I_{120}^* \setminus T_{24}^*$. Its 96 vertices are the vertices of the 600-cell minus those of the 24-cell. 

\subsubsection{Gosset polyhedra}
The \emph{Gosset polyhedra} $2_{21}, 3_{21}, 4_{21}$ are the semiregular polyhedra in Table  \ref{semiregular:table} constructed from the last three diagrams in Figure \ref{Gosset:fig}. We now describe them explicitly, 
starting from the remarkable 8-dimensional $4_{21}$. 

We start by equipping $\matZ^8$ with the famous even unimodular positive-definite bilinear form determined by the matrix
$$
E_8=
{\footnotesize
\begin{pmatrix}
 2 & -1 &  0 &  0 &  0 &  0 &  0 &  0 \\
-1 &  2 & -1 &  0 &  0 &  0 &  0 &  0 \\
 0 & -1 &  2 & -1 &  0 &  0 &  0 &  0 \\
 0 &  0 & -1 &  2 & -1 &  0 &  0 &  0 \\
 0 &  0 &  0 & -1 &  2 & -1 &  0 & -1 \\
 0 &  0 &  0 &  0 & -1 &  2 & -1 &  0 \\
 0 &  0 &  0 &  0 &  0 & -1 &  2 &  0 \\
 0 &  0 &  0 &  0 & -1 &  0 &  0 &  2
\end{pmatrix}.
}
$$
Even unimodular positive definite bilinear forms on $\matZ^n$ exist only when $n$ is divisible by 8, and in dimension 8 this is the only one up to isomorphism \cite{MH}. 

It is convenient to embed isometrically $(\matZ^8, E_8)$ in $\matR^8$ equipped with its standard scalar product. 
By linear algebra there is a basis $v_1,\ldots, v_8$ of $\matR^8$ whose Gram matrix is $E_8$, for instance we may take the vectors
$$e_1-e_2, \quad e_2-e_3,\quad e_3-e_4,\quad e_4-e_5, \quad e_5-e_6, \quad e_6+e_7, \quad -\frac 12 \sum_{i=1}^8 e_i, \quad e_6-e_7.
$$
The \emph{$E_8$ lattice} is the lattice $\Lambda < \matR^8$ generated by these vectors $v_1,\ldots, v_8$. It consists of the elements $(x_1,\ldots,x_8)\in \matZ^8 \cup (\matZ + \tfrac 12)^8$ with even coordinate sum. The smallest non-zero elements in $\Lambda$ have norm $\sqrt 2$ and they are 240 in number. 

\begin{prop}
The polyhedron $4_{21}$ is the convex hull of these 240 vectors.
\end{prop}
\begin{proof}
The Gram matrix of the Coxeter polyhedron $P$ defined by the last diagram in Figure \ref{Gosset:fig} is $\tfrac 12E_8$. Therefore $P = \{\langle x,v_i \rangle \leq 0\} \cap \matS^7$. The seed is the vertex $v=\frac{\sqrt 2}2(e_8-e_1)$ of $P$ opposite to the facet that corresponds to $v_1$. The polyhedron $4_{21}$ is the convex hull of the translates of $v$ under the reflection group $\Gamma$ of $P$. We pick the renormalized $v'= e_8-e_1$ instead of $v$, that is an element in $\Lambda$ with smallest norm. The group $\Gamma$ preserves $\Lambda$ since a reflection along the $i$-th face is written as
$$x \longmapsto x-2\frac{\langle x,v_i \rangle}{\langle v_i, v_i \rangle}v_i = x- \langle x,v_i \rangle v_i$$
and $\langle x,v_i \rangle \in \matZ$. Therefore the orbit of $v'$ is contained in the set of 240 elements with smallest norm, and one can verify that it consists of that set.
\end{proof}

Having a concrete representation for $4_{21}$, we deduce one for $3_{21}$ and $2_{21}$ as the links of the vertices of $4_{21}$ and $3_{21}$, iteratively. 

\subsection{Semiregular tessellations} We now turn to semiregular tessellations.

\begin{figure}
 \begin{center}
  \includegraphics[width = 11 cm]{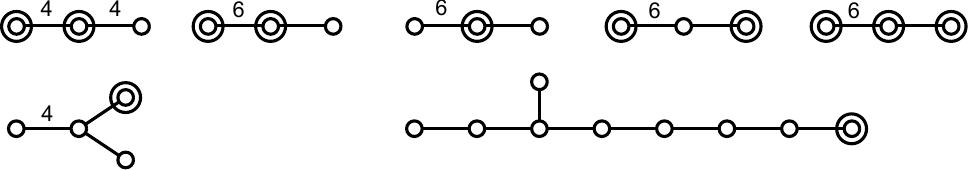}
 \end{center}
 \caption{Coxeter -- Wythoff diagrams that produce semiregular tessellations in $\matR^n$ that are not regular.}  \label{semireg_tess_Rn:fig}
\end{figure}

\begin{figure}
 \begin{center}
  \includegraphics[width = 11 cm]{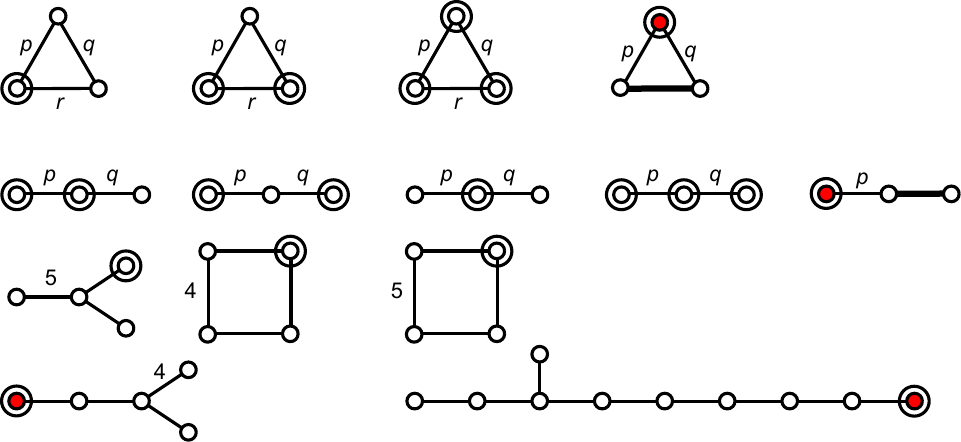}
 \end{center}
 \caption{Coxeter -- Wythoff diagrams that produce semiregular tessellations in $\matH^n$ that are not regular. Some planar tessellations can be reproduced via different diagrams.
 }  \label{semireg_tess_Hn:fig}
\end{figure}

\subsubsection{Wythoffian} As with polyhedra, we first classify the Wythoffian tessellations.

\begin{prop}
The Wythoffian semiregular tessellations in $\matR^n, \matH^n$ that are not regular are those produced by the Coxeter -- Wythoff diagrams in Figures \ref{semireg_tess_Rn:fig} and \ref{semireg_tess_Hn:fig}.
\end{prop}
\begin{proof}
By examining all the Euclidean and hyperbolic diagrams one checks that only those in the figure have only regular facets and are not themselves regular. Different diagrams that give rise to the same tessellation in dimension $n\geq 3$ have been cited only once. 
\end{proof}

We now distinguish between dimension $n=2$ where semiregular is equivalent to uniform, and $n\geq 3$ where the semiregular condition is more restrictive.

\subsubsection{Dimension 2}
The semiregular tessellations of $\matR^2$ are probably known since long. A proof of the following is in Gr\"unbaum -- Shephard \cite[Section 2.1]{Gru}.

\begin{figure} 
 \begin{center}
  \includegraphics[width = 2 cm]{T1}
  \includegraphics[width = 2 cm]{T2}
  \includegraphics[width = 2 cm]{T3} \\
  \vspace{.1 cm}
  \includegraphics[width = 2 cm]{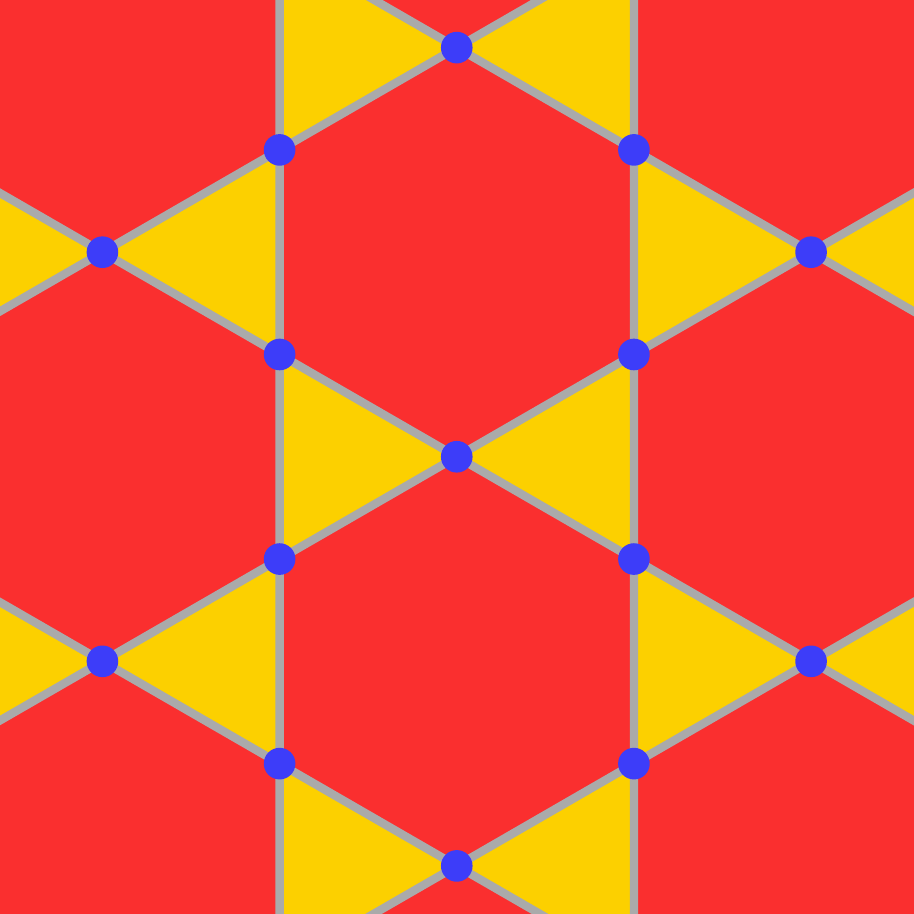}
  \includegraphics[width = 2 cm]{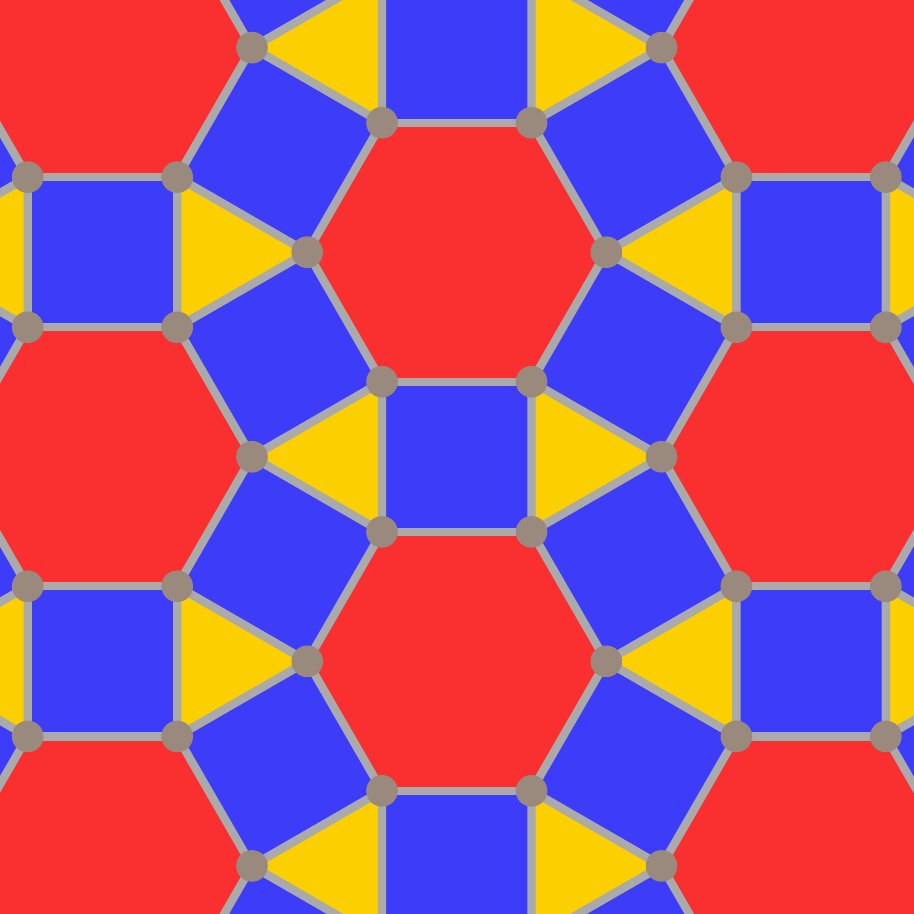}
  \includegraphics[width = 2 cm]{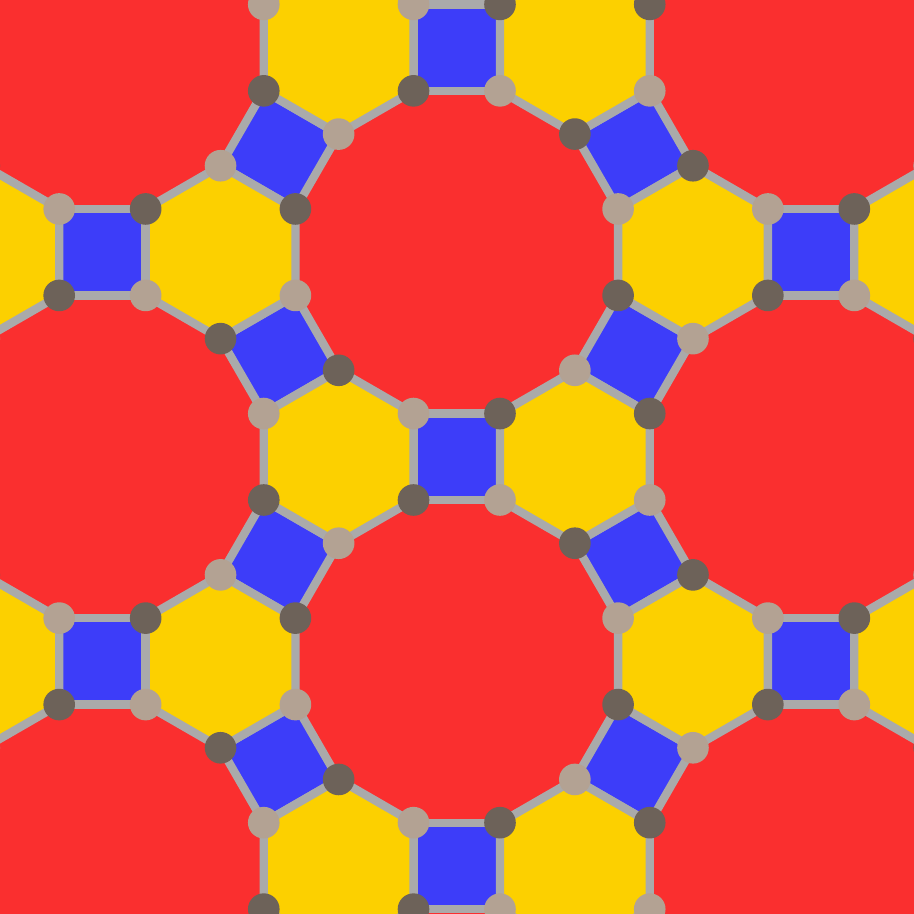} 
  \includegraphics[width = 2 cm]{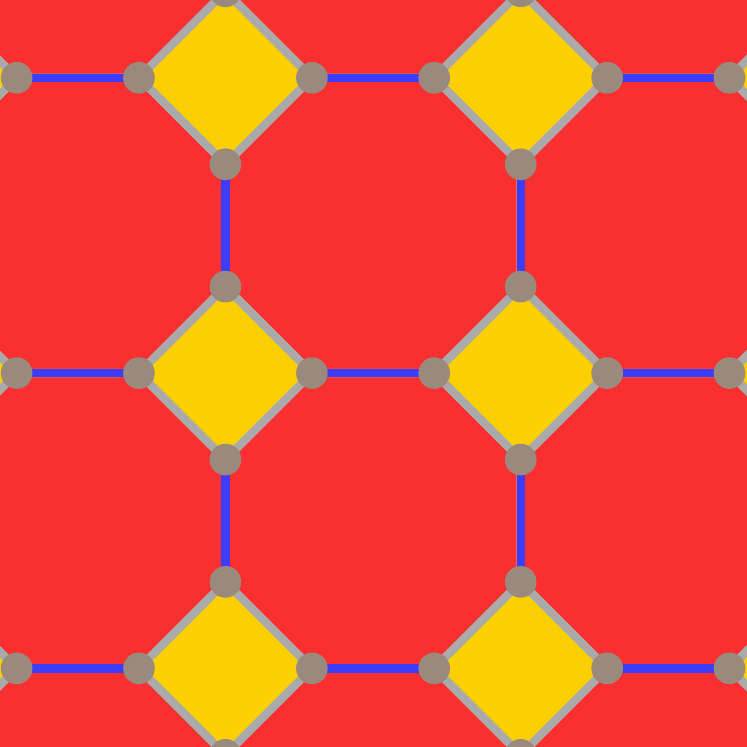} \\
  \vspace{.1 cm}
  \includegraphics[width = 2 cm]{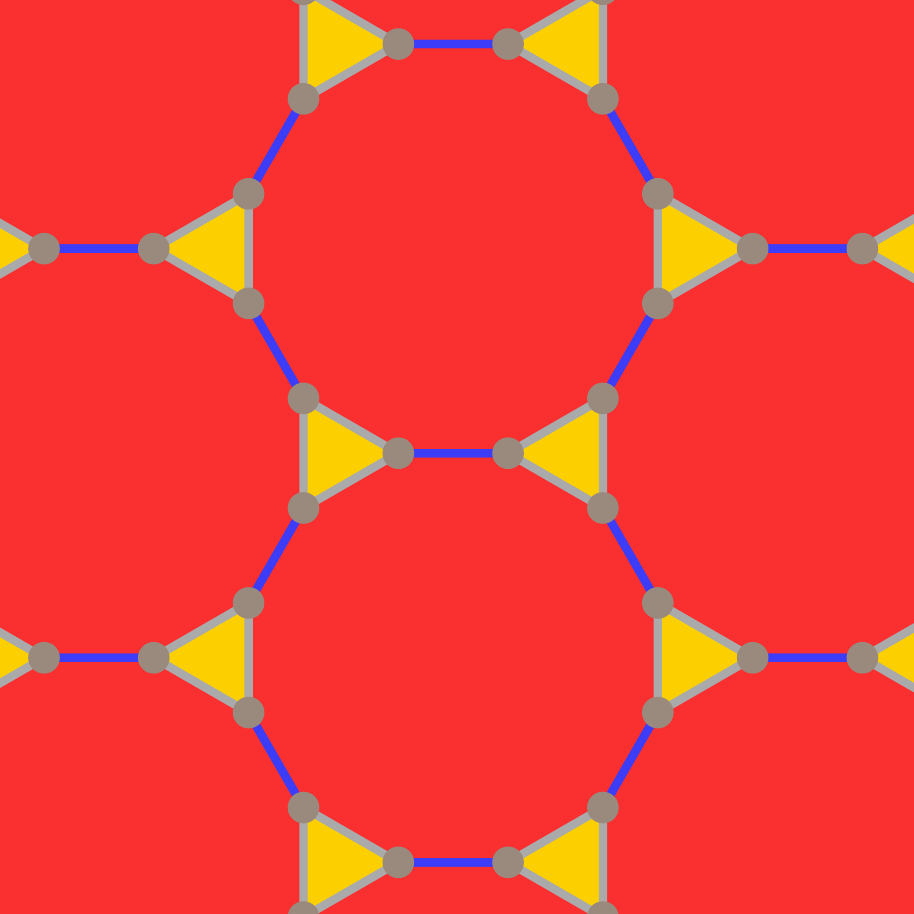}
  \includegraphics[width = 2 cm]{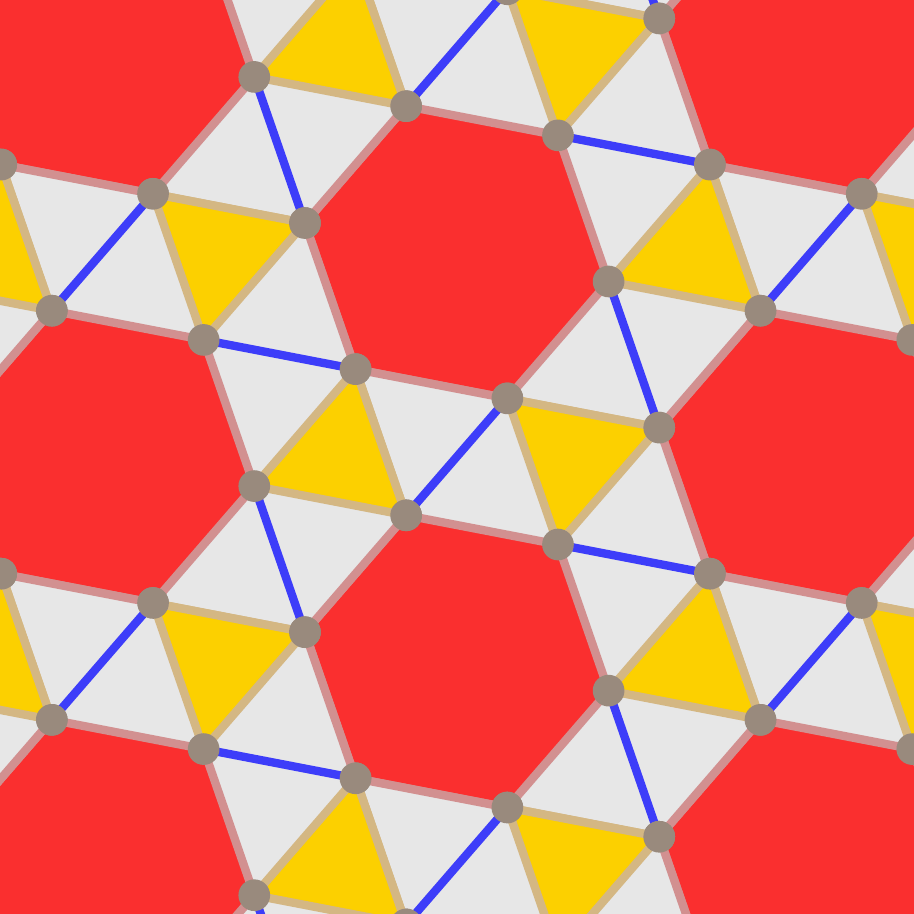}
  \includegraphics[width = 2 cm]{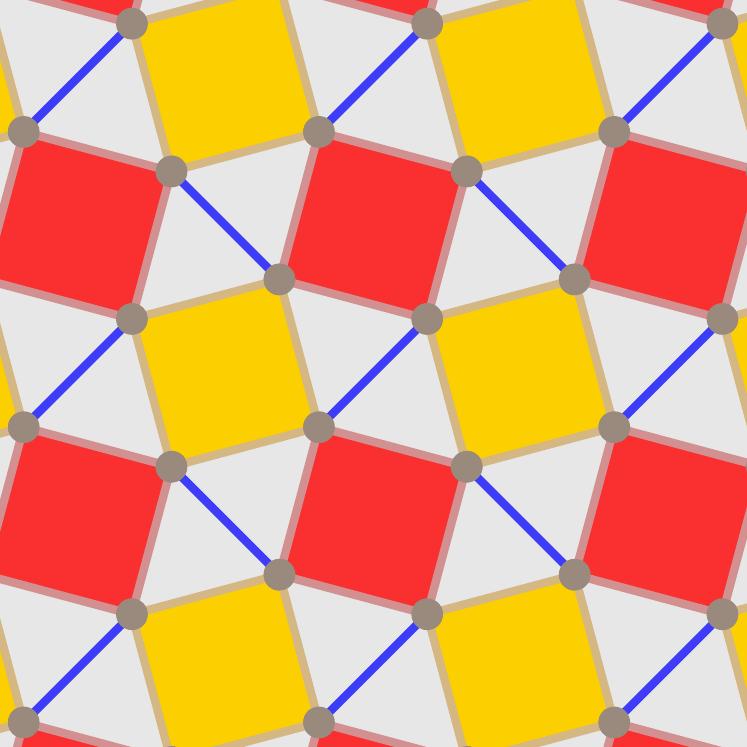}
  \includegraphics[width = 2 cm]{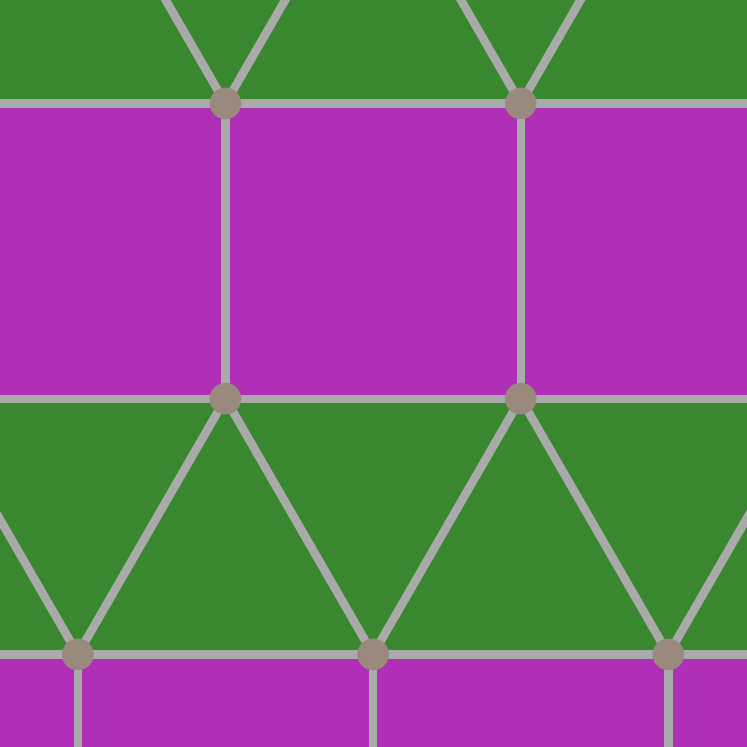}
 \end{center}
 \caption{The 11 semiregular tessellations of the Euclidean plane. The first three are regular.}  \label{T:fig}
\end{figure}

\begin{teo} There are 11 semiregular tessellations of $\matR^2$, shown in Figure \ref{T:fig}.
\end{teo}

The first 8 tessellations in Figure \ref{T:fig} are Wythoffian, the last 3 are not. 
There are infinitely many uniform tessellations of $\matH^2$, both Wythoffian (encoded in the first two lines of Figure \ref{semireg_tess_Hn:fig}) and non Wythoffian, and no nice classification seems known. However, it is possible to enumerate them algorithmically, see Max \cite{Max}. Some examples are shown in Figure \ref{some_tess:fig}. 

\begin{figure} 
 \begin{center}
  \includegraphics[width = 3 cm]{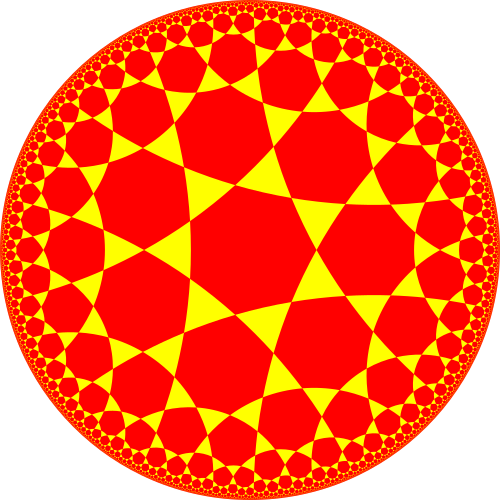}
  \includegraphics[width = 3 cm]{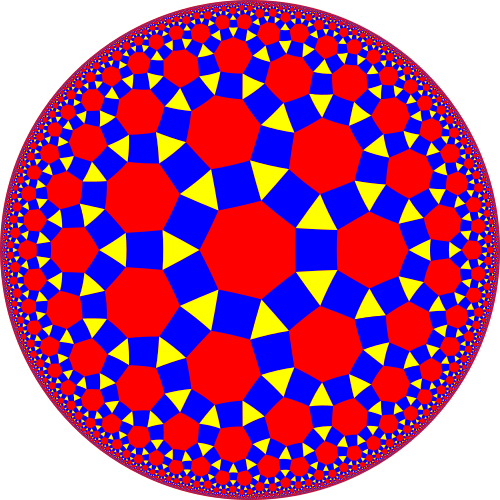}
  \includegraphics[width = 3 cm]{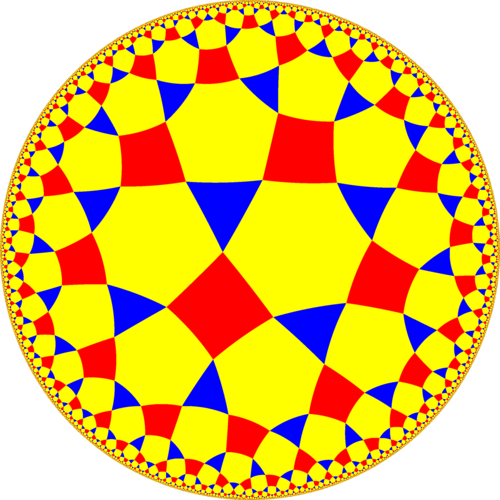}
  \includegraphics[width = 3 cm]{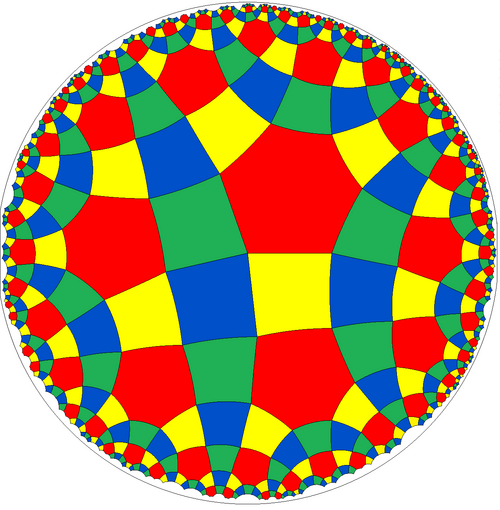}
 \end{center}
 \caption{Some quasi-regular non regular tessellations of $\matH^2$.}  \label{some_tess:fig}
\end{figure}

\begin{figure} 
 \begin{center}
  \includegraphics[width = 6 cm]{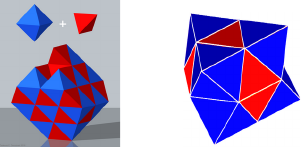}
 \end{center}
 \caption{The two semiregular not regular tessellations of $\matR^3$. They are both made of octahedra (blue) and tetrahedra (red).}  \label{Honeycombs_semiregular:fig}
\end{figure}

\subsubsection{Dimension $\geq 3$ in $\matR^n$}
A complete list of semiregular tessellations in $\matR^n$ for $n\geq 3$ does not seem to be known. 
We know only three non regular examples:
\begin{enumerate}
\item The two tessellations of $\matR^3$ into tetrahedra and octahedra in Figure \ref{Honeycombs_semiregular:fig};
\item The tessellation $5_{21}$ of $\matR^8$ into 8-simplexes and 8-crosspolytopes.
\end{enumerate}

The first tessellation in Figure \ref{Honeycombs_semiregular:fig} is Wythoffian, while the second is not. The first is produced by the diagram with 4 nodes in Figure \ref{semireg_tess_Rn:fig}, or by the circular one in Exercise \ref{diagonal_tess:ex}. The second is obtained from the first by selecting a layer of octahedra and simplexes bounded by two parallel planes, and then reflecting it recursively along the parallel planes.

The mysterious tessellation $5_{21}$ into 8-simplexes and 8-cross-polytopes is the Wythoffian one constructed from the largest diagram in Figure \ref{semireg_tess_Rn:fig}. It was discovered by Gosset \cite{G} and usually indicated with the symbol $5_{21}$ because it is related to the Gosset series $2_{21}, 3_{21}, 4_{21}$. It is the Delaunay tessellation of the $E_8$ lattice $\Lambda < \matR^8$, that is the dual of the Voronoi tessellation. The vertices of the Voronoi tessellation are by definition the \emph{holes} of $\Lambda$, that is the local maxima for the distance function from $\Lambda$. The lattice $\Lambda$ has two kinds of holes: the \emph{deep holes} like $e_1$ that are at distance 1 from $\Lambda$ and the \emph{shallow holes} like $\tfrac 16(5,1,1,1,1,1,1,1)$ that are at distance $2\sqrt 2/3$. Deep and shallow holes have 16 and 9 nearest vertices, that are the vertices of the cross-polytopes and simplexes of the tessellation, centered at the holes. The edges of the tessellation have length $\sqrt 2$, and the sphere-packing dual to the 1-skeleton has been proved by Viazovska to have maxmimum density \cite{Vi}.

\subsubsection{Dimension $\geq 3$ in $\matH^n$}
A complete list of semiregular tessellations in $\matH^n$ does not seem to be known. We get the Wythoffian ones by examining Figure \ref{semireg_tess_Hn:fig}.

\begin{figure} 
 \begin{center}
  \includegraphics[width = 3 cm]{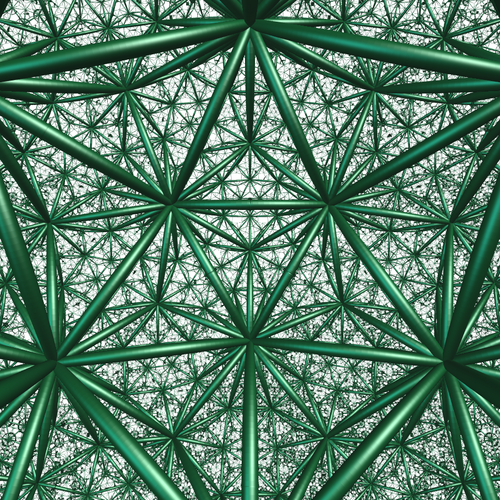} \quad 
  \includegraphics[width = 4 cm]{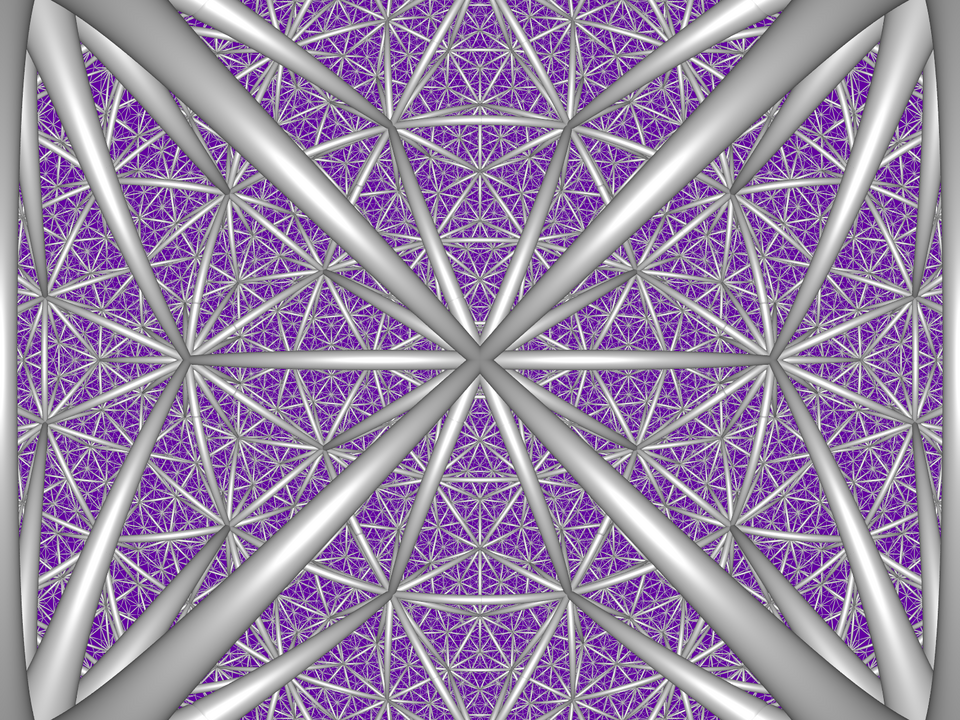} \quad
  \includegraphics[width = 4 cm]{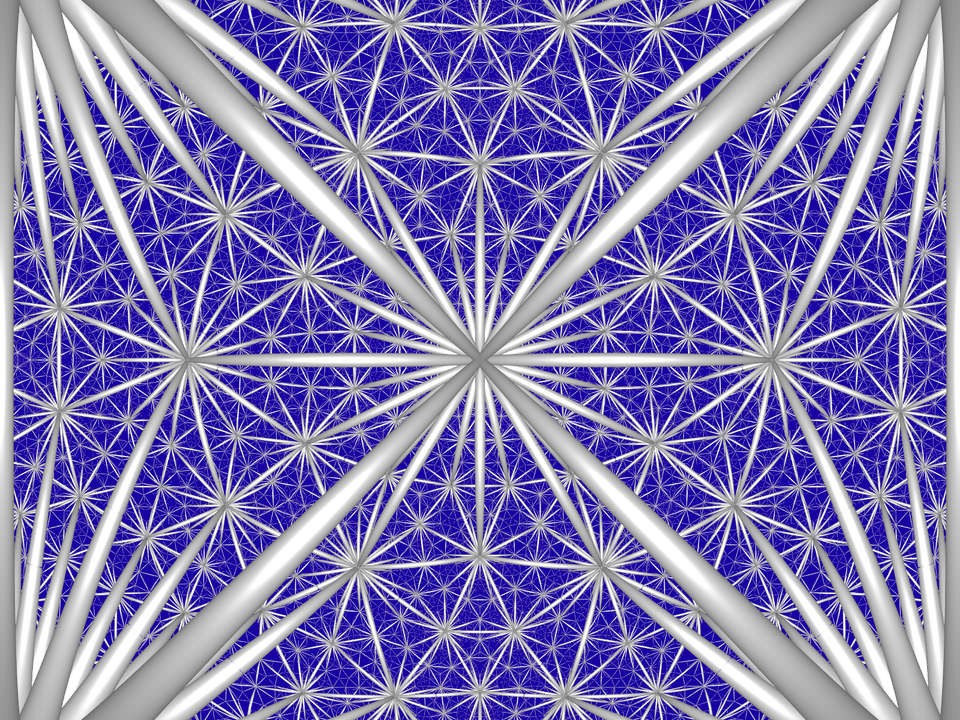} 
 \end{center}
 \caption{The semiregular not regular tessellations of $\matH^3$, 
 (i) by tetrahedra and icosahedra, (ii) by tetrahedra and octahedra, (iii) by tetrahedra, octahedra, and icosahedra.}  \label{Hsr:fig}
\end{figure}

\begin{teo}
The Wythoffian semiregular non regular tessellations of $\matH^n$ in dimension $n\geq 3$ are:
\begin{enumerate}
\item Three tessellations in $\matH^3$ by compact polyhedra:
\begin{itemize} 
\item by tetrahedra and icosahedra, 
\item by tetrahedra and octahedra, 
\item by tetrahedra, octahedra, and icosahedra;
\end{itemize}
\item One tessellation in $\matH^4$ by ideal 4-simplexes and 4-cross-polytopes;
\item One tessellation in $\matH^9$ by ideal 9-simplexes and 9-cross-polytopes.
\end{enumerate}
\end{teo}

The 3-dimensional tessellations are shown in Figure \ref{Hsr:fig}. Note that tetrahedra, octahedra, and icosahedra are the three regular polyhedra that yield only one regular tessellation of $\matH^3$, see Table \ref{nice:regular:table}. The tessellations in $\matH^4$ and $\matH^9$ intersect each horosphere centered at some ideal vertex into the semiregular tessellations in $\matR^3$ and $\matR^8$ with simplexes and cross-polytopes considered above.

\subsection{Uniform polyhedra}
We now turn to uniform polyhedra. The distinction between uniform and semiregular polyhedra is effective only in dimension $n\geq 4$.
The complete list of uniform polyhedra in $\matR^4$ was obtained by Conway -- Guy \cite{CG} in 1965, described with pictures in Conway -- Burgiel -- Goodman-Strauss \cite{CBG},
and finally proved to be complete by M\"oller \cite{Mo} in 2004. 


\begin{teo}
The uniform polyhedra in $\matR^4$ are:
\begin{enumerate}
\item $45$ Wythoffian polyhedra;
\item The snub 24-cell;
\item The grand antiprism;
\item Products of a uniform polyhedron in $\matR^3$ and an interval;
\item Products of two regular polygons.
\end{enumerate}
The types (4) and (5) contain infinitely many elements. 
\end{teo}

All the 45 Wythoffian polyhedra are obtained from linear Coxeter -- Wythoff diagrams. The snub 24-cell and the grand antiprism are not Wythoffian. The \emph{grand antiprism} is the convex hull of the $120-20=100$ points in $I_{120}^*$ minus the vertices of two decagons contained in two orthogonal planes. Its facets are simplexes and pentagonal antiprisms. 
The uniform 4-polyhedron with the largest number 14400  of vertices is obtained from the Coxeter -- Wythoff diagram
 \begin{center}
  \includegraphics[width = 2.5 cm]{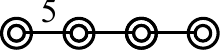}
 \end{center}
It has four type of uniform facets and is shown in Figure \ref{omnitrunc:fig}. The isometry group of the 120-cell (or 600-cell) acts freely and transitively on its vertices.

\begin{figure} 
 \begin{center}
  \includegraphics[width = 8 cm]{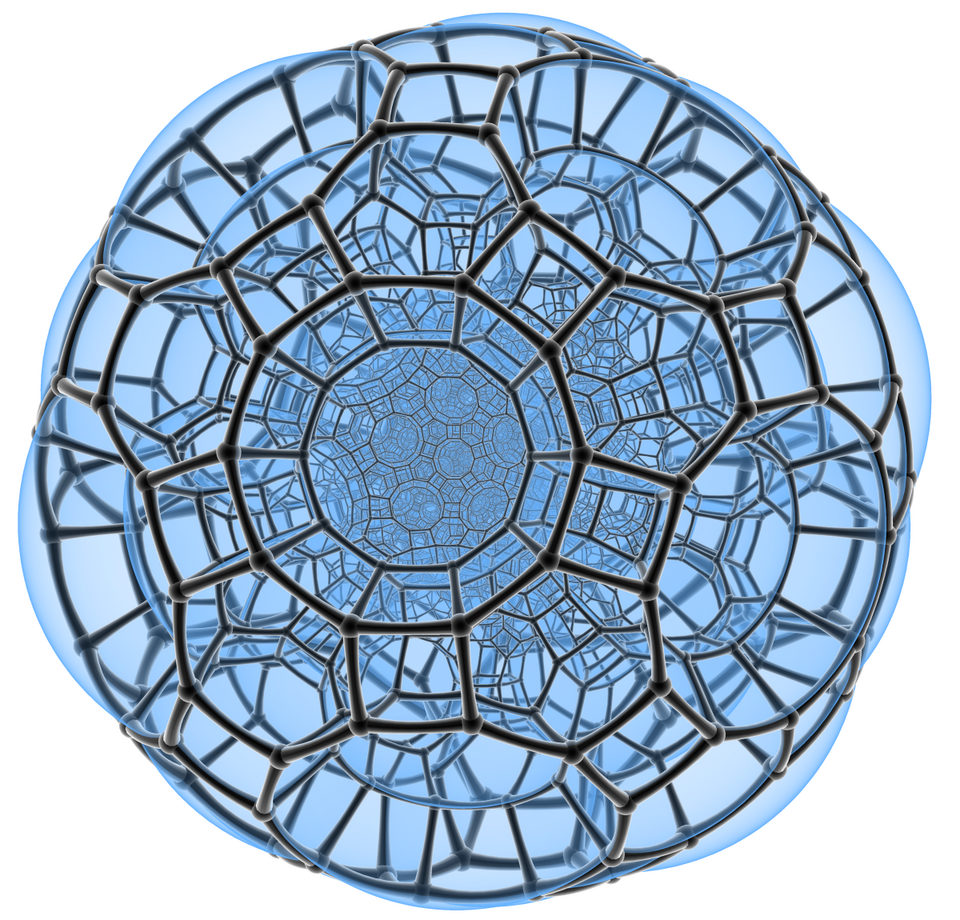}
 \end{center}
 \caption{The uniform 4-polyhedron with the largest number 14400 of vertices.
 The figure shows the stereographic projections in $\matR^3$ of the tessellation of $\matS^3$ induced by the 4-polyhedron.
 }  \label{omnitrunc:fig}
\end{figure}

\subsection{Uniform tessellations}
The distinction between semiregular and uniform tessellations is effective only in dimension $n\geq 3$. There are 28 semiregular tessellations in $\matR^3$ known. The list was completed only in 1994 by Gr\"unbaum \cite{Gr}, who fixed some crucial errors in a pre-existing enumeration. It comprises:
\begin{enumerate}
\item 12 Wythoffian tessellations shown in Figure \ref{Honeycombs:fig};
\item 5 non Wythoffian tessellations shown in Figure \ref{gyrati:fig};
\item 11 tessellations obtained by multiplying those of Figure \ref{T:fig} with an interval.
\end{enumerate}

No proof of the completeness of this list seems to be known.

\begin{figure} 
 \begin{center}
  \includegraphics[width = 12.5 cm]{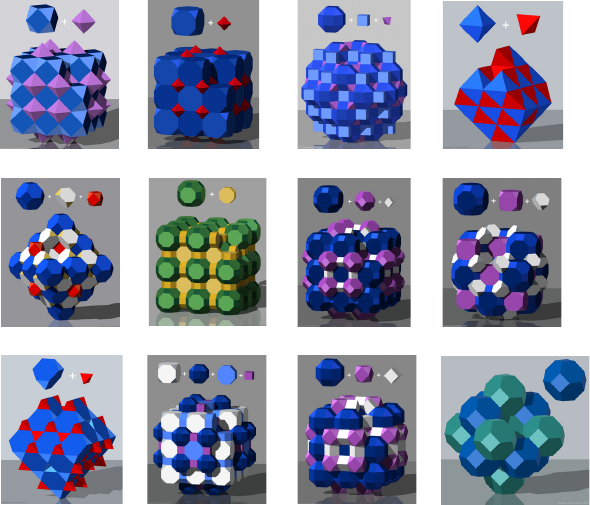}
 \end{center}
 \caption{Twelve Wythoffian uniform Euclidean tessellations.}  \label{Honeycombs:fig}
\end{figure}

\begin{figure} 
 \begin{center}
  \includegraphics[width = 9 cm]{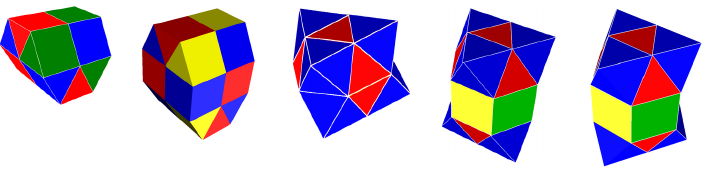}
 \end{center}
 \caption{Five non Wythoffian uniform Euclidean tessellations.}  \label{gyrati:fig}
\end{figure}

\subsection{Dual polyhedra}
Let $P \subset \matR^n$ be a uniform polyhedron. We now define a \emph{dual} polyhedron $P^* \subset \matR^n$ that maintains all the many symmetries of $P$ while inverting the face lattice. The dual polyhedron $P^*$ has many notable properties, and can be realized naturally both in $\matR^n$ and $\matH^n$.

Let $P\subset \matR^n$ be a polyhedron, that we suppose positioned so that its barycenter lies at the origin of $\matR^n$. In all the cases studied here the barycenter can be simply defined as the center of the symmetries of $P$. The \emph{dual polyhedron} is
$$P^* = \{x \in \matR^n\ |\ \langle x,y \rangle \leq 1\, \forall y\in P\} = \{x \in \matR^n \ |\ \langle x, v_i \rangle \leq 1 \}$$
where $v_1,\ldots, v_k$ are the vertices of $P$.

\begin{ex}
The polyhedron $P^*$ is combinatorially dual to $P$, that is there is a natural order-reversing isomorphism of the face lattices of $P$ and $P^*$. The normalized $v_1,\ldots, v_k$ are the normal vectors of the dual facets $F_1,\ldots, F_k$ of $P^*$. We have $P^{**}=P$ up to similarities.
\end{ex}

When $P$ is uniform we may suppose up to a similarity that the vertices $v_1,\ldots, v_k$ have all unitary norm and we also deduce the following.

\begin{prop}
If $P$ is uniform, the dihedral angles of the ridges of $P^*$ are equal.
\end{prop}
\begin{proof}
A ridge $r$ of $P^*$ is dual to an edge $e$ of $P$, and since the vertices of $P$ have the same norm, the dihedral angle of $r$ depends only on the length of $e$. Since $P$ is uniform all its edges have the same length.
\end{proof}

The duals of the semiregular not regular polyhedra in $\matR^3$ of Figure \ref{archimedean:fig} are in Figure \ref{Catalan:fig}, and they consist of the 13 \emph{Catalan polyhedra}, dual to the Archimedean ones, and the infinite families of bipyramids and trapezohedra.

\begin{figure}
 \begin{center}
  \includegraphics[width = 12.5 cm]{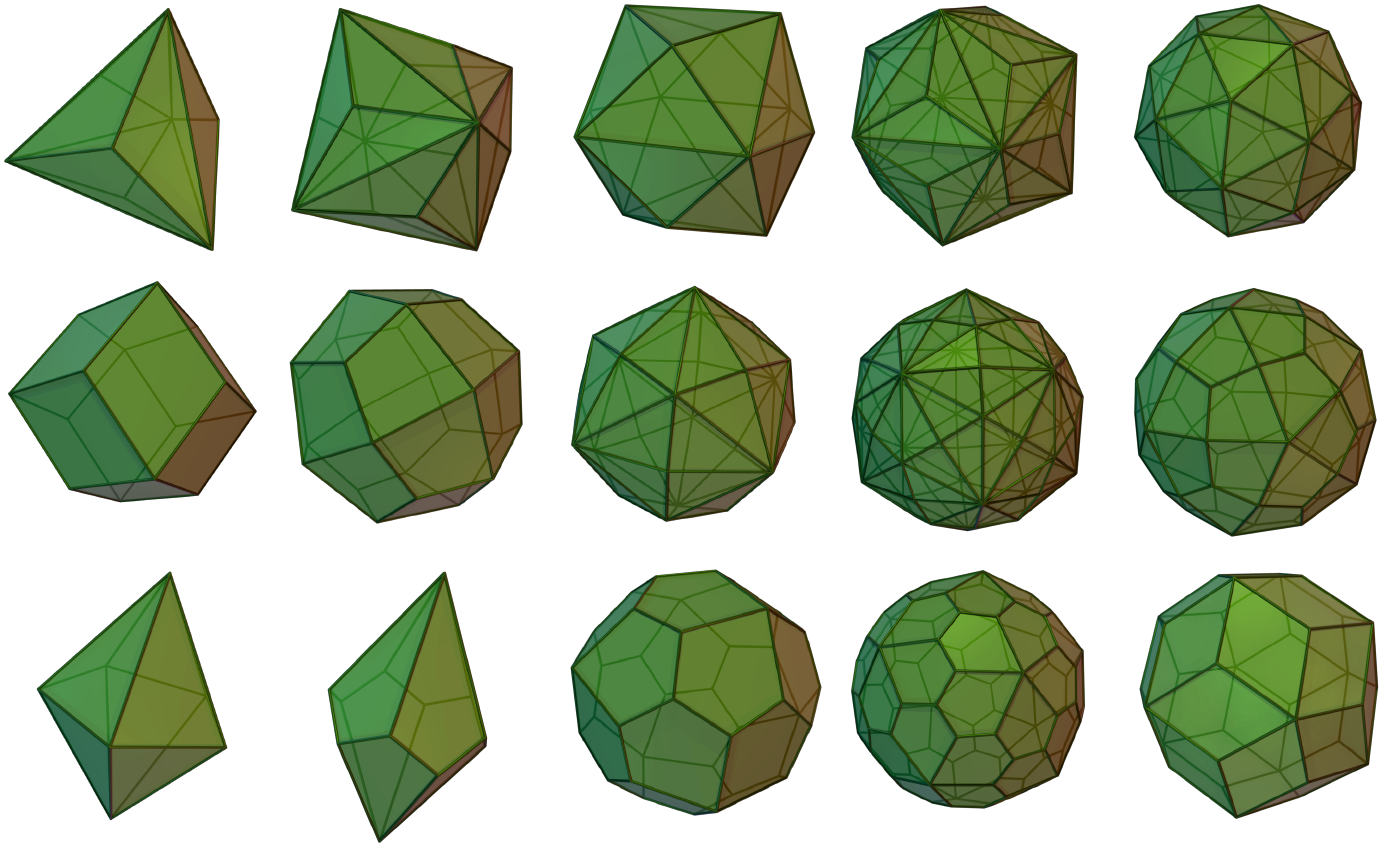}
 \end{center}
 \caption{The dual polyhedra of the semiregular non regular Euclidean polyhedra. They are the 13 \emph{Catalan polyhedra} (the duals of the Archimedean polyhedra, first two rows, plus the last three of the third row), and two infinite families of bipyramids and trapezohedra (the first two of the last row). }  \label{Catalan:fig}
\end{figure}

The polyhedra $P$ and $P^*$ share the same symmetry group $\Gamma$. If $P$ is uniform, $\Gamma$ act transitively on the facets of $P^*$, but not necessarily on its vertices, that are positioned in spheres of different radii (typically corresponding to their $\Gamma$-orbits).

\begin{prop}
If $P$ is semiregular, the links of the vertices of $P^*$ are regular.
\end{prop}
\begin{proof}
By construction these are dual to the facets of $P$, that are regular.
\end{proof}

Summing up, all the polyhedra $P$ in Figure \ref{Catalan:fig} have equal dihedral angles, the links of the vertices are all regular (of different types), and isometries act transitively on the faces.

\subsection{Right-angled hyperbolic polyhedra}
If $P$ is semiregular, the dual polyhedron $P^*$ has a very interesting \emph{hyperbolic realization} in the Klein model where the vertices that lie in the largest sphere are positioned at infinity. 
The polyhedron $P^*$ has typically both ideal and real vertices. This realization is nice because the links of the ideal vertices are all regular. 

\begin{prop}
If $P$ is semiregular and its facets are cross-polytopes and simplexes, the hyperbolic realization of $P^*$ is a right-angled hyperbolic polyhedron.
\end{prop}
\begin{proof}
The link of the vertices of $P^*$ are $n$-cubes and regular simplexes. The dihedral angle of a Euclidean $n$-cube is larger than the one of the Euclidean simplex, so since $P^*$ has constant dihedral angles $\theta$ the only possibility is that the vertices at infinity have $n$-cubes links, and therefore $\theta = \pi/2$.
\end{proof}

We can classify an interesting class of very symmetric right-angled polyhedra.

\begin{teo}
The right-angled hyperbolic $n$-polyhedra whose isometry group acts transitively on their facets, and whose ideal vertex links are $(n-1)$-cubes, are:
\begin{enumerate}
\item The right-angled regular $k$-gons, $k\geq 5$, and regular ideal $h$-gons, $h\geq 3$;
\item The right-angled regular polyhedra of dimension 3 and 4 listed in Table \ref{nice:regular:table};
\item The following 5 polyhedra in dimension 3:
\begin{center}
\includegraphics[width = 2 cm]{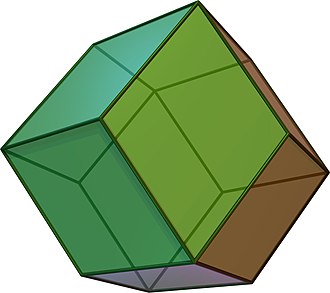}  \
\includegraphics[width = 1.8 cm]{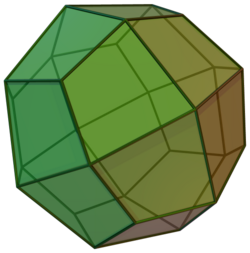} \
\includegraphics[width = 1.8 cm]{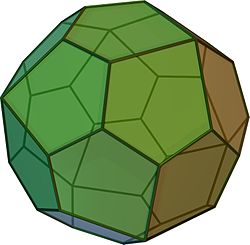} \
\includegraphics[width = 2.5 cm]{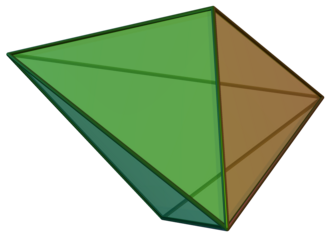} \
\includegraphics[width = 1.7 cm]{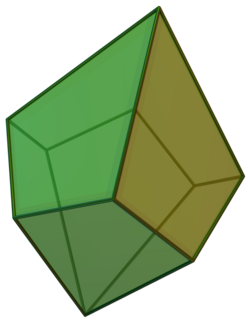}
\end{center}
\item The polyhedra in dimension $n\geq 4$ listed in Table \ref{right:table}.
\end{enumerate}
The polyhedra of type (3) and (4) have both real and ideal vertices. 
\end{teo}
\begin{proof}
The dual of such a polyhedron (considered in the Klein model) is semi-regular with facets that are simplexes and cross-polytopes. We have listed these polyhedra in Section \ref{semiregular:subsection}.
\end{proof}

\begin{table} 
\begin{center}
\begin{tabular}{c||cccc}
\phantom{\Big|} dim & polyhedron & dual & facets & vertices \\
\hline\hline
\phantom{\Big|} $4$ & $P_4$ & rectified 4-simplex & 10 $P_3$ & 5 ideal, 5 real \\
\hline
\phantom{\Big|} $5$ & $P_5$ & 5-demicube & 16 $P_4$ & 10 ideal, 16 real \\
\hline
\phantom{\Big|} $6$ & $P_6$ & $2_{21}$ & 27 $P_5$ & 27 ideal, 72 real \\
\hline
\phantom{\Big|} $7$ & $P_7$ & $3_{21}$ & 56 $P_6$ & 126 ideal, 576 real \\
\hline
\phantom{\Big|} $8$ & $P_8$ & $4_{21}$ & 240 $P_7$ & 2160 ideal, 17280 real 
\end{tabular}
\end{center}
\vspace{.3 cm}
\caption{Right-angled hyperbolic polyhedra of dimension $n\geq 4$. Here $P_3$ is a triangular bipyramid.}
\label{right:table}
\end{table}

The right-angled polyhedra $P_4,\ldots, P_8$ have been discovered by various authors, and being right-angled they are well-suited to build interesting hyperbolic manifolds, see Agol -- Long -- Reid \cite{ALR}, 
Potyagailo -- Vinberg \cite{PV}, Ratcliffe -- Tschantz \cite{RT, RT4, RT5}, 
Everitt -- Ratcliffe -- Tschantz \cite{ERT}, Italiano -- Martelli -- Migliorini \cite{IMM}.


\begin{thebibliography}{99}

\bibitem{ALR} \textsc{I. Agol -- D. Long -- A. Reid}, 
\emph{The Bianchi groups are separable on geometrically finite subgroups}, Ann. of Math., \textbf{153} (2001), 599--621.

\bibitem{A} \textsc{E. Andreev}, \emph{On convex polyhedra in Lobacevskii spaces (English Translation)}, Math. USSR Sbornik, \textbf{10} (1970), 413--440.

\bibitem{A2} \bysame, \emph{On convex polyhedra of finite volume in Lobachevskii space},
Mathematics of the USSR-Sbornik, \textbf{12} (1970), 255--259.

\bibitem{A3} \bysame, \emph{Intersection of plane boundaries of a polyhedron with acute angles}, Mat. Zametki \textbf{8} (1970), 521--527.

\bibitem{BB} \textsc{G. Blind -- R. Blind}, \emph{The semiregular polyhedrons}, Comment. Math. Helv. \textbf{66} (1991), 150--154.


\bibitem{Ch} \textsc{M.~Chein}, \emph{Recherche des sch\'emas de complexes de Coxeter hyperboliques}, C. R. Acad. Sci. Paris Ser. A-B, 
\textbf{268} (1969), 439--442.

\bibitem{CBG} \textsc{J.~Conway -- H.~Burgiel -- C.~Goodman-Strauss},
\emph{The Symmetries of Things}, A K Peters/CRC Press, Wellesley, MA, 2008.

\bibitem{CG} \textsc{J. Conway -- M. Guy}, \emph{Four-Dimensional Archimedean Polytopes}, Proceedings of the Colloquium on Convexity, Copenhagen 1965, 38--39. 

\bibitem{C} \textsc{H. Coxeter}, \emph{Discrete groups generated by reflections}, Annals Math., \textbf{35}(1934), 588--621.

\bibitem{C2} \bysame, ``Regular polytopes'', Dover Publications, Inc., New York
1973.


\bibitem{ERT} \textsc{B.~Everitt -- J.~Ratcliffe -- S.~Tschantz}, \emph{Right-angled Coxeter polytopes, hyperbolic six-manifolds, and a problem of Siegel}, Math. Ann. \textbf{354} (2012), 871--905.

\bibitem{G} \textsc{T. Gosset}, \emph{On the regular and semiregular figures in space of $n$ dimensions}, Messenger Math. \textbf{29} (1899), 43--48.

\bibitem{Gr} \textsc{B.~Gr\"unbaum}, \emph{Uniform tilings of 3-space}, Geombinatorics \textbf{4} (1994), 49 -- 56.

\bibitem{Gru} \textsc{B.~Gr\"unbaum -- G.~C.~Shephard},
``Tilings and Patterns'', W.~H.~Freeman, New York, 1987.

\bibitem{HW} \textsc{W.~Hantzsche -- H.~Wendt}, \emph{Dreidimensionale euklidische raumformen}, Mathematische Annalen \textbf{110} (1935), 593--611.

\bibitem{K} \textsc{J. Koszul}, \emph{Unit\'e des groupes simples de Lie}, S\'eminaire Bourbaki \textbf{55} (1950), 101--108.

\bibitem{IMM} \textsc{G.~Italiano -- B.~Martelli -- M.~Migliorini}, \emph{Hyperbolic manifolds that fiber algebraically up to dimension 8}, J. Inst. Math. Jussieu \textbf{23} (2024), 609--646.

\bibitem{L} \textsc{F. Lann\'er}, \emph{On complexes with transitive groups of automorphisms}, Comm. Sem. Math. Univ. Lund \textbf{11} (1950), 1--71.

\bibitem{M} \textsc{B.~Martelli}, ``An Introduction to Geometric Topology,'' Independently published, 2016.

\bibitem{Max} \textsc{N. Max}, \emph{Constructing and Visualizing Uniform Tilings}, Computers 2023, 12(10), 208.

\bibitem{MH} \textsc{J.~Milnor -- D.~Husemoller}, 
``Symmetric Bilinear Forms'', Ergebnisse der Mathematik und ihrer Grenzgebiete, Vol. 73, 
Springer-Verlag, Berlin-Heidelberg, 1973.

\bibitem{Mo} \textsc{M. M\"oller}, ``Vierdimensionale Archimedische Polytope'' (2004), Doctoral thesis.

\bibitem{PV} \textsc{L.~Potyagailo -- E-~V.~Vinberg}, \emph{On right-angled reflection groups in hyperbolic spaces}, Comment. Math. Helv. \textbf{80} (2005), 63--73.

\bibitem{RT} \textsc{J.~Ratcliffe -- S.~Tschantz}, \emph{Volumes of integral congruence hyperbolic manifolds}, J. Reine Angew. Math. \textbf{488} (1997), 55--78.

\bibitem{RT4} \bysame, \emph{The volume spectrum of hyperbolic 4-manifolds}, Experiment. Math. \textbf{9} (2000), 101--125.

\bibitem{RT5} \bysame, \emph{Integral congruence two hyperbolic 5-manifolds}, Geom. Dedicata \textbf{107} (2004), 187--209.

\bibitem{R} \textsc{R.~Roeder}, \emph{Compact hyperbolic tetrahedra with non-obtuse dihedral angles}, Publ. Mat., Barc. \textbf{50} (2006), 211--227.

\bibitem{RHD} \textsc{R.~Roeder -- J.~Hubbard -- W. Dunbar}, \emph{Andreev's Theorem on hyperbolic polyhedra}, Annales de l'Institut Fourier \textbf{57} (2007), 825--882.

\bibitem{S} \textsc{L. Schl\"afli}, ``Theorie der vielfachen Kontinuit\"at'', Hrsg. im Auftrage der
Denkschriften-Kommission der schweizerischen naturforschenden Gesellschaft
von J. H. Graf, Georg \& Co., Z\"urich--Basel 1901.

\bibitem{Vi} \textsc{M. Viazovska}, \emph{The sphere packing problem in dimension 8}, Annals Math. \textbf{185} (2017), 991--1015.

\bibitem{V} \textsc{E. B. Vinberg}, \emph{Hyperbolic reflection groups}, Russian Math. Surveys \textbf{40} (1985), 31--75.

\bibitem{W} \textsc{T.~Walsh}, \emph{Characterization of the semiregular polyhedra}, Geometriae Dedicata, 
\textbf{1} (1972), 117--123.

\end{thebibliography}
\end{document}